\documentclass[10pt]{amsart}

\usepackage[english]{babel}
\usepackage[T1]{fontenc}
\usepackage[utf8]{inputenc}

\usepackage[a4paper,top=3cm,bottom=3cm,left=2.7cm,right=2.7cm,marginparwidth=2cm]{geometry}
\usepackage{indentfirst}
\usepackage[dvipsnames]{xcolor}
\usepackage[all,cmtip]{xy}

\usepackage{amsmath}
\usepackage{amsthm}
\usepackage{thmtools}

\usepackage[color=pink]{todonotes}

\usepackage[colorlinks=true, linkcolor=blue, citecolor=blue]{hyperref}
\usepackage{cleveref}

\crefname{theorem}{Theorem}{Theorems}
\crefname{lemma}{Lemma}{Lemmas}
\crefname{corollary}{Corollary}{Corollaries}
\crefname{proposition}{Proposition}{Propositions}
\crefname{conjecture}{Conjecture}{Conjectures}
\crefname{question}{Question}{Questions}
\crefname{definition}{Definition}{Definitions}
\crefname{remark}{Remark}{Remarks}
\crefname{example}{Example}{Examples}
\crefname{case}{Case}{Cases}

\usepackage{tikz}
\usepackage{tikz-cd}
\usepackage{amsfonts}
\usepackage{amssymb}
\usepackage{graphicx}  
\usepackage{mathrsfs}
\usepackage{mathtools}

\usepackage{extarrows}
\usepackage{tensor}

\theoremstyle{plain}
\newtheorem{theorem}{Theorem}[section]
\newtheorem{lemma}[theorem]{Lemma}
\newtheorem{corollary}[theorem]{Corollary}
\newtheorem{conjecture}[theorem]{Conjecture}

\newtheorem{proposition}[theorem]{Proposition}

\theoremstyle{definition}
\newtheorem{definition}[theorem]{Definition}

\newtheorem{remark}[theorem]{Remark}
\newtheorem{example}[theorem]{Example}
\numberwithin{equation}{section}

\theoremstyle{remark}

\DeclareMathOperator{\rk}{\operatorname{\mathsf{rk}}}
\DeclareMathOperator{\Hom}{Hom}
\DeclareMathOperator{\coker}{coker}

\DeclareMathOperator{\modu}{mod}

\usetikzlibrary{matrix,arrows,decorations.markings}

\usepackage{pdflscape}
\usepackage{afterpage}
\usepackage{caption} 
\usepackage{capt-of}
\usepackage{tikz}

\newcommand{\suchthat}{\ | \ }

\newcommand{\myid}{1\hspace{-0.125cm}1}

\newcommand{\image}{\operatorname{im}}
\newcommand{\marked}{\mathbb{M}}

\newcommand{\orbset}{\mathbb{O}}
\newcommand{\surf}{(\Sigma,\marked,\orbset)}

\newcommand{\decrep}{\operatorname{decRep}}
\newcommand{\mycurve}{\mathfrak{c}}

\newcommand{\arrow}{E}

\newcommand{\compactbigotimesk}[1][0pt]{%
  \mathrel{\raisebox{#1}{$\underset{\scalebox{0.9}{
$\Bbbk\,$}}{\boldsymbol{\otimes}}$}}%
}
\newcommand{\smallotimesk}{\scalebox{0.6}{$\compactbigotimesk[0.45em]$}}

\newcommand{\kink}{\iota}

\newcommand{\stau}[1]{\operatorname{s\tau-tilt}{#1}}

\title[Gentle algebras from surfaces with orbifold points II]{Gentle algebras arising from surfaces with orbifold points, Part II: Locally free Caldero--Chapoton functions}
\author{Daniel Labardini-Fragoso}
\address{Daniel Labardini-Fragoso\newline
Instituto de Matem\'aticas, UNAM, Mexico}
\email{labardini@im.unam.mx}

\author{Lang Mou}
\address{Lang Mou\newline
Mathematisches Institut, Universit\"at zu K\"oln, Weyertal 86-90, 50931 K\"oln, Germany}
\email{langmou@math.uni-koeln.de}
\date{}

\begin{document}

\begin{abstract}
We prove that in the skew-symmetrizable cluster algebras associated by Felikson--Shapiro--Tumarkin to unpunctured surfaces with orbifold points of order $2$ and a specific choice of weights, the Laurent expansion of any cluster variable with respect to any cluster coincides with the locally free Caldero--Chapoton function of a $\tau$-rigid representation of a gentle algebra. These cluster algebras are typically non-acyclic and of infinite type, whereas for polygons with one orbifold point one recovers cluster algebras of finite type $C$; so, our result is an ample extension of a seminal result established by Geiss--Leclerc--Schröer for skew-symmetrizable cluster algebras of finite type and acyclic initial seeds. 
As the main means to achieve the result, we provide a generalization of Derksen--Weyman--Zelevinsky's mutation theory of loop-free quivers with potential to the quivers-with-loops with potential we associate to the triangulations of unpunctured surfaces with orbifold points, and study the relation with $\tau$-tilting theory.

As a result of independent interest, we compute the aforementioned $\tau$-rigid representations explicitly. To this end, we show that the indecomposable $\tau$-rigid string modules arising from arcs on the surface, and the quasi-simple band modules arising from simple closed curves, are well-behaved under the mutations of representations we define in the paper, thus extending results of the first author's Ph.D. thesis.
\end{abstract}

\maketitle

\tableofcontents

\section{Introduction}

Cluster algebras associated with marked bordered surfaces have been a rich area of study since Fock and Goncharov's discovery of cluster structure on the decorated Teichm\"uller spaces \cite{fock2006moduli} and the systematic study by Fomin, Shapiro and Thurston \cite{fomin2008cluster} as a special class of cluster algebras of Fomin and Zelevinsky \cite{fomin2002cluster, fomin2003cluster}. Among many significant features they possess, these algebraic structures capture the combinatorial intricacies of surface triangulations. In particular, these cluster algebras are of finite mutation type, a fact reflecting that there are only finitely many isomorphic triangulations of a given surface.

Felikson, Shapiro, and Tumarkin \cite{felikson2012cluster} expanded the class of cluster algebras associated to marked bordered surfaces to incorporate orbifold points. The resultant cluster structures are not necessarily skew-symmetric, meaning that the cluster exchange relations are no longer governed by skew-symmetric matrices (or equivalently quivers without 2-cycles) but by more general skew-symmetrizable ones induced from triangulations of orbifolds. These cluster algebras form a major part of the classification of cluster algebras of finite mutation type \cite{felikson2012unfolding}, i.e., the ones with only finitely many distinguished exchange matrices.

For skew-symmetric cluster algebras, a powerful framework of (additive) categorification modeled on quiver representations has been established since the seminal works \cite{buan2006tilting} and \cite{caldero2006cluster}. This methodology was later significantly fortified by Derksen, Weyman and Zelevinsky in \cite{derksen2008quivers, derksen2010quivers}, where for any skew-symmetric matrix, there is a distinguished set of representations of the adjacency quiver whose Caldero--Chapoton functions precisely express (non-initial) cluster variables. In the context of skew-symmetric cluster algebras, this noteworthy correlation between quiver representation theory and cluster theory serves as a keystone in solving a collection of conjectures proposed in \cite{fomin2007cluster} which are particularly challenging if using pure combinatorial methods; see \cite[Theorem 1.7]{derksen2010quivers}.

Let us revisit the case of cluster algebras associated to unpunctured marked bordered (oriented) surfaces and explain how quiver representations are useful to understand them. Given any such surface $\Sigma$, the cluster structure inherent to the associated cluster algebra $\mathcal A(\Sigma)$ can be explicitly described combinatorially via the \emph{arc complex} $\mathbf{Arc}(\Sigma)$. In this simplicial complex, $k$-simplices are collections of $k$ arcs in $\Sigma$ considered up to homotopy that only intersect at their endpoints on the boundary marks. Meanwhile, the cluster complex $\Delta(\mathcal A(\Sigma))$ is a simplicial complex on the ground set of all cluster variables whose maximal simplices are clusters.

Fomin, Shapiro and Thurston \cite{fomin2008cluster} showed that the two simplicial complexes $\mathbf{Arc}(\Sigma)$ and $\Delta(\mathcal A(\Sigma))$ are actually isomorphic. In particular, the isomorphism induces bijections
\[
    \{\text{Arcs in $\Sigma$}\} \longleftrightarrow \{ \text{Cluster variables in $\mathcal A(\Sigma)$} \}  \quad \text{and} \quad \{\text{Triangulations of $\Sigma$}\} \longleftrightarrow \{\text{Clusters in $\mathcal A(\Sigma)$}\}
\]
such that flips of triangulations correspond to mutations of clusters.

Each triangulation $T$ of $\Sigma$ gives a quiver $Q(T)$ whose vertices are arcs in $T$ and an arrow is introduced if one arc follows another clockwise within a triangle. One can impose gentle relations $I(T)$ on the associated quiver $Q(T)$ as per \cite{assem2010gentle} and consider the subsequent finite-dimensional gentle algebra $\Bbbk Q(T)/I(T)$. Alternatively there is a potential $W(T)$ defined in \cite{labardini2009quivers} such that the Jacobian algebra $\mathcal P(T)$ of the quiver with potential $(Q(T), W(T))$ is isomorphic to the gentle algebra.

To state the following correspondence between the cluster structure of $\mathcal A(\Sigma)$ and the representation theory of $\mathcal P(T)$, for convenience we choose to employ the language of $\tau$-tilting theory of Adachi--Iyama--Reiten \cite{adachi2014tau} over that of cluster categories or \cite{derksen2010quivers}, which historically may precede. We consider the simplicial complex $\Delta(\mathcal P(T))$ of (basic) $\tau$-rigid pairs whose maximal simplices are (basic) support $\tau$-tilting pairs \cite{demonet2019tau}.

The two major points of the categorification can be summarized as follows:

\begin{itemize}
    \item[(C1)]\label{c1} The complex $\Delta(\mathcal P(T))$ is isomorphic to the cluster complex $\Delta(\mathcal A(\Sigma))$ (thus also to the arc complex $\mathbf{Arc}(\Sigma)$). This isomorphism induces a bijection on $0$-simplices
    \[
        \{ \text{Cluster variables in $\mathcal A(\Sigma)$} \} \longleftrightarrow \{ \text{Indecomposable $\tau$-rigid pairs of $\mathcal P(T)$}\}/\!\cong
    \] such that their respective parameterizations by $\mathbf g$-vectors coincide. It also induces another bijection on maximal simplices
    \[
        \{\text{Clusters in $\mathcal A(\Sigma)$}\} \longleftrightarrow \{\text{Basic support $\tau$-tilting pairs of $\mathcal P(T)$}\}/\!\cong
    \]
    where the cluster given by $T$ corresponds to the pair $(0, \mathcal P(T))$, manifesting respective mutations and thus giving an isomorphism between respective mutation graphs;
    \item[(C2)]\label{c2} The (representation-theoretic) $F$-polynomial of an indecomposable $\tau$-rigid module equals the $F$-polynomial of the corresponding cluster variable. Together with the coincidence of $\mathbf g$-vectors in (C1), this ensures the existence of a valid cluster character - the Caldero--Chapoton function of an indecomposable $\tau$-rigid module equals the Laurent polynomial expression of the cluster variable in the initial cluster variables associated with arcs in $T$.
\end{itemize}

The above conclusions (C1) and (C2) can be derived from \cite{derksen2010quivers} and \cite{labardini2009quivers} combined. The interpretation of $\mathcal A(\Sigma)$ in terms of $\mathcal P(T)$-modules implies properties of $\mathbf g$-vectors and $F$-polynomials on the cluster algebra side; see for example \cite[Conjecture 5.4 and 6.13]{fomin2007cluster}.

A square integer matrix $B$ is called skew-symmetrizable if there exists a diagonal matrix $D$ with positive integer entries such that $DB$ is skew-symmetric, i.e. $DB + B^\intercal D^\intercal = 0$. A cluster algebra is said to be skew-symmetrizable if its exchange matrices are skew-symmetrizable. In this case, a uniform representation-theoretic approach to the cluster algebra $\mathcal A(B)$ in a framework analogous to (C1) and (C2) remains elusive. The complexity mostly arises from that skew-symmetrizable matrices no longer correspond to quivers, but rather to valued quivers whose representation theories deviate significantly from that of quivers.

Surfaces with orbifold points (of order 2), referred to as orbifolds $\mathcal O$ in short, provide a nice combinatorial model for a class of skew-symmetrizable cluster algebras $\mathcal A(\mathcal O)$ \cite{felikson2012cluster}. In this case, cluster variables are still represented by arcs and clusters by triangulations. Due to the existence of orbifold points, the cluster exchange matrices are no longer skew-symmetric but skew-symmetrizable.

In this paper we provide a representation-theoretic approach to the skew-symmetrizable cluster algebras associated to unpunctured marked bordered orbifolds; see the precise definition in \Cref{section: orbifolds}. This is a class of cluster algebras considered in \cite{felikson2012cluster} with a specific choice of \emph{weights} on orbifold points; see \Cref{rmk: matrix}. To each triangulation $T$ of an orbifold $\mathcal O$, we construct a gentle algebra $\mathcal P(T)$ from a quiver with relations as in \cite{LM1}. Now the quivers may have loops arising from the orbifold points. We achieve to relate the cluster structure of $\mathcal A(\mathcal O)$ with the $\tau$-tilting theory of $\mathcal P(T)$ by proving the property (C1) in \Cref{thm: main theorem} (A) and \Cref{cor: app in cluster algebra} (2).

The more challenging task is perhaps (C2), as the usual definition of the $F$-polynomial of a module has to be modified. In our skew-symmetrizable situation, the $\tau$-rigid $\mathcal P(T)$-modules actually possess a \emph{locally free} structure, meaning that for a module $M$, the subspace $e_iM$ is a free module over the local algebra $e_i\mathcal P(T)e_i$ for every primitive idempotent $e_i$. Thus one can take the \emph{locally free} $F$-polynomial of a locally free module, a generating series `counting' only submodules that are locally free. With this modification, we have obtained a valid Caldero--Chapoton type formula for cluster variables; see \Cref{thm: main theorem} (B) and \Cref{cor: app in cluster algebra} (3).

The key technique we employ is a suitable generalization of the Derksen--Weyman--Zelevinsky (DWZ) mutation \cite{derksen2008quivers} designed for $\mathcal P(T)$-modules in \Cref{section: mutation}. We show that iterative mutations yield a method generating $\tau$-rigid modules. In addition, we prove the desired recurrence of representation-theoretic $\mathbf g$-vectors and locally free $F$-polynomials under mutations, thus relating the representation theory of $\mathcal P(T)$ with the cluster structure of $\mathcal A(\mathcal O)$.

Lastly, we offer explicit descriptions of indecomposable $\tau$-rigid $\mathcal P(T)$-modules in \Cref{section: arc representation}. They are given by certain strings associated with arcs on the orbifold. Such a characterization of $\tau$-rigid modules should be consistent with more general descriptions, such as \cite{PPP} for gentle algebras.

We now provide a more comprehensive overview of the paper with more detailed elaborations regarding our main results.

\subsection{Cluster algebras associated to orbifolds}

In this paper we study a class of skew-symmetrizable cluster algebras associated to orbifolds considered in \cite{felikson2012cluster}. For us, an \emph{orbifold} $\mathcal O = \surf$ is an underlying oriented compact surface $\Sigma$ with non-empty boundary $\partial \Sigma$, a finite set of boundary marked points $\mathbb M\subset \partial \Sigma$, and a finite set of orbifold points $\mathbb O \subset \Sigma\setminus \partial \Sigma$ of order 2. An \emph{arc} $\gamma$ in $\mathcal O$ is a curve in $\Sigma \setminus \mathbb O$ with only possible self-intersections at its endpoints in $\mathbb M$. A \emph{triangulation} of $\mathcal O$ is a maximal collection of compatible arcs. We refer to \Cref{subsection: orbifold and triangulations} for precise definitions.

The cluster algebra $\mathcal A(\mathcal O)$ is a commutative ring generated by (cluster) variables associated to arcs in $\mathcal O$ subject to cluster exchange relations arising from flips of triangulations. The first kind of relations appear when flipping the diagonal in a quadrilateral as in \Cref{fig: ptolemy relation quadrilateral}. The second kind is illustrated in \Cref{fig: exchange relation pending} where the flip exchanges \emph{pending arcs} $\gamma$ and $\gamma'$ encircling the orbifold point $\star$. We say that the cluster variables belonging to one triangulation form a cluster. The change of cluster by flipping a diagonal in \Cref{fig: ptolemy relation quadrilateral} or a pending arc in \Cref{fig: exchange relation pending} is called a cluster mutation.

\begin{figure}[h]
    \begin{tikzpicture}
    \filldraw[fill=gray!20, thick] (0,0) rectangle (2,2);
    \draw[thick] (0,2) -- (2,0);
    \node at (1.2,1.1) {$\gamma$};

    \filldraw[fill = gray!20, thick] (5,0) rectangle (7,2);
    \draw[thick] (5,0) -- (7,2);
    \node at (6.2,0.9) {$\gamma'$};

    \node[below] at (1,0) {$\gamma_1$}; 
    \node[left] at (0,1) {$\gamma_2$}; 
    \node[above] at (1,2) {$\gamma_3$}; 
    \node[right] at (2,1) {$\gamma_4$}; 

    \node[below] at (6,0) {$\gamma_1$}; 
    \node[left] at (5,1) {$\gamma_2$}; 
    \node[above] at (6,2) {$\gamma_3$}; 
    \node[right] at (7,1) {$\gamma_4$}; 

    \draw[->] (3,1.1) -- (4,1.1) node[midway, above] {flip at $\gamma$};
    \draw[->] (4,0.9) -- (3,0.9) node[midway, below] {flip at $\gamma'$};
    \end{tikzpicture}
    \centering
    \caption{Exchange relation: $X_{\gamma}X_{\gamma'} = X_{\gamma_1}X_{\gamma_3} + X_{\gamma_2}X_{\gamma_4}$.}\label{fig: ptolemy relation quadrilateral}
\end{figure}
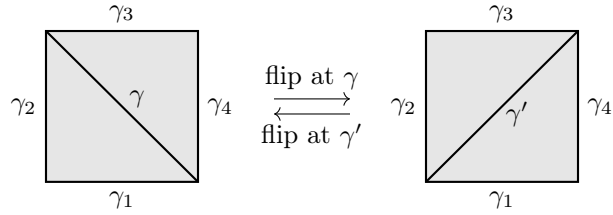

\begin{figure}[h]
    \begin{tikzpicture}
        \filldraw[fill=gray!20, thick](0,0) circle (1);
        \filldraw[black] (0, 1) circle (1.5pt);
        \filldraw[black] (0, -1) circle (1.5pt);
        \draw[] (0, 0) node[]{$\star$};
        \draw[thick] plot [smooth cycle] coordinates {(0,-1) (-0.2, 0) (0, 0.2) (0.2, 0)};
        \draw[] (0, 0.4) node[]{\small $\gamma$};
        \draw[] (-0.7, 0) node[]{\small $\gamma_1$};
        \draw[] (0.7, 0) node[]{\small $\gamma_2$};
        \draw[->] (2, 0) -- (3, 0) node[midway, above]{flip at $\gamma$};
        \draw[->] (3, -0.2) -- (2, -0.2) node[midway, below]{flip at $\gamma'$};
    \end{tikzpicture}
    \quad
    \quad
    \begin{tikzpicture}
        \filldraw[fill=gray!20, thick](0,0) circle (1);
        \filldraw[black] (0, 1) circle (1.5pt);
        \filldraw[black] (0, -1) circle (1.5pt);
        \draw[] (0, 0) node[]{$\star$};
        \draw[thick] plot [smooth cycle] coordinates {(0, 1) (-0.2, 0) (0, -0.2) (0.2, 0)};
        \draw[] (0, -0.4) node[]{\small $\gamma'$};
        \draw[] (-0.7, 0) node[]{\small $\gamma_1$};
        \draw[] (0.7, 0) node[]{\small $\gamma_2$};
    \end{tikzpicture}
    \centering
    \caption{Exchange relation: $X_\gamma X_{\gamma'} = X_{\gamma_1}^2 + X_{\gamma_2}^2$} \label{fig: exchange relation pending}
\end{figure}
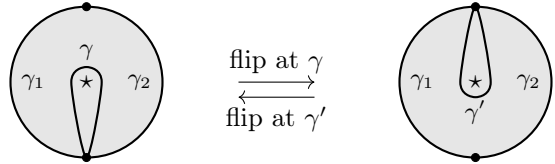

By the remarkable Laurent phenomenon \cite{fomin2002cluster}, cluster exchange relations allow the cluster variable $X_\gamma$ of any arc $\gamma$ to be expressed as
\[
    X_\gamma = X_\gamma(X_{\gamma_1}, \dots, X_{\gamma_n}) \in \mathbb Z[X_{\gamma_1}^\pm, \dots, X_{\gamma_n}^\pm]
\]
a Laurent polynomial (with integer coefficients) in $\{X_{\gamma_1}, \dots, X_{\gamma_n}\}$ where $T = \{\gamma_1, \dots, \gamma_n\}$ forms a triangulation. By the general theory of cluster algebras \cite{fomin2007cluster}, this Laurent expression is determined by the $\mathbf g$-vector $\mathbf g^T(X_\gamma)\in \mathbb Z^n$ and the $F$-polynomial $F^T(X_\gamma)\in \mathbb Z[y_1, \dots, y_n]$. One problem we address in this paper is to give a representation-theoretic interpretation of $\mathbf g$-vectors and $F$-polynomials, thus also the Laurent expressions of cluster variables.

Let $\mathbf{Arc}(\mathcal O)$ be the \emph{arc complex} of $\mathcal O$ whose simplices are partial triangulations. In particular, its vertices are arcs and maximal simplices are triangulations. It is shown in \cite{felikson2012cluster} (extending the result of \cite{fomin2008cluster}) that $\mathbf {Arc}(\mathcal O)$ is isomorphic to the cluster complex of $\mathcal A(\mathcal O)$. Note that each codimension one simplex is a face of exactly two maximal ones exchanged by a flip. We call the $n$-regular dual graph $E(\mathcal O)$ of $\mathbf{Arc}(\mathcal O)$ the \emph{flip graph of triangulations} or the \emph{cluster exchange graph} of $\mathcal A(\mathcal O)$.

\subsection{Gentle algebras from triangulations}
Following \cite{LM1}, for each triangulation $T$ of $\mathcal O$, we construct a finite-dimensional gentle algebra $\mathcal P(T)$ from a quiver $Q(T)$ with relations. The construction is local to each triangle in $T$ as illustrated in \Cref{fig: triangulation quiver}. The existence of loops thus comes from the orbifold points $\mathbb O$. The quiver $Q(T)$ is then obtained by pasting together these rank 3 quivers along their common vertices.
\begin{figure}[h]
    \centering
    \begin{tikzpicture}
        \filldraw[fill=gray!20, thick] (0, 0) -- (2, 0) -- (1, 1.5) -- (0, 0);
        \filldraw[black] (0, 0) circle (1.5pt);
        \filldraw[black] (2, 0) circle (1.5pt);
        \filldraw[black] (1, 1.5) circle (1.5pt);
        \draw[] (1, 0) node[above] {$\gamma_3$};
        \draw[] (1.7, 0.8) node[] {$\gamma_1$};
        \draw[] (0.25, 0.8) node[] {$\gamma_2$};
        \draw[] (2.5, 0.75) node[]{$\leadsto$};
        \draw[] (4, 0.75) node[]{\begin{tikzcd}[column sep = small]
        & 3 \ar[dr, leftarrow, "b"] & \\
        1 \ar[ru, leftarrow, "a"] & & 2 \ar[ll, leftarrow, "c"]
        \end{tikzcd}};
    \end{tikzpicture}
    \quad\quad\quad
    \begin{tikzpicture}
        \filldraw[fill=gray!20, thick](0,0) circle (1);
        \filldraw[black] (0, 1) circle (1.5pt);
        \filldraw[black] (0, -1) circle (1.5pt);
        \draw[] (0, 0) node[]{$\star$};
        \draw[thick] plot [smooth cycle] coordinates {(0,-1) (-0.2, 0) (0, 0.2) (0.2, 0)};
        \draw[] (0, 0.4) node[]{\small $\gamma_3$};
        \draw[] (-0.7, 0) node[]{\small $\gamma_2$};
        \draw[] (0.7, 0) node[]{\small $\gamma_1$};
        \draw[] (2, 0) node[]{$\leadsto$};
        \draw[] (3.5, 0) node[]{\begin{tikzcd}[column sep = small]
        & 3 \ar[dr, leftarrow, "b"]\ar[loop above, leftarrow, "\varepsilon_3"] & \\
        1 \ar[ru, leftarrow, "a"] & & 2 \ar[ll, leftarrow, "c"]
        \end{tikzcd}};
    \end{tikzpicture}  
    \caption{From triangulation to quiver}
    \label{fig: triangulation quiver}
\end{figure}
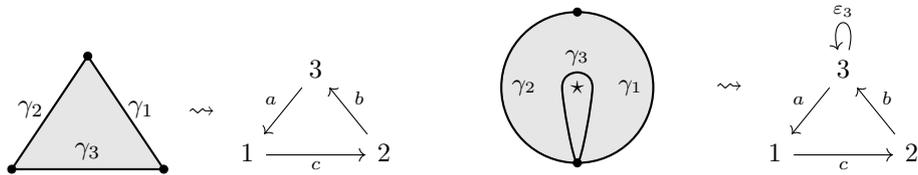

We impose a set of gentle relations $ba = ac = cb = 0$ for each triangle and $\varepsilon_k^2 = 0$ for each pending arc. Denote by $I(T)$ the ideal in the path algebra $\Bbbk Q(T)$ generated by these relations. We consequently obtain a gentle algebra $\mathcal P(T) \coloneqq \Bbbk Q(T)/ I(T)$. When $\mathbb O = \varnothing$, the algebra $\mathcal P(T)$ is the one defined for surfaces in \cite{labardini2009quivers} and \cite{assem2010gentle}. The case when $\Sigma$ is a disk and $|\mathbb O| = 1$ goes back to \cite{labardini2019on}.

\subsection{Relation between $\tau$-tilting theory and cluster combinatorics}\label{subsection: intro tau tilting}
Let $\Lambda$ be a finite-dimensional basic algebra over some field $\Bbbk$, with a complete set of $n$ pairwise orthogonal primitive idempotents $e_1,\dots, e_n$. Write $\modu \Lambda$ for the category of finitely generated left $\Lambda$-modules. Let $\tau$ denote the Auslander--Reiten translation. A \emph{support $\tau$-tilting pair} $(M, P)$ consists of $M\in \modu \Lambda$ and $P\in \modu \Lambda$ that is projective such that 
\[
    \Hom_\Lambda(M, \tau M) = 0 \quad \text{and} \quad \Hom_\Lambda(P, M) = 0,
\]
and $|M| + |P| = n$ where $|\cdot |$ denotes the number of pairwise non-isomorphic indecomposable summands for an object in $\modu \Lambda$. Denote by $\stau{\Lambda}$ the set of basic support $\tau$-tilting pairs up to isomorphism. We enhance $\stau{\Lambda}$ to a graph such that any two $\mathcal M$ and $\mathcal M'$ sharing $n-1$ indecomposable summands are joined by an edge. Adachi, Iyama, and Reiten \cite{adachi2014tau} have shown that $\stau{\Lambda}$ is $n$-regular. In other words, any $\mathcal M = \oplus_{i = 1}^n \mathcal M_i \in \stau{\Lambda}$ mutates in exactly $n$ different directions and the AIR-mutation $\mu_{\mathcal M_i}(\mathcal M)$ is just another support $\tau$-tilting pair $(\mathcal M/\mathcal M_i) \oplus \mathcal M_i'$. Hence the graph $\stau{\Lambda}$ is cluster-like where $\mathcal M \in \stau{\Lambda}$ is analogous to a cluster and a summand $\mathcal M_i$ to a cluster variable.

An indecomposable summand of $\mathcal M\in \stau{\Lambda}$ is either of the form $(M, 0)$ or $(0, P)$, which we call an indecomposable $\tau$-rigid pair. In fact, it is considered in \cite{demonet2019tau} the simplicial complex $\Delta(\Lambda)$ on the ground set of all indecomposable $\tau$-rigid pairs (up to isomorphism) where the maximal simplices are elements in $\stau{(\Lambda)}$. The structure of mutations implies that $\Delta(\Lambda)$ is a pseudo-manifold of pure dimension $n$, just as a cluster complex. The dual graph of $\Delta(\Lambda)$ is $\operatorname{s\tau-tilt} (\Lambda)$.

It is useful to work with a realization of $\Delta(\Lambda)$ as a simplicial cone complex in $\mathbb R^n$ \cite{demonet2019tau}, which is through an injective map
\[
    \begin{tikzcd}[column sep = small]
        \mathbf g^\Lambda \colon \{\text{$\tau$-rigid pairs of $\Lambda$}\}/\!\cong \ar[r] & \mathbb Z^n \subset \mathbb R^n
    \end{tikzcd}
\]
respecting direct sums that are still $\tau$-rigid. We call $\mathbf g^\Lambda(\mathcal M)$ the (representation-theoretic) $\mathbf g$-vector of a $\tau$-rigid pair $\mathcal M$; see \Cref{prop: g vector by inj} for a definition.

Returning to our situation, for the cluster algebra $\mathcal A(\mathcal O)$ and the finite-dimensional algebra $\mathcal P(T)$ associated to a triangulation $T = \{\gamma_1, \dots, \gamma_n\}$, we have

\begin{theorem}[\Cref{thm: main theorem} (A) and \Cref{cor: app in cluster algebra} (2)]\label{thm: main thm 1}
    There is an isomorphism $\Phi_T \colon \Delta(\mathcal P(T)) \rightarrow \Delta(\mathcal A(\mathcal O)) = \mathbf{Arc}(\mathcal O)$ of (abstract) simplicial complexes such that $\Phi_T((0, \mathcal P(T)e_i)) = \gamma_i$. In particular, it gives a bijection from indecomposable $\tau$-rigid pairs (up to isomorphism) to cluster variables in $\mathcal A(\mathcal O)$ (arcs in $\mathcal O$) and a bijection from basic support $\tau$-tilting pairs (up to isomorphism) to clusters in $\mathcal A(\mathcal O)$ (triangulations of $\mathcal O$), respecting mutations on both sides. Moreover, for any indecomposable $\tau$-rigid pair $\mathcal M$, we have
    \[
        \mathbf g^{\mathcal P(T)}(\mathcal M)= \mathbf g^T(\Phi_T(\mathcal M)) \in \mathbb Z^n,
    \]
    where the latter is the $\mathbf g$-vector associated to the cluster variable $\Phi_T(\mathcal M)$.
\end{theorem}

We note that in the above theorem, the right-hand $\mathbf g$-vector is defined through cluster algebras of principal coefficients which depend on a choice of initial cluster which is exactly a triangulation $T$ in our case. As the theorem is valid for any triangulation $T$, a flip between $T$ and $T' = \mu_k(T)$ gives rise to the isomorphism
\begin{equation}\label{eq: iso complexes of two triangulations}
    \Phi_{T'}^{-1}\Phi_T \colon \Delta(\mathcal P(T)) \rightarrow \Delta(\mathcal P(T')).
\end{equation}
In particular, it induces a bijection between support $\tau$-tilting pairs intertwining AIR-mutations on the two sides. In fact, the map $\Phi_{T'}^{-1}\Phi_T$ is explicitly operated by a DWZ-type mutation of decorated representations we construct in \Cref{section: mutation}. It is through these operations that we achieve to show \Cref{thm: main thm 1}.

\subsection{Mutations of decorated representations}
Let us restrict the isomorphism in (\ref{eq: iso complexes of two triangulations}) to the sets of vertices, i.e., indecomposable $\tau$-rigid pairs. The induced bijection is given explicitly by a DWZ-type mutation $\mu_k$, our main tool developed in this paper.

Recall that the DWZ mutation \cite{derksen2008quivers} of a representation $M$ for the Jacobian algebra of a quiver with potential $(Q, W)$ is defined through the following diagrams locally at $k\in Q_0$:
\[
    \begin{tikzcd}[column sep = small]
        & M(k) \ar[rd, "\alpha_k"]& \\
        M(k)_{\mathrm{in}} \ar[ru, "\beta_k"] & & M(k)_{\mathrm{out}} \ar[ll, "\gamma_k"]
    \end{tikzcd}
    \quad
    \xrightarrow{\ \mu_k\ }
    \begin{tikzcd}[column sep = small]
        & \overline M(k) \ar[ld, "\bar \alpha_k", swap]& \\
        M(k)_{\mathrm{in}} \ar[rr, "\bar \gamma_k", swap] & & M(k)_{\mathrm{out}} \ar[lu, "\bar \beta_k", swap]
    \end{tikzcd}
\]
On the left diagram, $M(k)$, $M(k)_{\mathrm{in}}$ and $M(k)_{\mathrm{out}}$ are $\Bbbk$-vector spaces and the map $\gamma_k$ is defined through the potential $W$. The right diagram describes the module structure of the mutation $\overline M = \mu_k(M)$ incident to vertex $k$. For example, as a $\Bbbk$-vector space, we have
\[
    \overline M(k) \coloneqq \frac{\ker \gamma_k}{\image \beta_k} \oplus \image \gamma_k \oplus \frac{\ker \alpha_k}{\image \gamma_k}.
\]

Our construction of the mutation $\mu_k$ undergoes a modification of the above diagrams. First, in our situation the mutation $\mu_k$ is local--it only uses the information within the (generalized) quadrilateral where $k$ is a diagonal; see \Cref{subsection: proof define rep mutations}. Therefore the gadgets in the left diagram have concrete descriptions through the triangulation. In addition, the diagrams are now viewed as of $H_k$-modules where $H_k = \Bbbk$ if $k$ is an ordinary arc and $H_k = \Bbbk [\varepsilon_k]/\varepsilon_k^2$ if $k$ is a pending arc enclosing an orbifold point.

Almost for purely formal reasons, we allow mutations to perform on pairs $(M, P)$ (\Cref{subsection: mutation decorations}), where the projective module $P$ can be viewed as the `decoration' part of a decorated representation. For any vertex $i$ of $Q(T)$, define $E_i(T) \coloneqq H_i$ regarded as a $\mathcal P(T)$-module via extension by zero elsewhere, and let $P_i(T)$ be the projective cover of $E_i(T)$. The following theorem can be obtained from \Cref{cor: app in cluster algebra} (2) by considering two triangulations related by a flip.

\begin{theorem} \label{thm: mutation relate tau rigid introduction}
    The decorated mutation $\mu_k$ defines a bijection between the isomorphism classes of indecomposable $\tau$-rigid pairs of $\mathcal P(T)$ and those of $\mathcal P(T')$ such that
    \[
        \mu_k((0, P_i(T))) =  \begin{cases}
            (0, P_i(T')) \quad &\text{for $i\neq k$},\\
            (E_i(T'), 0) \quad &\text{for $i = k$}.
        \end{cases}
    \]
    This bijection induces the isomorphism $\Phi_{T'}^{-1}\Phi_T \colon \Delta(\mathcal P(T))\rightarrow \Delta(\mathcal P(T'))$ in (\ref{eq: iso complexes of two triangulations}).
\end{theorem}

\subsection{Locally free Caldero--Chapoton functions}
\Cref{thm: main thm 1} establishes a bijection between indecomposable $\tau$-rigid modules and non-initial cluster variables. To achieve (C2), we need a formula expressing (Laurent expansions of) cluster variables directly from modules. Such a formula goes back to Caldero and Chapoton, who proved it in \cite{caldero2006cluster} for Dynkin quivers. Closely related is the $F$-polynomial of a ($\mathbb C$-linear) representation $M$ of a quiver $Q$ where $Q_0 = \{1, \dots, n\}$
\[
    F(M)(y_1, \dots, y_n) \coloneqq \sum_{\mathbf e = (e_i)_i} \chi(\mathrm{Gr}(\mathbf e, M)) \prod_{i\in Q_0} y_i^{e_i}
\]
where $\mathrm{Gr}(\mathbf e, M)$ is the (complex) quiver Grassmannian of subrepresentations with dimension vector $\mathbf e$ and $\chi(\cdot)$ denotes the Euler characteristic in analytic topology. For any (2-acyclic) quiver $Q$ and $X$ a cluster variable in the cluster algebra $\mathcal A(Q)$, the (cluster-theoretic) $F$-polynomial of $X$ (defined through making principal coefficients at some cluster \cite{fomin2007cluster}) is known to equal $F(M)$ for a suitable representation $M$ of $Q$ corresponding to $X$; see for example \cite{derksen2010quivers}.

In our skew-symmetrizable situation, a variation of the classic $F$-polynomial seems necessary. We call $M\in \mathcal P(T)$ locally free if $e_i M$ is a free module over the local algebra $H_i$. The tuple $(\rk_{H_i} (e_i M))_i \in \mathbb N^n$ is called the rank vector of $M$. The \emph{locally free $F$-polynomial} of $M$ then only concerns locally free submodules, that is,
\begin{equation}\label{eq: loc free f poly}
    F(M)(y_1, \dots, y_n) = \sum_{\mathbf r = (r_i)_i} \chi(\mathrm{Gr}_\mathrm{l.f.}(\mathbf r, M)) \prod_{i\in Q_0} y_i^{r_i}
\end{equation}
where $\mathrm{Gr}_\mathrm{l.f.}(\mathbf r, M)$ is the set of locally free submodules of $M$ with rank vector $\mathbf r$; see \cite{geiss2018quivers} where this notion is firstly introduced in the acyclic case.

\begin{theorem}[\Cref{cor: app in cluster algebra} (3)]\label{thm: main thm 2}
    For $T = \{\gamma_1, \dots, \gamma_n\}$ a triangulation of $\mathcal O$, the bijection in \Cref{thm: main thm 1}
    \[
        \Phi_T\colon \{ \text{Indecomposable $\tau$-rigid modules of $\mathcal P(T)$} \}/\!\cong \ \xrightarrow{\ \sim \ } \{ \text{non-initial cluster variables in $\mathcal A(\mathcal O)$} \}
    \]
    is given precisely by the locally free Caldero--Chapoton function
    \[
        \mathrm{CC}^\mathrm{l.f.}(M) \coloneqq \left( \prod_{i} X_{\gamma_i}^{g_i} \right) \cdot F(M)(\hat y_1, \dots, \hat y_n)
    \]
    where $(g_i)_i\in \mathbb Z^n$ form the representation-theoretic $\mathbf g$-vector $\mathbf g^{\mathcal P(T)}(M)$ and $\hat y_i = \prod_{j = 1}^n X_{\gamma_j}^{b_{j,i}}$ with $(b_{j,i})$ forming the cluster exchange matrix associated to $T$.
\end{theorem}

Combining \Cref{thm: main thm 1} and \Cref{thm: main thm 2}, we have accomplished the `categorification' (C1) and (C2) for skew-symmetrizable cluster algebras $\mathcal A(\mathcal O)$.

\begin{remark}
    The first time that cluster variables in non-skew-symmetric cluster algebras were expressed with respect to other clusters as locally free Caldero--Chapoton functions of representations of quivers-with-loops with relations, occurred in the seminal paper \cite{geiss2018quivers} for cluster algebras of finite type when the initial seed is taken to be acyclic. Our result Theorem \ref{thm: main thm 2} recovers Geiss--Leclerc--Schröer's locally free Caldero--Chapoton expression of cluster variables for finite type $C$, and proves such expression for many cluster algebras of infinite type that do not have acyclic seeds at all. Thus, Theorem \ref{thm: main thm 2} is an ample extension of Geiss--Leclerc--Schröer's result.
\end{remark}

\subsection{Relation to \cite{LM1}}

This work is a sequel to \cite{LM1} from which we inherit the construction of the gentle algebra $\mathcal P(T)$. One may notice that the orbifold model in the prequel features orbifold points of order 3, whereas in the current paper they are of order 2. This discrepancy in orbifold structures results in different cluster exchange relations in their respective Teichm\"uller spaces, corresponding respectively to the generalized cluster structure of Chekhov--Shapiro \cite{chekhov2011teichmller} and the ordinary one of Felikson--Shapiro--Tumarkin \cite{felikson2012cluster}. In \cite{LM1}, we categorify the generalized cluster structure using $\mathcal P(T)$-representations.

Nevertheless, the combinatorics of triangulations are identical in these two situations, leading to the same construction of $\mathcal P(T)$. \Cref{thm: main thm 1} can be essentially obtained from the results of \cite{LM1} and bridging the generalized cluster structure with the ordinary one. However, \Cref{thm: main thm 2} seems out of reach by the methods of \cite{LM1}. This difficulty serves as one of our motivations to develop the current variation of the DWZ mutation, which turns out powerful enough to give another proof of \Cref{thm: main thm 1}.

\subsection{Organization of the paper}

We start by reviewing preliminaries on cluster algebras in \Cref{section: preliminaries}. In particular, the notions of $\mathbf g$-vectors and $F$-polynomials are defined through cluster algebras with principal coefficients.

In \Cref{section: orbifolds}, we review the orbifold models of a class of skew-symmetrizable cluster algebras by Felikson--Shapiro--Thurston \cite{felikson2012cluster} and then describe the construction of finite-dimensional algebras associated to orbifold triangulations. We emphasize the view point from modulated graphs in \Cref{subsection: modulated graph} that the latter construction of mutations of representations is based on.

We introduce our main construction of mutations of representations in \Cref{section: mutation}. We also show that mutations behave particularly well on a class of locally free representations. In \Cref{section: recurrence}, we prove the recurrence of (representation-theoretic) $\mathbf g$-vectors and $F$-polynomials under mutations, which forms the main step of our categorification results \Cref{thm: main thm 1} and \Cref{thm: main thm 2}.

In \Cref{section: tau tilting} and \Cref{section: e invariant}, we make connection to Adachi--Iyama--Reiten's $\tau$-tilting theory and also to the $E$-invariants of Derksen--Weyman--Zelevinsky originally defined in the skew-symmetric case. We show that in our situation for certain locally free modules, the $E$-invariants are preserved under mutations and thus prove that mutations also relate $\tau$-rigid modules of the algebras $\mathcal P(T)$ and $\mathcal P(T')$ associated to two triangulations related by a flip.

\Cref{section: arc representation} is devoted to an explicit description of certain $\mathcal P(T)$-modules in terms of arcs and bands. They behave particularly well under mutations.

Our main categorification result is stated and proved in \Cref{section: application cluster algebras}. We explain in \Cref{section: example} how the result applies to the affine type $\widetilde C_n$ cluster algebras as a special case. More detailed descriptions are given in type $\widetilde C_2$.

\section{Preliminaries on cluster algebras}\label{section: preliminaries}

In this section we review basics of cluster algebras of Fomin and Zelevinsky \cite{fomin2002cluster,fomin2003cluster,fomin2007cluster}. The notion of cluster complex is explained in \Cref{subsection: cluster complex}, which will be useful describing the combinatorial structure of cluster algebras in later sections.

\subsection{Definition}
Let $\mathbb P = \operatorname{Trop}(y_1, \dots, y_n)$, the \emph{tropical semifield} of rank $n$. As an abelian group (written multiplicatively), it is freely generated by $y_1, \dots, y_n$ with any element written in the form $\prod y_j^{a_j}$ for $a_j\in \mathbb Z$. The structure of $\mathbb P$ includes another operation $\oplus$, called addition, defined by
\[
    \prod_j y_j^{a_j} \oplus \prod_j y_j^{b_j} = \prod_j y_j^{\operatorname{min}(a_j, b_j)}.
\]
Denote by $\mathbb Z\mathbb P$ the group ring of $\mathbb P$. Let $\mathcal F$ be a field of rational functions in $n$ variables with coefficients in $\operatorname{Frac}(\mathbb {Z}\mathbb P)$, the fraction field of $\mathbb {Z}\mathbb P$.

\begin{definition}
    A (labeled) \emph{seed} is a triple $(\mathbf x, \mathbf p, B)$, where
    \begin{itemize}
        \item $\mathbf p = (p_i)_{i=1}^n$ is an $n$-tuple of elements of $\mathbb P$ called $\emph{coefficients}$;
        \item $\mathbf x = (x_i)_{i=1}^n$ is an $n$-tuple of algebraically independent rational functions in $\mathcal F$;
        \item $B = (b_{i,j})$ is an $n\times n$ skew-symmetrizable integer matrix.
    \end{itemize}
\end{definition}

\begin{definition}
    For any index $k \in \{1, \dots, n\}$, the \emph{mutation} of a seed $(\mathbf x, \mathbf p, B)$ in direction $k$ is a new seed $(\mathbf x', \mathbf p', B') = \mu_k(\mathbf x, \mathbf p, B)$ such that
    \begin{align}
        & x_i' = \begin{cases}
            x_i \quad & \text{if $i \neq k$} \\
            \dfrac{p_k \prod x_i^{[b_{ik}]_+} + \prod x_i ^{[-b_{ik}]_+}}{(p_k \oplus 1)x_k} \quad & \text{if $i = k$}
        \end{cases} \label{eq: exchange relation principal} \\
        & p_i' = \begin{cases}
            p_ip_k^{[b_{ki}]_+}(p_k \oplus 1)^{-b_{ki}} \quad & \text{if $i \neq k$} \\
            p_k^{-1} \quad & \text{if $i = k$}\\
        \end{cases} \\
        & b_{ij}' = \begin{cases}
            -b_{ij} \quad & \text{if $i = k$ or $j = k$} \\
            b_{ij} + [-b_{ik}]_+b_{kj} + b_{ik} [b_{jk}]_+ \quad & \text{otherwise}.
        \end{cases} \label{eq: matrix mutation}
    \end{align}
\end{definition}

The mutation $\mu_k$ is involutive, that is, $\mu_k(\mu_k(\mathbf x, \mathbf p, B)) = (\mathbf x, \mathbf p, B)$. Two seeds are \emph{mutation-equivalent} if one can be obtained from the other by a finite sequence of mutations (in possibly various directions).

For a skew-symmetrizable matrix $B$, we set the \emph{initial seed with principal coefficients} to be 
\[
    (\mathbf x, \mathbf y, B) = ((x_1, \dots, x_n), (y_1, \dots, y_n), B).
\]
For any seed $(\mathbf x', \mathbf y', B')$ mutation equivalent to the initial seed, we shall call
\[
    \text{$\mathbf x'$ a \emph{cluster}, $B'$ an \emph{exchange matrix} and any $x'_i\in \mathbf x'$ a \emph{cluster variable}.}
\]

\begin{definition}
    The \emph{cluster algebra with principal coefficients} $\mathcal A_\bullet (B)$ associated with the matrix $B$ is the ring of $\mathcal F$ generated by all cluster variables.
\end{definition}

\subsection{Laurent phenomenon and $F$-polynomial}
It is convenient to organize all (labeled) seeds of $\mathcal A_\bullet$ as follows. We work with the infinite $n$-regular tree $\mathbb T_n$ with a chosen root $t_0$ whose edges are labeled by $\{1, \dots, n\}$ such that the labels of the $n$ edges incident to any vertex are distinct. We associate the initial seed $(\mathbf x, \mathbf y, B)$ to $t_0$. This determines a pattern assigning any $t\in \mathbb T_n$, $t \mapsto (\mathbf x_t, \mathbf y_t, B_t)$ by the mutation rule that if $t$ and $t'$ are joined by a $k$-labeled edge in $\mathbb T_n$, then 
\[
    \mu_k(\mathbf x_t, \mathbf y_t, B_t) = (\mathbf x_{t'}, \mathbf y_{t'}, B_{t'}).
\]

Denote the $\ell$-th cluster variable in $\mathbf x_t$ by $X_{\ell; t} = X_{\ell; t}^{B; t_0}$ with the superscripts emphasizing the initial exchange matrix and the root. Notice that each $X_{\ell; t}$ is a rational function of $x_1, \dots, x_n, y_1, \dots, y_n$ as a subtraction-free rational expression. We define the rational function
\begin{equation}\label{eq: cluster f polynomial}
    F_{\ell; t} = F_{\ell; t}^{B; t_0}(y_1, \dots, y_n) \coloneqq X_{\ell; t}^{B; t_0}(1, \dots, 1; y_1, \dots, y_n) \in \mathbb Q(y_1, \dots, y_n).
\end{equation}
The following theorem of Fomin and Zelevinsky states that each $F_{\ell; t}^{B; t_0}$ is actually a polynomial with integer coefficients. It is well-known by the name \emph{Laurent phenomenon} as it says each cluster variable is a Laurent polynomial.

\begin{theorem}[{\cite[Proposition 3.6]{fomin2007cluster}}]
    The cluster algebra with principal coefficients $\mathcal A_\bullet (B)$ is contained in $\mathbb Z[x_1^\pm, \dots, x_n^\pm; y_1, \dots, y_n]$. Thus in particular we have
    \[
        X_{\ell; t}\in \mathbb Z[x_1^\pm, \dots, x_n^\pm; y_1, \dots, y_n] \quad \text{and} \quad F_{\ell; t}\in \mathbb Z[y_1, \dots, y_n].
    \]
\end{theorem}
The polynomial $F_{\ell; t}$ is called the \emph{$F$-polynomial} of $X_{\ell;t}$. We note that it depends on the choice of an initial seed.

\subsection{Principal grading and $\mathbf g$-vector}\label{subsection: principal grading g vector}

The cluster algebra with principal coefficients $\mathcal A_\bullet (B)$ has a $\mathbb Z^n$-grading such that each cluster variable is homogeneous.

\begin{proposition}[{\cite[Proposition 6.1]{fomin2007cluster}}]
    Every $X_{\ell;t}^{B;t_0} \in \mathcal A_\bullet (B)$ is homogeneous with respect to the $\mathbb Z^n$-grading in $\mathbb Z[x_1^\pm, \dots, x_n^\pm; y_1, \dots, y_n]$ given by 
    \[
        \deg (x_i) = \mathbf e_i\quad \text{and} \quad \deg(y_j) = - \mathbf b_j,
    \]
    where $\mathbf e_1,\dots, \mathbf e_n$ are the standard basis vectors of $\mathbb Z^n$ and $\mathbf b_j$ is the $j$th column vector of $B$.
\end{proposition}

We define the $\mathbf g$-vector of the cluster variable $X_{\ell;t}^{B; t_0}$ to be
\begin{equation}\label{eq: cluster g vector}
    \mathbf g_{\ell; t} = \mathbf g_{\ell; t}^{B; t_0} \coloneqq \deg\left( X_{\ell; t}^{B; t_0} \right) \in \mathbb Z^n.
\end{equation}

The Laurent polynomial $X_{\ell; t}$ is in fact determined by its $F$-polynomial and $\mathbf g$-vector through the expression
\begin{equation}\label{eq: express cluster variable g vector f poly}
    X_{\ell; t}(x_1, \dots, x_n; y_1, \dots, y_n) = \mathbf x^{\mathbf g_{\ell; t}} \cdot F_{\ell; t}(\hat y_1, \dots, \hat y_n)
\end{equation}
where $\mathbf x^{\mathbf g_{\ell; t}}$ denotes $\prod_i x_i^{g_i}$ if writing $\mathbf g_{\ell; t} = (g_1, \dots, g_n)$ and $\hat y_i = y_i \prod_j x_j^{b_{ji}}$.

\subsection{Recurrences}

We now discuss the changing of $\mathbf g$-vectors and $F$-polynomials when moving the root $t_0$ along $\mathbb T_n$.

We define a new class of integer vectors $\mathbf h^{B; t_0}_{\ell; t} = (h_1, \dots, h_n)$ from the $F$-polynomial $F_{\ell; t}$ given by
\begin{equation}\label{eq: def cluster h vector}
    x_1^{h_1} \cdots x_n^{h_n} \coloneqq F_{\ell; t}^{B; t_0}|_{\operatorname{Trop}(x_1, \dots, x_n)}(x_i^{-1}\prod_{j\neq i}x_j^{[-b_{ji}]_+} \leftarrow y_i).
\end{equation}
Namely, with a subtraction-free rational expression of $F_{\ell; t}$, we can substitute $y_i$ by $x_i^{-1}\prod_{j\neq i}x_j^{[-b_{ji}]_+}$ and evaluate in the tropical semifield $\operatorname{Trop}(x_1, \dots, x_n)$ to obtain the element $x_1^{h_1}\cdots x_n^{h_n}$.

Suppose that we have $t_0 \frac{k}{\quad\quad} t_1$ in $\mathbb T_n$. Consider the seed $\mu_k(\mathbf x, \mathbf y, B) = (\mathbf x_{t_1}^{B; t_0}, \mathbf y_{t_1}^{B; t_0}, B_{t_1}^{B; t_0})$. Denote $B_1 = B_{t_1}^{B; t_0}$. Let the vertex $t_1$ be the root of $\mathbb T_n$ but leave the labeling of edges unchanged. The tree $\mathbb T_n$ now carries a whole new pattern assigning seeds to vertices
\[
    t\mapsto (\mathbf x_{t}^{B_1; t_1}, \mathbf y_{t}^{B_1; t_1}, B_{t}^{B_1; t_1})\quad \text{for $t\in \mathbb T_n$}.
\]

\begin{proposition}[{\cite[Proposition 2.4]{derksen2010quivers}}]\label{prop: recurrence of f poly g vector}
    Denote $\mathbf h^{B; t_0}_{\ell; t} = (h_1, \dots, h_n)$, $\mathbf h^{B_1; t_1}_{\ell; t} = (h_1', \dots, h_n')$, and $\mathbf y_{t_1}^{B; t_0} = (y_1', \dots, y_n')$. The $\mathbf g$-vectors $\mathbf g_{\ell; t}^{B; t_0} = (g_1, \dots, g_n)$ and $\mathbf g_{\ell; t}^{B_1; t_1} = (g_1', \dots, g_n')$ are related by
    \begin{equation}\label{eq: recursion g vector}
        g_j' = \begin{cases}
            -g_k \quad & \text{if $j = k$};\\
            g_j + [b_{jk}]_+g_k - b_{jk}h_k \quad & \text{if $j\neq k$}.
        \end{cases}
    \end{equation}
    Moreover, $g_k = h_k - h_k'$.
    The $F$-polynomials $F_{\ell; t}^{B; t_0}$ and $F_{\ell; t}^{B_1; t_1}$ are related by
    \begin{equation} \label{eq: recursion f poly}
        (y_k + 1)^{h_k} F_{\ell; t}^{B; t_0}(y_1, \dots, y_n) = (y_k' + 1)^{h_k'} F_{\ell; t}^{B_1; t_1}(y_1', \dots, y_n').
    \end{equation}
\end{proposition}

We briefly explain the idea of how the above recurrence can be used to determine the $\mathbf g$-vectors and $F$-polynomials for any $t\in \mathbb T_n$ recursively. Suppose that $t$ and $t_0$ are connected via the path
\[
    t_0 \frac{k_1}{\quad\quad} t_1 \frac{k_2}{\quad\quad} \dots\dots \frac{k_m}{\quad\quad} t_m \frac{k_{m+1}}{\quad\quad}  \dots\dots \frac{}{\quad\quad}t.
\]
First we view $t$ as the root and thus have for any $\ell\in \{1, \dots, n\}$ that
\[
    \mathbf g_{\ell; t}^{B_t; t} = \mathbf e_\ell \quad \text{and} \quad F_{\ell; t}^{B_t; t} = 1.
\]
Assume that the $\mathbf g$-vectors $\mathbf g_{\ell; t}^{B_{t_m}; t_m}$ and $F$-polynomials $F_{\ell; t}^{B_{t_m}; t_m}$ have all been obtained. Then by the formulas (\ref{eq: recursion g vector}) and (\ref{eq: recursion f poly}), the $\mathbf g$-vectors $\mathbf g_{\ell; t}^{B_{t_{m-1}}; t_{m-1}}$ and $F$-polynomials $F_{\ell; t}^{B_{t_{m-1}; t_{m-1}}}$ are all determined. The vertex $t_{m-1}$ is one edge closer to $t_0$ than $t_m$. Continuing this process gives a way computing $\mathbf g_{\ell; t}^{B; t_0}$ and $F_{\ell; t}^{B; t_0}$.

\subsection{Coefficient free cluster algebras}\label{subsection: coef free cluster}

The definition of $\mathcal A_\bullet (B)$ can be easily adapted to define the \emph{coefficient-free cluster algebra} $\mathcal A(B)$. A (labeled) seed $(\mathbf x, B)$ now has only two components a cluster $\mathbf x$ and an exchange matrix $B$. The change of cluster under mutation is simplified from (\ref{eq: exchange relation principal}) by evaluating the coefficients $p_i$ all at $1$:
\begin{equation}
    x_i' = \begin{cases}
    x_i \quad & \text{if $i\neq k$} \\
    x_k^{-1}(\prod x_i^{[b_{ik}]_+} + \prod x_i^{[-b_{ik}]_+}) & \text{if $i = k$}.
    \end{cases}
\end{equation}
The matrix mutation rule remains the same as (\ref{eq: matrix mutation}).

One can associate an initial seed $(\mathbf x, B)$ to a chosen root $t_0\in \mathbb T_n$. Under the above coefficient-free version of mutations, this determines an association of (coefficient-free) seeds to any $t\in \mathbb T_n$ in the same way as in the case of $\mathcal A_\bullet(B)$. The cluster algebra $\mathcal A(B)$ is then defined to be the subring in $\mathbb Q(x_1, \dots, x_n)$ generated by all cluster variables.

Alternatively, one can consider the homomorphism
\[
    \pi \colon \mathbb Z[x_1^\pm, \dots, x_n^\pm; y_1, \dots, y_n] \rightarrow \mathbb Z[x_1^\pm, \dots, x_n^\pm],\quad x_i\mapsto x_i, \quad y_i \mapsto 1,
\]
that maps $\mathcal A_\bullet(B)$ to the image $\mathcal A(B)$. Applied to (\ref{eq: express cluster variable g vector f poly}), it provides an expression of cluster variables in $\mathcal A(B)$ through $\mathbf g$-vectors and $F$-polynomials.

\begin{lemma}
    The cluster variable $x_{\ell; t}^{B; t_0}$ in the coefficient-free cluster algebra $\mathcal A(B)$ can be expressed as
    \[
        x_{\ell; t}^{B; t_0} = \mathbf x^{\mathbf g_{\ell; t}^{B; t_0}} F_{\ell; t}^{B; t_0}(\hat y_1, \dots, \hat y_n)
    \]
    where $\hat y_i$ now denotes $\prod_j x_j^{b_{ji}}$.
\end{lemma}

We remark that in the coefficient-free case, the patterns of assigning seeds to $\mathbb T_n$ actually remain unchanged when changing roots.

\subsection{Exchange graph and cluster complex}\label{subsection: cluster complex}

Let $\mathcal A$ be either the cluster algebra $\mathcal A_\bullet (B)$ or $\mathcal A(B)$. A seed in $\mathcal A$ as previously defined is labeled. Now an \emph{unlabeled seed} is an equivalence class of seeds under permutations of indices $\{1, \dots, n\}$.

\begin{definition}
    The \emph{exchange graph} $\mathbf E(\mathcal A)$ of $\mathcal A$ is the $n$-regular graph whose vertices are labeled by unlabeled seeds, and whose edges correspond to mutations. It can be viewed as the quotient graph of $\mathbb T_n$ by identifying labeled seeds up to permutations.
\end{definition}

\begin{conjecture}[{\cite[Section 1.5]{fomin2003cluster}}]\label{conj: cluster determins seed}
    Any seed $(\mathbf x, \mathbf p, B)$ (or $(\mathbf x, B)$) is uniquely determined by its cluster $\mathbf x$. Thus for any cluster $\mathbf x$ and any $x\in \mathbf x$, there is a unique cluster $\mathbf x'$ with $\mathbf x' \cap \mathbf x = \mathbf x \setminus \{x\}$.
\end{conjecture}

\begin{definition}[{\cite{fomin2003cluster}}]
    Under the \Cref{conj: cluster determins seed}, the \emph{cluster complex} $\Delta(\mathcal A)$ of $\mathcal A$ is the simplicial complex on the vertex set of cluster variables whose maximal simplices are (unlabeled) clusters. The exchange graph $\mathbf E(\mathcal A)$ is then the dual graph of $\Delta(\mathcal A)$: the vertices of $\mathbf E(\mathcal A)$ are (unlabeled) clusters, with edges corresponding to pairs of clusters whose intersection has cardinality $n-1$.
\end{definition}

\begin{remark}
    \Cref{conj: cluster determins seed} is proved in \cite{fomin2008cluster} for cluster algebras associated to surfaces. Using quiver representations, this conjecture is proved in general for all skew-symmetric cluster algebras in \cite{derksen2010quivers}. Extending the DWZ representation-theoretic approach to the skew-symmetrizable cluster algebras associated to orbifolds, our results imply the conjecture in this case.
\end{remark}

\section{Orbifolds, triangulations, and associated algebras}\label{section: orbifolds}

In this section we view the orbifold model of a class of skew-symmetrizable cluster algebras \cite{felikson2012cluster} and then give the construction of gentle algebras associated to triangulations of orbifolds.

\subsection{Orbifolds and triangulations}\label{subsection: orbifold and triangulations}

Let $\Sigma$ be an oriented surface with non-empty boundary $\partial \Sigma$. Let $(\Sigma, \mathbb M, \mathbb O)$ be a triple of the surface $\Sigma$, a finite set $\mathbb M \subset \partial \Sigma$ called \emph{(boundary) marked points}, and a finite set $\mathbb O \subset \Sigma \setminus \partial \Sigma$ of special points called \emph{orbifold points}.

We call such a triple $\mathcal O = (\Sigma, \mathbb M, \mathbb O)$ an \emph{orbifold}.

An $\emph{arc}$ $\gamma$ in $\mathcal O$ is a continuous map $\gamma \colon [0, 1] \rightarrow \Sigma \setminus \mathbb O$ considered up to homotopy relative to the endpoints such that
\begin{itemize}
    \item both endpoints of $\gamma$ belong to $\mathbb M$;
    \item $\gamma$ has no self-intersections, except that its two endpoints may coincide;
    \item $\gamma$ does not cut out a monogon not containing points of $\mathbb O$;
    \item $\gamma$ is not homotopic to a boundary segment.
\end{itemize}

If an arc $\gamma$ cuts out a monogon containing exactly one orbifold point, we call $\gamma$ a \emph{pending arc}. Otherwise, it is called an \emph{ordinary arc}.

In \Cref{fig: orbifold example}, the points in $\mathbb M$ are drawn as $\bullet$ and orbifold points in $\mathbb O$ as $\star$.

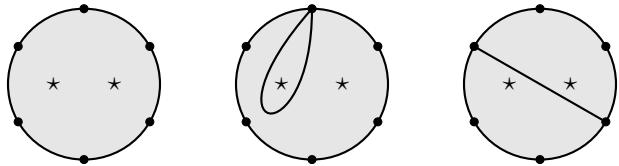
\begin{figure}[h]
    \centering
    \begin{tikzpicture}
        \filldraw[fill=gray!20, thick](0,0) circle (1);
        \node[] at (-0.4, 0) {$\star$};
        \node[] at (0.4, 0) {$\star$};
        
        \filldraw(0, -1) circle (1.5pt);
        \filldraw(0, 1) circle (1.5pt);
        \filldraw(-0.866, -0.5) circle (1.5pt);
        \filldraw(0.866, -0.5) circle (1.5pt);
        \filldraw(-0.866, 0.5) circle (1.5pt);
        \filldraw(0.866, 0.5) circle (1.5pt);

        \filldraw[fill=gray!20, thick](0+3,0) circle (1);
        \draw[thick] (0+3, 1) .. controls (-1.5+3, -0.7) and (0+3, -1) .. (0+3, 1);
        \node[] at (-0.4+3, 0) {$\star$};
        \node[] at (0.4+3, 0) {$\star$};
        
        \filldraw(0+3, -1) circle (1.5pt);
        \filldraw(0+3, 1) circle (1.5pt);
        \filldraw(-0.866+3, -0.5) circle (1.5pt);
        \filldraw(0.866+3, -0.5) circle (1.5pt);
        \filldraw(-0.866+3, 0.5) circle (1.5pt);
        \filldraw(0.866+3, 0.5) circle (1.5pt);

        \filldraw[fill=gray!20, thick](0+6,0) circle (1);
        \draw[thick] (-0.866+6, 0.5) to (0.866+6, -0.5);
        \node[] at (-0.4+6, 0) {$\star$};
        \node[] at (0.4+6, 0) {$\star$};
        
        \filldraw(0+6, -1) circle (1.5pt);
        \filldraw(0+6, 1) circle (1.5pt);
        \filldraw(-0.866+6, -0.5) circle (1.5pt);
        \filldraw(0.866+6, -0.5) circle (1.5pt);
        \filldraw(-0.866+6, 0.5) circle (1.5pt);
        \filldraw(0.866+6, 0.5) circle (1.5pt);
    \end{tikzpicture}
    \caption{Hexagon with two orbifold points; a pending arc; an ordinary arc}
    \label{fig: orbifold example}
\end{figure}

We say two arcs $\gamma$ and $\gamma'$ are compatible if they do not intersect in the interior of $\Sigma$. A \emph{triangulation} $T$ of $\mathcal O$ is a maximal collection of distinct pairwise compatible arcs. The (closures of) the connect components of $\Sigma\setminus \bigcup_{\gamma\in T} \gamma([0,1])$ are called \emph{triangles}. A triangle is \emph{singular} if it contains an orbifold point of $\mathbb O$ and otherwise \emph{regular}. Notice that any triangulation of $\mathcal O = (\Sigma, \mathbb M, \mathbb O)$ has $|\mathbb O|$ many singular triangles as there must be $|\mathbb O|$ pending arcs, each enclosing an orbifold point of $\mathbb O$.

Given a triangulation $T$, the \emph{flip} of $T$ at an arc $\gamma\in T$ is another triangulation $T' \neq T$ obtained by replacing $\gamma$ with a unique arc $\gamma' \neq \gamma$, which evidently always exists. We denote the flip by $\mu_\gamma(T) = T'$.

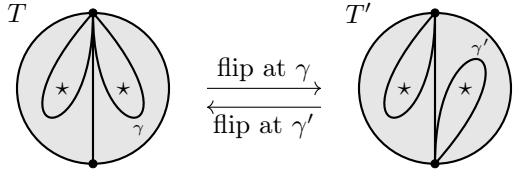
\begin{figure}[h]
    \centering
    \begin{tikzpicture}
        \filldraw[fill=gray!20, thick](0,0) circle (1);
        \draw[thick] (0,-1) -- (0,1);
        \draw[thick] (0,1) .. controls (-1.5, -0.7) and (0, -1) .. (0,1);
        \draw[thick] (0,1) .. controls (1.5, -0.7) and (0, -1) .. (0, 1);
        \node[] at (-0.4, 0) {$\star$};
        \node[] at (0.4, 0) {$\star$};
        \node[] at (0.6, -0.55) {\tiny $\gamma$};
        
        \filldraw(0, -1) circle (1.5pt);
        \filldraw(0, 1) circle (1.5pt);
        \node[] at (-1, 1) {$T$};

        \draw[->] (1.5, 0) -- (3, 0);
        \node[] at (2.25, 0.25) {flip at $\gamma$};
        \draw[->] (3,-0.25) -- (1.5, -0.25);
        \node[] at (2.25, -0.5) {flip at $\gamma'$};

        \filldraw[fill=gray!20, thick](2.5+2,0) circle (1);
        \node[] at (2.1+2, 0) {$\star$};
        \node[] at (2.9+2, 0) {$\star$};
        \draw[thick] (2.5+2,-1) -- (2.5+2,1);
        \draw[thick] (2.5+2,1) .. controls (1+2, -0.7) and (2.5+2, -1) .. (2.5+2,1);
        \draw[thick] (2.5+2,-1) .. controls (4+2, 0.7) and (2.5+2, 1) .. (2.5+2, -1);
        
        \node[] at (3.1+2, 0.55) {\tiny $\gamma'$};
        \node[] at (1.5+2, 1) {$T'$};
        \filldraw(2.5+2, -1) circle (1.5pt);
        \filldraw(2.5+2, 1) circle (1.5pt);
    \end{tikzpicture}
    \caption{Flip at a pending arc}
    \label{fig: flip of a pending arc}
\end{figure}

\subsection{Cluster algebras associated to orbifolds}

We start by associating a skew-symmetrizable matrix $B(T)$ to a triangulation $T$. For convenience, the arcs in $T$ are labeled by $\{1, \dots, |T|\}$ where the number $|T|$ is determined by $\mathcal O$ and shared by all triangulations. Precisely,
\[
    |T| = 6g + 3b + |\mathbb{M}| + |\mathbb{O}| - 4,
\]
where $g$ is the genus of $\Sigma$ and $b$ is the number of boundary components.

For any two arcs $i$ and $j$ in $T$, we define if $j$ is ordinary,
\[
    b_{i,j}\coloneqq |\{\text{triangles where $j$ follows $i$ clockwise}\}|-|\{\text{triangles where $i$ follows $j$ clockwise} \}|,
\]
and if $j$ is pending,
\[
    b_{i,j}\coloneqq 2|\{\text{triangles where $j$ follows $i$ clockwise}\}|-2|\{\text{triangles where $i$ follows $j$ clockwise} \}|.
\]
The $n\times n$ matrix $B(T)$ is defined to have entries $b_{i,j}$. Let $D(T) = \mathrm{diag} (d_i\mid i \in T)$ such that
\[
    d_i = \begin{cases}
        1 \quad & \text{if $i$ is ordinary}\\
        2 \quad & \text{if $i$ is pending}.
    \end{cases}
\]
The matrix $B=B(T)$ is skew-symmetrized by $D=D(T)$ in the sense that $DB$ is skew-symmetric, that is, $DB + (DB)^\intercal = 0$. In particular, the matrix $B$ becomes skew-symmetric when the set $\mathbb O$ of orbifold points is empty.

\begin{remark}\label{rmk: matrix}
(1) The matrix $B(T)$ defined above is equal to the matrix considered by the authors in \cite{LM1}, see the paragraph preceding Lemma 3.7 therein; (2) if $\orbset=\varnothing$, then $B(T)$ is equal to the matrix associated to $T$ by Fomin--Shapiro--Thurston \cite[Definition 4.1]{fomin2008cluster} and Fomin--Thurston \cite[Definition 5.15]{fomin2018cluster}; (3) for $\orbset\neq\varnothing$, Felikson--Shapiro--Tumarkin \cite{felikson2012cluster} associate to each triangulation $2^{\orbset}$ skew-symmetrizable matrices, based on the choice of a \emph{weight} ($1/2$ or $2$ in their setting) for each orbifold point. Each choice of weights corresponds to choosing a skew-symmetrizable matrix for the underlying \emph{diagram}; see \cite[Definition 7.3]{fomin2003cluster}, \cite[Remark 4.7]{Geuenich-LF1} or \cite[Definition 2.12, Remark 3.5-(2) and \S13.1]{Geuenich-LF2}. The constructions of the present paper correspond to a constant choice of weights, which is the exactly the one taken in \cite{Geuenich-LF1}. In the case of polygons with exactly one orbifold point, this choice yields skew-symmetrizable matrices of mutation type $C$ --the other choice yields mutation type $B$.
\end{remark}

We list all three types of regular triangles in \Cref{fig: regular triangles}. The associated $3 \times 3$ matrices are respectively
\[
    \begin{bmatrix}
    0 & -1 & 1\\
    1 & 0 & -1\\
    -1 & 1 & 0
    \end{bmatrix}, \quad\quad \begin{bmatrix}
                    0 & -1 & 2\\
                    1 & 0 & -2\\
                    -1 & 1 & 0
                    \end{bmatrix}, \quad\quad \begin{bmatrix}
                                    0 & -2 & 2\\
                                    1 & 0 & -2\\
                                    -1 & 2 & 0
                                    \end{bmatrix}. 
\]
If any side of a triangle is a boundary segment, delete the corresponding row and column of the matrix. The following lemma is straightforward to check.

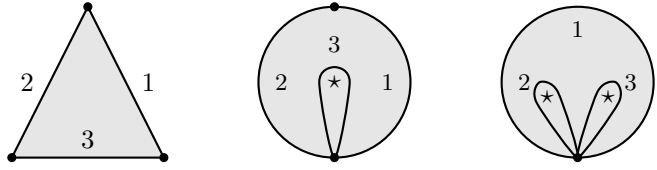
\begin{figure}[h]
    \centering
    \begin{tikzpicture}
        \filldraw[fill=gray!20, thick] (0, 0) -- (2, 0) -- (1, 2) -- (0, 0);
        \filldraw[black] (0, 0) circle (1.5pt);
        \filldraw[black] (2, 0) circle (1.5pt);
        \filldraw[black] (1, 2) circle (1.5pt);
        \draw[] (1, 0) node[above] {$3$};
        \draw[] (1.8, 1) node[] {$1$};
        \draw[] (0.2, 1) node[] {$2$};
    \end{tikzpicture}
    \quad\quad\quad
    \begin{tikzpicture}
        \filldraw[fill=gray!20, thick](0,0) circle (1);
        \filldraw[black] (0, 1) circle (1.5pt);
        \filldraw[black] (0, -1) circle (1.5pt);
        \draw[] (0, 0) node[]{$\star$};
        \draw[thick] plot [smooth cycle] coordinates {(0,-1) (-0.2, 0) (0, 0.2) (0.2, 0)};
        \draw[] (0, 0.5) node[]{\small $3$};
        \draw[] (-0.7, 0) node[]{\small $2$};
        \draw[] (0.7, 0) node[]{\small $1$};
    \end{tikzpicture}
    \quad\quad\quad
    \begin{tikzpicture}
        \filldraw[fill=gray!20, thick](0,0) circle (1);
        \filldraw[black] (0, -1) circle (1.5pt);
        \draw[] (-0.4, -0.2) node[]{$\star$};
        \draw[] (0.4, -0.2) node[]{$\star$};
        \draw[thick] plot [smooth cycle] coordinates {(0,-1) (-0.54, -0.27) (-0.5, 0) (-0.26, -0.13)};
        \draw[thick] plot [smooth cycle] coordinates {(0,-1) (0.54, -0.27) (0.5, 0) (0.26, -0.13)};
        \draw[] (0, 0.7) node[]{\small $1$};
        \draw[] (-0.7, 0) node[]{\small $2$};
        \draw[] (0.7, 0) node[]{\small $3$};
    \end{tikzpicture}
    \caption{Regular triangles}
    \label{fig: regular triangles}
\end{figure}

\begin{lemma}\label{lemma: flip and mutation matrix}
    If $T$ and $T'$ are related by a flip at $k\in T$, then $\mu_k(B(T)) = B(T')$.
\end{lemma}

The following statement can be extracted from Felikson--Shapiro--Tumarkin \cite[Theorem 9.1, Corollary 9.3]{felikson2012cluster}, extending the result from the surface case in \cite{fomin2008cluster}.

\begin{theorem}\label{thm: structure cluster algebra triangulation}
    Let $\mathcal A(T)$ be either the cluster algebra $\mathcal A(B(T))$ or $\mathcal A_\bullet(B(T))$. Then the following hold:
    \begin{enumerate}
        \item Each seed in $\mathcal A(T)$ is determined by its cluster.
        \item The cluster complex $\Delta(\mathcal A(T))$ is isomorphic to the arc complex $\mathbf {Arc}(\mathcal O)$ such that the initial cluster variables correspond to arcs in $T$. Hence cluster variables are in bijection with arcs and clusters with triangulations.
        \item The exchange matrix of the seed associated to a triangulation $T'$ is $B(T')$.
    \end{enumerate}
\end{theorem}

\begin{remark}
    In this paper we will not rely on the above theorem although some statements in the introduction are stated with its help. In fact, our representation-theoretic approach in this paper gives another proof of the theorem.
\end{remark}

\subsection{Gentle algebras associated to orbifolds}

Besides cluster algebras, we associate another class of algebras to orbifold triangulations. These (associative and non-commutative) algebras are finite-dimensional over a ground field $\Bbbk$, thus of a quite different flavor. However, we will see that the module categories of these algebras reveal cluster structures.

Let $T$ be a triangulation of an orbifold $\mathcal O$ and for convenience label the arcs of $T$ by $\{1, \dots, n = |T|\}$. Define a quiver $Q(T) = (Q_0, Q_1, t\colon Q_1 \rightarrow Q_0, h\colon Q_1 \rightarrow Q_0)$ as follows. The set of vertices $Q_0$ is the set of arcs $\{1, \dots, n\}$ in $T$. For each regular triangle $\Delta$, if an arc $i$ precedes another arc $j$ in $\Delta$ clockwise, then there is an arrow $a:j\rightarrow i$, i.e. $t(a) = j$ and $h(a) = i$. In this case we denote $a\in \Delta$. For each pending arc $k\in T$, there is a loop $\varepsilon_k$ at $k$, i.e. $t(\varepsilon_k) = h(\varepsilon_k) = k$. Denote by $\Bbbk Q$ the path algebra of a quiver $Q$.

Let $I(T)$ be the (2-sided) ideal generated by
\begin{equation}\label{eq: gentle relations}
    R(T) = \{ \varepsilon_k^2 \mid \text{$k\in T$ pending} \} \cup \{ a\cdot b \mid h(b) = t(a)\ \text{and}\  a, b\in \Delta\}
\end{equation}
where the second set runs over all regular triangles of $T$.

\begin{definition}\label{def: gentle algebra of orbifold}
    For any triangulation $T$ of an orbifold $\mathcal O$, we define the $\Bbbk$-algebra $\mathcal P(T)$ to be the quotient $\Bbbk Q(T)/I(T)$.
\end{definition}

\begin{remark}
It is not hard to see that $\mathcal P(T)$ is finite-dimensional over $\Bbbk$ and \emph{gentle} (as defined in \cite{AS87}). The gentleness of $P(T)$ will play a role only in Section \ref{section: arc representation}.
\end{remark}

\begin{example}\label{ex: digon two orbifold points}
We exhibit below the quiver $Q(T)$ of a triangulation $T$ of a digon with two orbifold points. The ideal $I$ is  $\langle \varepsilon_2^2, \varepsilon_3^2, ba, cb, ac \rangle$.
\[
    \begin{tikzpicture}
        \node[] at (-2, 0.2){$T = $};
        \filldraw[fill=gray!20, thick] (0, 0.2) circle (1.2);
        \filldraw[fill=gray!20, thick](0, -0.1) circle (0.9);
        \draw[] (-0.4, -0.2) node[]{$\star$};
        \draw[] (0.4, -0.2) node[]{$\star$};
        \draw[thick] plot [smooth cycle] coordinates {(0,-1) (-0.54, -0.27) (-0.5, 0) (-0.26, -0.13)};
        \draw[thick] plot [smooth cycle] coordinates {(0,-1) (0.54, -0.27) (0.5, 0) (0.26, -0.13)};
        \draw[] (0, 0.6) node[]{\small $1$};
        \draw[] (-0.7, 0) node[]{\small $2$};
        \draw[] (0.7, 0) node[]{\small $3$};
        \filldraw[black] (0, -1) circle (1.5pt);
        \filldraw[black] (0, 1.4) circle (1.5pt);
        \node[] at (3.5, 0.2){$Q(T) = $};
        \node[] at (6, 0.2){\begin{tikzcd}[column sep = small]
                & 1 \ar[dl, swap, "b"] & \\
                2 \ar[loop left, leftarrow, "\varepsilon_2"] \ar[rr, swap, "c"] & & 3  \ar[loop right, leftarrow, "\varepsilon_3"]\ar[ul, swap, "a"]
            \end{tikzcd}};    
    \end{tikzpicture}
\]
\end{example}

\subsection{Modulated graphs}\label{subsection: modulated graph}

We introduce yet more structures to our construction, which will be essential for what follows.

For each $i\in Q_0(T)$, let $H_i = \Bbbk [\varepsilon_i]/\varepsilon_i^{d_i}$ where $(d_i)_i$ is the symmetrizer, that is $d_i = 1$ for $i$ ordinary and $d_i = 2$ for $i$ pending. For $(i, j)\in Q_0(T)^2$, denote $E_{i,j} = \{a\in Q_1(T)\mid t(a)=j,\, h(a)=i\}$. When $i\neq j$, define the (free) $(H_i, H_j)$-bimodule
\[
    A(T)_{i, j} \coloneqq H_i\langle a\mid a\in E_{i,j}\rangle H_j.
\]
In other words, it is freely generated by all arrows from $j$ to $i$ as an $(H_i, H_j)$-bimodule. It is easy to see that $A(T)_{i,j}$ embeds in $\mathcal P(T)$ as a sub-bimodule.

We also define the (free) $(H_j, H_i)$-bimodule
\[
    \Delta(T)_{i, j} \coloneqq H_j\langle \bar a \mid a\in E_{i,j} \rangle H_i,
\]
where $\bar a$ is a formal symbol viewed as the reversed arrow of $a$.

The module $A(T)_{i,j}$ is free over $H_i$ of rank $d_j|E_{i,j}|$ with basis $L = \{a, \dots, a\varepsilon_j^{d_j-1}\mid a\in E_{i,j}\}$ and free over $H_j$ of rank $d_i|E_{i,j}|$ with basis $R = \{a, \dots, \varepsilon_i^{d_i-1} a\mid a\in E_{i,j}\}$. Consider $(H_j, H_i)$-bimodules
\[
    \Hom_{H_i}({A(T)_{i, j}}, H_i)\quad \text{and} \quad \Hom_{H_j}({A(T)_{i, j}}, H_j).
\]
They have dual bases $\{a^*_L, \dots, (a\varepsilon_j^{d_j-1})^*_L\mid a\in E_{i,j}\}$ (with respect to $L$) and $\{a^*_R, \dots, (\varepsilon_i^{d_i-1}a)^*_R\mid a\in E_{i,j}\}$ (with respect to $R$). These two bimodules are isomorphic. In fact, there are bimodule isomorphisms
\begin{align}
    \rho \colon {\Delta(T)_{i, j}} \rightarrow \Hom_{H_i}({A(T)_{i, j}}, H_i), \quad \varepsilon_j^{d_j-1-f}\bar a \mapsto (a \varepsilon_j^{f})^*_L; \label{eq: map rho left dual}\\
    \lambda \colon {\Delta(T)_{i, j}} \rightarrow \Hom_{H_j}({A(T)_{i, j}}, H_j), \quad \bar a \varepsilon_i^{d_i-1-f} \mapsto (\varepsilon_i^f a)^*_R. \label{eq: map lambda right dual}
\end{align}

\begin{example}
In \Cref{ex: digon two orbifold points}, we have
    \[
        A(T)_{1, 3} \cong \Bbbk \otimes_\Bbbk \Bbbk[\varepsilon_3]/(\varepsilon_3^2),\quad A(T)_{2,1} \cong \Bbbk[\varepsilon_2]/(\varepsilon_2^2) \otimes_\Bbbk \Bbbk, \quad A(T)_{3, 2} \cong \Bbbk[\varepsilon_3]/(\varepsilon_3^2)\otimes_\Bbbk \Bbbk[\varepsilon_2]/(\varepsilon_2^2).      
    \]
Notice that $c\in {A(T)_{3,2}}$ is a generator (of the bimodule). As an $H_3$-module of rank 2, $A(T)_{3, 2}$ has a basis $\{c, c \varepsilon_2\}$. Through $\rho$, the element $\bar c\in \Delta(T)_{3, 2}$ then defines an element in $\Hom_{H_3}(A(T)_{3, 2}, H_3)$ by $\bar c(c) = 0$ and $\bar c(c \varepsilon_2) = 1$.
\end{example}

\subsection{Flips of triangulations}\label{subsection: flip bimodules}

We now take a closer look at how the bimodules $\{ A(T)_{i, j} \mid i, j \in T, i\neq j \}$ and $\{ A(T')_{i, j} \mid i,j \in T', i\neq j\}$ are related, where $T' = \mu_k(T)$ is the flip of $T$ at some arc $k$. We assume that the arcs in $T$ are indexed by $\{1, \dots, n\}$ and so are those in $T'$.

In this subsection, we aim to specify for any distinct $i,j\in \{1,\dots, n\}$ other than $k$, an $(H_i,H_j)$-bimodule homomorphism
\[
    \theta_k(i,j)\colon A(T')_{i,j} \rightarrow e_i\mathcal{P}(T)e_j,
\]
where $e_v$ denotes the primitive idempotent associated with $v\in \{1, \dots, n\}$.

In view of flips of triangulations, it is straightforward to check that unless both $A(T)_{k,j}\neq 0$ and $A(T)_{i,k}\neq 0$, one always has $A(T)_{i,j} = A(T')_{i,j}$. In this situation, we simply define $\theta_k(i,j)$ to be the natural embedding
\[
    A(T')_{i,j} = A(T)_{i,j} \hookrightarrow e_i\mathcal P(T)e_j.
\]

Now we assume both $A(T)_{k,j}\neq 0$ and $A(T)_{i,k}\neq 0$. In this situation, we construct $\theta_k(i,j)$ by distinguishing two cases according to whether $k$ is pending.

The arc $k$ is confined in a unique (generalized) quadrilateral $S(k)$. The arcs $i$ and $j$ must appear as two sides of $S(k)$. If $k$ is ordinary, then $S(k)$ consists of two regular triangles whereas if $k$ is pending, $S(k)$ contains a regular and a singular triangle. In any case, the part of $\mathcal P(T)$ and that of $\mathcal P(T')$ outside $S(k)$ are identical. Inside $S(k)$, we identify in the obvious way that
\[
    \Delta (T)_{k, j} = A(T')_{j, k} \quad \text{and} \quad \Delta (T)_{i, k} = {A(T')_{k, i}}.
\]

\emph{Case 1.} Suppose that $k$ is pending. The flip at $k$ occurs inside $S(k)$ and is depicted below.
\begin{center}
    \begin{tikzpicture}
        \filldraw[fill=gray!20, thick](0,0) circle (1);
        \filldraw[black] (0, 1) circle (1.5pt);
        \filldraw[black] (0, -1) circle (1.5pt);
        \draw[] (0, 0) node[]{$\star$};
        \draw[thick] plot [smooth cycle] coordinates {(0,-1) (-0.2, 0) (0, 0.2) (0.2, 0)};
        \draw[] (0, 0.5) node[]{\small $k$};
        \draw[] (-0.8, 0) node[]{\small $j$};
        \draw[] (0.8, 0) node[]{\small $i$};
        \draw[->] (1.5, 0) -- (2.5, 0) node[midway, above]{$\mu_k$};

        \begin{scope}[shift={(4,0)}]
        \filldraw[fill=gray!20, thick](0,0) circle (1);
        \filldraw[black] (0, 1) circle (1.5pt);
        \filldraw[black] (0, -1) circle (1.5pt);
        \draw[] (0, 0) node[]{$\star$};
        \draw[thick] plot [smooth cycle] coordinates {(0, 1) (-0.2, 0) (0, -0.2) (0.2, 0)};
        \draw[] (0, -0.5) node[]{\small $k$};
        \draw[] (-0.8, 0) node[]{\small $j$};
        \draw[] (0.8, 0) node[]{\small $i$};
        \end{scope}
    \end{tikzpicture}
\end{center}
The associated local quivers (with possible loops omitted at $j$ or $i$) are
\begin{center}
\begin{tikzpicture}[scale=0.8]
    \begin{scope}[shift={(0,-0.5)}]
        \node[] (k) at (0, 0) {$k$};
        \node[] (j) at (-1, {sqrt(3)}) {$j$};
        \node[] (i) at (1, {sqrt(3)}) {$i$};

        \draw[->] (k) to node[right]{\scriptsize $b$} (i);
        \draw[->] (i) to node[above]{\scriptsize $c$} (j);
        \draw[->] (j) to node[left]{\scriptsize $a$} (k);

        \draw[->] (k) to[out=240, in=300, looseness=6] node[below]{\scriptsize $\varepsilon_k$} (k);
    \end{scope}

    \draw[->] (2,0) to node[above]{$\mu_k$} (3,0);

    \begin{scope}[shift={(5,0.5)}]
        \node[] (k) at (0, 0) {$k$};
        \node[] (j) at (-1, -{sqrt(3)}) {$j$};
        \node[] (i) at (1, -{sqrt(3)}) {$i$};

        \draw[->] (i) to node[right]{\scriptsize $\bar b$} (k);
        \draw[->] (j) to node[below]{\scriptsize $c'$} (i);
        \draw[->] (k) to node[left]{\scriptsize $\bar a$} (j);

        \draw[->] (k) to[out=60, in=120, looseness=6] node[above]{\scriptsize $\varepsilon_k$} (k);
    \end{scope}
\end{tikzpicture}
\end{center}
By examining possible triangles glued to $S(k)$, one sees that the only arrow from $i$ to $j$ in $Q(T')$ is $c'$. Thus $A(T')_{i,j} = H_i\langle c'\rangle H_j$, the free $(H_i,H_j)$-bimodule generated by $c'$. Now define
\begin{equation}\label{eq: def theta k pending}
    \theta_k(i,j)(c') = b\varepsilon_k a.
\end{equation}

\emph{Case 2.} Suppose that $k$ is ordinary. The flip at $k$ inside $S(k)$ is depicted below.

\begin{center}
    \begin{tikzpicture}[scale=0.8]
    \fill[gray!20] (0,0) rectangle (2,2);
    \draw[thick] (0,2) -- (2,0);
    \node at (1.2,1.1) {$k$};

    \fill[gray!20] (5,0) rectangle (7,2);
    \draw[thick] (5,0) -- (7,2);
    \node at (6.2,0.9) {$k$};

    \draw[thick] (0,0) rectangle (2,2); 
    \draw[thick] (5,0) rectangle (7,2); 

    \node[below] at (1,0) {$v$}; 
    \node[left] at (0,1) {$u$}; 
    \node[above] at (1,2) {$q$}; 
    \node[right] at (2,1) {$p$}; 

    \node[below] at (6,0) {$v$}; 
    \node[left] at (5,1) {$u$}; 
    \node[above] at (6,2) {$q$}; 
    \node[right] at (7,1) {$p$}; 

    \draw[->] (3,1) -- (4,1) node[midway, above] {$\mu_k$};
    \end{tikzpicture}
\end{center}
The associated local quivers (with possible loops omitted) are
\begin{center}
    \begin{tikzpicture}[scale=1.2]
        \begin{scope}[shift={(0,0)}]
            \node (k) at (0, 0) {$k$};
            \node (p) at (1, 0) {$p$};
            \node (q) at (0, 1) {$q$};
            \node (u) at (-1, 0) {$u$};
            \node (v) at (0, -1) {$v$};

            \draw[->] (k) to node[below]{\scriptsize $b_1$} (p);
            \draw[->] (p) to node[right]{\scriptsize $c_{(q,p)}$} (q);
            \draw[->] (q) to node[left]{\scriptsize $a_1$} (k);
            \draw[->] (v) to node[right]{\scriptsize $a_2$} (k);
            \draw[->] (k) to node[above]{\scriptsize $b_2$} (u);
            \draw[->] (u) to node[left]{\scriptsize $c_{(v,u)}$} (v);
        \end{scope}

        \draw[->] (2,0) to node[above]{$\mu_k$} (3,0);

        \begin{scope}[shift={(5,0)}]
            \node (k) at (0, 0) {$k$};
            \node (p) at (1, 0) {$p$};
            \node (q) at (0, 1) {$q$};
            \node (u) at (-1, 0) {$u$};
            \node (v) at (0, -1) {$v$};

            \draw[->] (k) to node[right]{\scriptsize $\bar{a}_1$} (q);
            \draw[->] (p) to node[above]{\scriptsize $\bar{b}_1$} (k);
            \draw[->] (k) to node[left]{\scriptsize $\bar{a}_2$} (v);
            \draw[->] (v) to node[right]{\scriptsize $c'_{(p, v)}$} (p);
            \draw[->] (u) to node[below]{\scriptsize $\bar{b}_2$} (k);
            \draw[->] (q) to node[left]{\scriptsize $c'_{(u,q)}$} (u);
        \end{scope}
    \end{tikzpicture}
\end{center}

\begin{remark}\label{rmk: possible arrow from q to i}
    (1) We observe that in either $Q(T)$ or $Q(T')$, there can be no arrow from one side of $S(k)$ to the next side in clockwise order. Indeed, such an arrow, say from $u$ to $q$, would necessarily come from an additional triangle containing the sides $u$ and $q$ such that gluing this triangle to $S(k)$ would make the upper left marked point as an interior marked point, a situation excluded throughout. However, arrows in the counterclockwise direction may occur, for example from $q$ to $u$. (2) Notice also that $u$ and $p$ may coincide or that $v$ and $q$ may coincide. 
\end{remark}

The arcs $i$ and $j$ are chosen from $\{u, v, p, q\}$. By the assumption of this case, $(i, j)$ can be $(u, q)$, $(u, v)$, $(p, q)$, $(p, v)$.

Let $(i, j) = (u, q)$. By the analysis in \Cref{rmk: possible arrow from q to i}, we can express 
\[
    A(T')_{u,q} = H_u\langle c'_{(u,q)} \rangle H_q \oplus B,
\]
where $B$ is an $(H_u, H_q)$-bimodule arising from the part of $Q(T)$ (equivalently $Q(T')$) lying outside $S(k)$. In particular, $B$ is also a direct summand of $A(T)_{u,q}$. Now we define $\theta_k(u,q)$ so that
\[
    \theta_k(u, q)(c'_{(u, q)}) = b_2a_1 \quad \text{and} \quad \theta_k(u, q)\vert_B = \mathrm{id}.
\]

Let $(i, j) = (u, v)$. When $v = q$, then $\theta_k(u,v) = \theta_k(u, q)$. If $v\neq q$, according to \Cref{rmk: possible arrow from q to i}, we have $A(T')_{(u,v)} = 0$ and thus $\theta_k(u,v) = 0$.

For $(i, j) = (p, v)$, the map $\theta_k(p, v)$ is constructed in the exact same way as $\theta_k(u, q)$. In particular,
\begin{equation}\label{eq: def theta ordinary}
    \theta_k(p, v)(c'_{(p,v)}) = b_1a_2.
\end{equation}
Once $\theta_k(p, v)$ is constructed, $\theta_k(p, q)$ is defined in the same way as $\theta_k(u, v)$.

\section{Mutations of representations}\label{section: mutation}
Let $T$ be a triangulation of an orbifold $\mathcal O$ and $\mathcal P(T)$ be the finite-dimensional algebra associated to $T$ as defined in the last section. The goal of this section is to define \emph{mutations} of representations of $\mathcal P(T)$. For a vertex $k$ of $Q(T)$, the \emph{mutation} $\mu_k$ is an operation turning $M\in \modu \mathcal P(T)$ into a representation $\mu_k(M)$ of $\mathcal P(\mu_k(T))$.

\subsection{\texorpdfstring{Representations of $\mathcal P(T)$}{Modules over P(T)}}\label{subsection: reps of P(T)}

By a \emph{representation} of $\mathcal P(T)$, we mean a finitely generated left $\mathcal P(T)$-module; all of such form an abelian category $\modu \mathcal P(T)$. The data of a module $M\in \modu \mathcal P(T)$ can be described as follows. 
\begin{itemize}
    \item For each vertex $i\in Q_0(T)$, there is an associated vector space $M(i) \coloneqq e_iM$ where $e_i$ is the primitive idempotent in $\Bbbk Q$ of vertex $i$.
    \item For each arrow $a \in Q_1(T)$, there is a linear map $M(a) \colon M(t(a)) \rightarrow M(h(a))$.
    \item The linear maps $\{M(a)\mid a \in Q_1(T)\}$ satisfy relations in $R(T)$, that is, for any $i\in Q_0(T)$ pending, $M(\varepsilon_i)^2 = 0$ and for any arrow $a$ following arrow $b$ in a regular triangle, $M(a)M(b) = 0$.
\end{itemize}

There is yet another equivalent way to describe the module structure of $M$ in terms of modulated graphs. This point of view will become dominant in the remaining sections. Recall that we have defined $(H_i, H_j)$-bimodules $A(T)_{i, j}$ in \Cref{subsection: modulated graph}. Now a module $M\in \modu \mathcal P(T)$ is determined by the data including
\begin{itemize}
    \item for each $i\in Q_0(T)$, a finitely generated $H_i$-module $M(i)$, and
    \item for any (ordered) pair of vertices $(i,j)\in Q_0(T)^2$ with $i \neq j$, an $H_i$-homomorphism
    \[
        M(i,j) \colon A(T)_{i, j} \otimes_{H_j} M(j) \rightarrow M(i).
    \]
\end{itemize}
To recover $M(a)$ for $a\in Q_1(T)$, simply let $M(\varepsilon_i)$ act on $M(i)$ by the action of $\varepsilon_i$ through the $H_i$-module structure, and when $h(a)\neq t(a)$ and $m\in M(t(a))$, let
\[
    M(a)(m) = M(h(a),t(a))(a\otimes m).
\]

The following notion will be essential for what follows.

\begin{definition}[Locally free representation]\label{def: loc free reps}
    We say that a representation $M$ of $\mathcal P(T)$ is \emph{locally free} if $M(i)$ is a free $H_i$-module for any $i\in Q_0(T)$, that is, $M(i) = H_i^{\oplus r}$ for some $r\in \mathbb N$.
\end{definition}

\subsection{Dualities between bimodules}

Consider a more general situation. Let $A$ be an $(R, S)$-bimodule such that it is free of finite rank both as a left $R$-module and as a right $S$-module. Let $M\in \modu R$ and $N\in \modu S$. The tensor-hom adjunction gives the following canonical isomorphism
\[
    \Hom_{R} (A \otimes_{S} N, M) \cong \Hom_{S} (N, \Hom_{R}(A, M)).
\]
Through this isomorphism, any $R$-morphism $f \colon A\otimes_S N \rightarrow M$ corresponds to an $S$-morphism
\[
    {^\circ f} \colon N \rightarrow \Hom_R(A, M),\quad {^\circ f}(n)(a) = f(a\otimes n).
\]

Furthermore, since $A$ is a free $R$-module, we have another canonical isomorphism
\[
    \Hom_{R}({A}, {M}) \cong \Hom_{R}({A}, R)\otimes_{R} M.
\]
Then ${^\circ f}$ can be further identified with an element in $\Hom_S(N, \Hom_R(A, R)\otimes_R M)$, which we denote again by ${^\circ f}$. Given an $R$-basis ${a_1, \dots a_\ell}$ for $A$, write the dual basis as ${a_1^*, \dots, a_\ell^*}$. For $n\in N$, we have
\begin{equation}\label{eq: dual map in basis}
    {^\circ f}(n) = \sum_{i=1}^\ell a_i^*\otimes f(a_i\otimes n) \in \Hom_R(A, R)\otimes_R M.
\end{equation}

Recall that for $M\in \modu \mathcal P(T)$ and $(i, j)\in Q_0(T)^2$ with $i\neq j$, there is the structure morphism $M(i, j)\colon A(T)_{i, j}\otimes_{H_j} M(j) \rightarrow M(i)$. It corresponds to the $H_j$-homomorphism
\begin{equation}\label{eq: dual structure map}
    ^\circ M(i, j) \colon M(j) \rightarrow \Hom_{H_i}(A(T)_{i, j}, H_i) \otimes_{H_i} M(i),
\end{equation}
where the codomain can be identified with $\Delta (T)_{i, j} \otimes_{H_i} M(i)$ through $\rho$ in (\ref{eq: map rho left dual}).

\begin{example} \label{ex: duality between bimodule}
    Let $\mathcal O$ be a monogon with two orbifold points. Then any triangulation $T$ induces $Q(T)$ of the form $\begin{tikzcd}
        1 \ar[r, "a"] \ar[loop left, "\varepsilon_1"] & 2 \ar[loop right, "\varepsilon_2"]
    \end{tikzcd}$. 
    Let $M\in \modu \mathcal P(T)$. The map $M(2, 1) \colon  A(T)_{2, 1} \otimes_{H_1} M(1) \rightarrow M(2)$ is given by $a\otimes m \mapsto M(a)(m)\in M(2)$ for $m\in M(1)$, whereas the map $^\circ M(2,1) \colon \, M(1) \rightarrow \Delta (T)_{2, 1} \otimes_{H_2} M(2)$ is given by for any $m\in M(1)$,
    \[
        ^\circ M(2, 1) (m) =  \varepsilon_1\bar a \otimes M(a)(m) + \bar a \otimes M(a\varepsilon_1)(m).
    \]
\end{example}

\subsection{The construction of mutation}\label{subsection: construction of mutation}

We now present the construction of the mutation $\overline M = \mu_k(M) \in \modu \mathcal P(\mu_k(T))$ for $M\in \modu \mathcal P(T)$. Notice that the algebras $H_i$ do not change.

\textbf{Step I. Preparations.}

For $k\in T$, let $T(k, -) \coloneqq \{j\in T \mid b_{kj}>0\}$ and $T(-, k) \coloneqq \{i\in T \mid b_{ik} > 0\}$ be respectively incoming and outgoing vertices in $Q(T)$ with respect to $k$. We form two $H_k$-modules
\begin{equation}\label{eq: def m in and m out}
    M_{\mathrm{in}}(k) = \bigoplus_{j\in T(k, -)} A(T)_{k, j} \otimes_{H_j} M(j)\quad \text{and} \quad M_{\mathrm{out}}(k) = \bigoplus_{i\in T(-, k)} \Delta(T)_{i, k} \otimes_{H_i} M(i),
\end{equation}
and the following diagram of $H_k$-modules
\begin{equation}\label{eq: mutation diagram}
    \begin{tikzcd}[column sep=small]
        & M(k) \ar[rd, "\beta_{k;M}"] & \\
        M_{\mathrm{in}}(k) \ar[ur, "\alpha_{k;M}"] & & M_{\mathrm{out}}(k) \ar[ll, "\gamma_{k;M}"].
    \end{tikzcd}
\end{equation}
We write $\alpha_k = \alpha_{k;M}$, $\beta_k = \beta_{k;M}$, $\gamma_k = \gamma_{k;M}$ whenever $M$ is understood.

The map $\alpha_k$ has components $\alpha_k(k,j) = M(k, j) \colon A(T)_{k, j} \otimes_{H_j} M(j) \rightarrow M(k)$ such that for any $a\in A(T)_{k, j}$ and $m\in M(j)$,
\[
    M(k, j)(a\otimes m) \coloneqq M(a)(m) \in M(k).
\]
The map $\beta_k$ has components $\beta_k(i,k) = {^\circ} M(i, k) \colon M(k) \rightarrow \Delta(T)_{i, k} \otimes_{H_i} M(i)$. Using \eqref{eq: dual map in basis} and \eqref{eq: dual structure map}, we can express for any $m\in M(k)$,
\begin{equation}\label{eq: beta uniform}
    ^\circ M(i, k)(m) = \sum_{b\in E_{i,k}}\sum_{f = 0}^{d_k - 1} \varepsilon_k^f \bar b \otimes M(b \varepsilon_k^{d_k - 1 - f})(m).
\end{equation}
Notice that when $k$ is pending, the set $E_{i,k}$ contains a unique arrow for $i\in T(-,k)$. So more explicitly
\begin{equation}\label{eq: expression of beta}
    ^\circ M(i, k)(m) = \begin{dcases}
        \varepsilon_k \bar b \otimes M(b)(m) + \bar b \otimes M(b\varepsilon_k)(m) \quad &\text{if $k$ is pending,}\\
        \sum_{b\in E_{i,k}}\bar b \otimes M(b)(m) \quad &\text{if $k$ is ordinary},
        \end{dcases}
\end{equation}
where, if $k$ is pending, $b$ denotes the unique element in $E_{i,k}$.

Now we define $\gamma_k$ via the triangulation $T$. Let $i\in T(-, k)$ and $b\in E_{i,k}$. The arrow $b$ arises from a unique triangle formed by $i$, $k$ and a third side. If the third side belongs to the boundary, define $\gamma_k(\bar b\otimes m) = 0$ for any $m\in M(i)$. If it is an arc $j\in T(k, -)$, consider the associated local quiver (with possible loops omitted):
\begin{equation}\label{eq: local quiver gamma}
    \begin{tikzcd}[column sep = small]
        & k \ar[rd, "b"] & \\
        j \ar[ru, "a"] & & i \ar[ll, "c"]
    \end{tikzcd}
\end{equation}
In this situation, define for any $m\in M(i)$,
\begin{equation} \label{eq: def of gamma}
    \gamma_k(\bar b \otimes m) \coloneqq a\otimes M(c)(m) \in {_kA(T)_j\otimes_{H_j} M(j)}.
\end{equation}
Since we have the decomposition 
\begin{equation}\label{eq: express delta tensor M}
    \Delta(T)_{i,k}\otimes_{H_i}M(i) = \bigoplus_{b\in E_{i,k}}H_k(\bar b\otimes M(i)),
\end{equation}
the formula \eqref{eq: def of gamma} determines $\gamma_k$ by running through all $i\in T(-,k)$. For later use, we denote the components of $\gamma_k$ as
\[
    \gamma_k(j, i) \colon \Delta(T)_{i,k}\otimes_{H_i}M(i) \rightarrow A(T)_{k,j}\otimes_{H_j}M(j).
\]

We will show the following two lemmas in \Cref{subsection: alpha beta gamma}.

\begin{lemma}\label{lemma: relations of alpha beta gamma}
    For any $M\in \modu \mathcal P(T)$, the maps $\alpha_k, \beta_k, \gamma_k$ satisfy
    \[
        \alpha_k \gamma_k = 0 \quad \text{and} \quad \gamma_k \beta_k = 0.
    \]
\end{lemma}

\begin{lemma}\label{lemma: image gamma free}
    The $H_k$-morphism $\gamma_k \colon M_{\mathrm{out}}(k) \rightarrow M_{\mathrm{in}}(k)$ is between free $H_k$-modules and the image $\image \gamma_k$ is free.
\end{lemma}

\textbf{Step II. $H_i$-modules and structure morphisms.}

To construct the mutation $\overline M = \mu_k(M)\in \modu \mathcal P(T')$, as pointed out in \Cref{subsection: reps of P(T)} it suffices to determine
\begin{enumerate}
    \item for each $i\in T'$, an $H_i$-module $\overline M(i)$;
    \item for any pair $(i, j)$ with $i\neq j$, an $H_j$-morphism $\overline M(i, j) \colon A(T')_{i, j}\otimes M(j)\rightarrow M(i)$ or equivalently ${^\circ}\overline M(i, j)\colon M(j) \rightarrow \Delta(T')_{i, j} \otimes M_i$.
\end{enumerate}

For (1), we set $\overline M(i) = M(i)$ for any $i\neq k$ and
\[
    \overline M(k) \coloneqq \frac{\ker \gamma_k}{\image \beta_k} \oplus \image \gamma_k \oplus \frac{\ker \alpha_k}{\image \gamma_k}.
\]
Note that the inclusions $\image \beta_k \subset \ker \gamma_k$ and $\image \gamma_k \subset \ker \alpha_k$ are implied by \Cref{lemma: relations of alpha beta gamma}.

For (2), if the pair $(i, j)$ is none of the cases
\begin{itemize}
    \item [(A)] one of $i$ and $j$ is $k$;
    \item [(B)] $j\in T(k, -)$ and $i\in T(-, k)$,
\end{itemize}
we set $\overline M(i, j) \coloneqq M(i, j)$ as in this situation $A(T)_{i,j}=A(T')_{i,j}$ (see \Cref{subsection: flip bimodules}). For case (B), that is if $j\in T(k, -)$ and $i\in T(-, k)$, we define
\begin{equation}\label{eq: def struct morph using theta}
    \overline M(i, j)\colon A(T')_{i, j} \otimes_{H_j} \overline M(j) \rightarrow \overline M(i),\quad a\otimes m\mapsto M(\theta_k(i, j)(a))(m),
\end{equation}
where the $(H_i, H_j)$-morphism
\begin{equation*}\label{eq: theta}
    \theta_k(i, j) \colon A(T')_{i, j} \rightarrow e_i\mathcal{P}(T)e_j
\end{equation*}
is constructed in \Cref{subsection: flip bimodules}.

Now it only remains to define the structure morphisms for case (A), i.e., one of $i$ and $j$ is $k$, the $H_k$-morphisms
\[
    \overline M(k, i) \colon A(T')_{k, i} \otimes_{H_i} M(i) \rightarrow \overline M(k) \quad \text{and} \quad {^\circ \overline M}(j, k) \colon \overline M(k) \rightarrow \Delta (T')_{j, k} \otimes_{H_j} M(j).
\]
Notice that for $T' = \mu_k(T)$ we have identifications
\[
    A(T')_{k, i} = \Delta(T)_{i, k} \quad \text{and} \quad A(T)_{k, j} = \Delta(T')_{j, k}.
\]
We then define $\overline M(k, i)$ and $^\circ \overline M(j, k)$ to respectively be components of
\[
    \bar \alpha_k \colon \overline M_{\mathrm{in}}(k) \coloneqq M_{\mathrm{out}}(k) \rightarrow \overline M(k)\quad \text{and}\quad \bar \beta_k \colon \overline M(k) \rightarrow \overline M_{\mathrm{out}}(k) \coloneqq M_{\mathrm{in}}(k),
\]
which are to be defined as follows.

We borrow the notations from \cite[(10.8), (10.9)]{derksen2008quivers} that whenever there is a pair $U_1\subset U_2$ of modules, denote by $\iota \colon U_1 \rightarrow U_2$ the inclusion and by $\pi \colon U_2 \rightarrow U_2/U_1$ the natural quotient map. We will use abusively the notations $\iota$ and $\pi$ when their domains and codomains are unambiguous.

Since by \Cref{lemma: image gamma free}, the $H_k$-module $\image \gamma_k$ is free (and thus so is $\ker \gamma_k$), we can choose the following \emph{splitting data}:
\begin{enumerate}\label{splitting data}
    \item an $H_k$-morphism $\rho \colon M_{\mathrm{out}}(k) \rightarrow \ker \gamma_k$ such that $\rho \iota = \operatorname{id}_{\ker \gamma_k}$;
    \item an $H_k$-morphism $\sigma \colon \ker \alpha_k / \image \gamma_k \rightarrow \ker \alpha_k$ such that $\pi \sigma = \operatorname{id}_{\ker \alpha_k /\image \gamma_k}$.
\end{enumerate}
Now we define two $H_k$-morphisms
\[
    \bar \alpha_k = \begin{pmatrix}
    \pi \rho \\
    \gamma_k \\
    0
    \end{pmatrix}
    \quad \text{and} \quad \bar \beta_k = \begin{pmatrix}
        0 & \iota & \iota \sigma
    \end{pmatrix}.
\]
At this point the construction of $\mu_k(M) = \overline M$ is complete. The following proposition confirms that $\overline M$ is indeed a $\mathcal P(T')$-representation. The proof will be given in \Cref{subsection: proof define rep mutations}.

\begin{proposition}\label{prop: define rep mutations}
    The $H_i$-modules $\overline M(i)$ for each $i\in T'$ together with the structure morphisms $\overline M(i, j)$ or $^\circ \overline M(i, j)$ for $(i, j)\in (T')^2$ with $i\neq j$ indeed define a $\mathcal P(T')$-representation $\overline M$, that is, the structure morphisms satisfy the defining relations of $\mathcal P(T')$. 
\end{proposition}

\subsection{Proof of \Cref{lemma: relations of alpha beta gamma} and \Cref{lemma: image gamma free}} \label{subsection: alpha beta gamma}

\begin{proof}[Proof of \Cref{lemma: relations of alpha beta gamma}]
    We first prove $\alpha_k \gamma_k = 0$. Let $i\in T(-, k)$ and $b\in E_{i,k}$. Suppose that $b$ arises from a triangle formed by $i$, $k$ and $j\in T(k, -)$; otherwise $\gamma_k(\bar b\otimes m) = 0$ for any $m\in M(i)$. Recall the quiver \eqref{eq: local quiver gamma} with arrows $a,b,c$. Then $\gamma_k(\bar b\otimes m) = a\otimes M(c)(m)$ for any $m\in M(i)$, and thus 
    \[
        \alpha_k \gamma_k(\bar b\otimes m) = \alpha_k(k,j)(a\otimes M(c)(m))=M(a)(M(c)(m))=0.
    \]
    The last equality is due to the relation $a\cdot c=0$ in the algebra $\mathcal{P}(T)$.

    Now we prove $\gamma_k \beta_k = 0$. It suffices to show $\gamma_k(\beta_k(i,k)(m))=0$ for any $i\in T(-,k)$ and any $m\in M(k)$. Recall from \eqref{eq: beta uniform} we have
    \[
        \beta_k(k, i) (m) = \sum_{b\in E_{i,k}}\sum_{t = 0}^{d_k - 1} \varepsilon_k^t \bar b \otimes M(b \varepsilon_k^{d_k - 1 - t})(m).
    \]
    For each $b\in E_{i,k}$, consider again the unique triangle with sides $i$ and $k$ that gives rise to $b$. If the third side belongs to the boundary, then $\gamma_k(\bar b) = 0$. If the third side is $j\in T(k, -)$, then using the quiver \eqref{eq: local quiver gamma} and the formula \eqref{eq: def of gamma}, we compute
    \[
        \gamma_k\left( \varepsilon_k^t\bar b \otimes M(b\varepsilon_k^{d_k-1-t})(m) \right) = \varepsilon_k^t a\otimes M(c)(M(b\varepsilon_k^{d_k-1-t})(m)).
    \]
    It vanishes because $M(c)M(b) = 0$ for any $\mathcal P(T)$-representation $M$. Combining the two cases, we obtain $\gamma_k(\beta_k(i,k)(m))=0$, completing the proof.
\end{proof}

\begin{proof}[Proof of \Cref{lemma: image gamma free}]
    The statement is only non-trivial when $k$ is pending. Being a pending arc, $k$ must be confined in a regular triangle with $i$ and $j$ as the other two sides. In view of the diagram \eqref{eq: local quiver gamma}, we have
    \[
        \gamma_k \colon H_k(\bar b\otimes M(i)) \rightarrow H_k(a\otimes M(j)), \quad \bar b\otimes m \mapsto a\otimes M(c)(m).
    \]
    Therefore the map $\gamma_k$ is between free $H_k$-modules and $\image \gamma_k = H_k(a\otimes M(c)(M(i))$, a free $H_k$-module of rank $\dim_\Bbbk a\otimes M(c)(M(i))$.
\end{proof}

\subsection{Proof of \Cref{prop: define rep mutations}}\label{subsection: proof define rep mutations}

Before proving \Cref{prop: define rep mutations}, we show

\begin{lemma}\label{lemma: ker bar alpha = im beta & ker alpha = im bar beta}
    Let $\overline M = \mu_k(M)$ as constructed in \Cref{subsection: construction of mutation}, where $\bar \alpha_k$ and $\bar \beta_k$ are defined. Then
    \begin{enumerate}
        \item $\ker \bar \alpha_k = \image {\beta}_k$ and $\ker \alpha_k = \image \bar \beta_k$;
        \item $\bar \alpha_k \bar \beta_ k = \gamma_k$.
    \end{enumerate}
\end{lemma}

\begin{proof}
    We first prove part (1). Recall that we have defined
    \[
        \bar \alpha_k = \begin{pmatrix}
            \pi\rho \\
            \gamma_k\\
            0
        \end{pmatrix} \colon M_\mathrm{out}(k) \rightarrow \overline M(k) = \frac{\ker \gamma_k}{\image \beta_k} \oplus \image \gamma_k \oplus \frac{\ker \alpha_k}{\image \gamma_k},
    \]
    where $\rho \colon M_\mathrm{out}(k) \rightarrow \ker \gamma_k$ is part of the splitting data and $\pi$ denotes the natural quotient $\ker \gamma_k \rightarrow \ker \gamma_k/\image \beta_k$. Then we have
    \[
        \ker \bar \alpha_k = \{m\in M_\mathrm{out}(k)\mid \rho(m)\in \image \beta_k\} \cap \ker \gamma_k = \image \beta_k.
    \]

    We have also defined
    \[
        \bar \beta_k = (0, \iota, \iota\sigma) \colon \overline M(k) = \frac{\ker \gamma_k}{\image \beta_k} \oplus \image \gamma_k \oplus \frac{\ker \alpha_k}{\image \gamma_k} \rightarrow M_\mathrm{in}(k),
    \]
    where $\sigma \colon \ker \alpha_k / \image \gamma_k \rightarrow \ker \alpha_k$ is part of the splitting data and $\iota$ denotes the inclusion $\ker \alpha_k \subset M_\mathrm{in}(k)$. It is clear that $\image \bar \beta_k \subset \ker \alpha_k$. For any $x\in \ker \alpha_k$, using $\sigma$, we can decompose $x$ into the sum of $x-\sigma\pi(x)\in \image \gamma_k$ and $\sigma\pi(x)\in \ker\alpha_k/\image \gamma_k$. So we can express $x = \bar\beta_k((0, x-\sigma\pi(x), \sigma\pi(x)))$, which implies $\ker \alpha_k \subset \image \bar \beta_k$, hence $\ker \alpha_k = \image \bar \beta_k$.

    Part (2) follows directly from the definitions of $\bar \alpha_k$ and $\bar \beta_k$.
\end{proof}

\begin{proof}[Proof of \Cref{prop: define rep mutations}]\label{proof: define rep mutations}
    We show that $\overline M$, as constructed in \Cref{subsection: construction of mutation}, is indeed a representation of $\mathcal P(T')$. According to \Cref{subsection: reps of P(T)}, it suffices to check $\overline M(a)\overline M(b)=0$ for any arrow $a$ following another arrow $b$ in a regular triangle in $T'$. Notice that $\overline M(a)=M(a)$ for any arrow $a$ lying outside the quadrilateral $S(k)$ where $k$ is a diagonal (see \Cref{subsection: flip bimodules}). Therefore, it only remains to verify the relations for arrows in $Q_1(T')$ within $S(k)$.

    \textit{Case 1.} We start with the case where $k$ is pending.
    \begin{center}
        \begin{tikzpicture}
            \filldraw[fill=gray!20, thick](0,0) circle (1);
            \filldraw[black] (0, 1) circle (1.5pt);
            \filldraw[black] (0, -1) circle (1.5pt);
            \draw[] (0, 0) node[]{$\star$};
            \draw[thick] plot [smooth cycle] coordinates {(0,-1) (-0.2, 0) (0, 0.2) (0.2, 0)};
            \draw[] (0, 0.5) node[]{\small $k$};
            \draw[] (-0.8, 0) node[]{\small $j$};
            \draw[] (0.8, 0) node[]{\small $i$};
            \draw[->] (1.5, 0) -- (2.5, 0) node[midway, above]{$\mu_k$};
        \end{tikzpicture}
        \quad
        \begin{tikzpicture}
            \filldraw[fill=gray!20, thick](0,0) circle (1);
            \filldraw[black] (0, 1) circle (1.5pt);
            \filldraw[black] (0, -1) circle (1.5pt);
            \draw[] (0, 0) node[]{$\star$};
            \draw[thick] plot [smooth cycle] coordinates {(0, 1) (-0.2, 0) (0, -0.2) (0.2, 0)};
            \draw[] (0, -0.5) node[]{\small $k$};
            \draw[] (-0.8, 0) node[]{\small $j$};
            \draw[] (0.8, 0) node[]{\small $i$};
        \end{tikzpicture}
    \end{center}
    The corresponding quivers of $T$ and $T' = \mu_k(T)$ are (partly) given below:
    \begin{center}
        \begin{tikzpicture}[scale=0.8]
            \begin{scope}[shift={(0,-0.5)}]
                \node[] (k) at (0, 0) {$k$};
                \node[] (j) at (-1, {sqrt(3)}) {$j$};
                \node[] (i) at (1, {sqrt(3)}) {$i$};
    
                \draw[->] (k) to node[right]{\scriptsize $b$} (i);
                \draw[->] (i) to node[above]{\scriptsize $c$} (j);
                \draw[->] (j) to node[left]{\scriptsize $a$} (k);
    
                \draw[->] (k) to[out=240, in=300, looseness=6] node[below]{\scriptsize $\varepsilon_k$} (k);
            \end{scope}
    
            \draw[->] (2,0) to node[above]{$\mu_k$} (3,0);
    
            \begin{scope}[shift={(5,0.5)}]
                \node[] (k) at (0, 0) {$k$};
                \node[] (j) at (-1, -{sqrt(3)}) {$j$};
                \node[] (i) at (1, -{sqrt(3)}) {$i$};
    
                \draw[->] (i) to node[right]{\scriptsize $\bar b$} (k);
                \draw[->] (j) to node[below]{\scriptsize $c'$} (i);
                \draw[->] (k) to node[left]{\scriptsize $\bar a$} (j);
    
                \draw[->] (k) to[out=60, in=120, looseness=6] node[above]{\scriptsize $\varepsilon_k$} (k);
            \end{scope}
        \end{tikzpicture}
    \end{center}
    We now describe the three $\Bbbk$-linear maps $\overline M(\bar a)$, $\overline M(\bar b)$ and $\overline M(c')$.

Recall $\bar \beta_k \colon \overline M(k) \rightarrow A(T)_{k, j} \otimes_{H_j} M(j)$ defined in \Cref{subsection: construction of mutation}. Now with the isomorphism $ A(T')_{j, k} = \Delta(T)_{k, j} \cong \Hom_{H_j}(A(T)_{k, j}, H_j)$ from (\ref{eq: map lambda right dual}), we have
\[
    \overline M(\bar a)(m) = \langle \bar a, \bar \beta_k(m) \rangle \quad \text{for any $m\in \overline M(k)$},
\]
where $\langle -, - \rangle$ denotes the natural pairing $\Hom_{H_j}(A(T)_{k, j}, H_j) \times A(T)_{k, j} \otimes_{H_j} M(j) \rightarrow M(j)$. We have also defined $\bar \alpha_k \colon \Delta(T)_{i, k} \otimes _{H_i}M(i) \rightarrow \overline M(k)$ in \Cref{subsection: construction of mutation}. The map $\overline M(\bar b)$ is given by
\begin{equation}\label{eq: overline M bar b}
    \overline M(\bar b)(m) = \bar \alpha_k (\bar b \otimes m) \quad \text{for any $m\in M(i)$}.
\end{equation}
The map $\overline M(c')$ is defined through $\theta_k(i, j)$ (see \eqref{eq: def struct morph using theta} and \eqref{eq: def theta k pending}). We have
\[
    \overline M(c')(m) = M(b\varepsilon_ka)(m) \quad \text{for any $m\in M(j)$}.
\]

We check (1) $\overline M(c')\overline M(\bar a) = 0$; (2) $\overline M(\bar b) \overline M(c') = 0$; (3) $\overline M(\bar a)\overline M(\bar b) = 0$.

(1) Let $m\in \overline M(k)$. Then $\bar \beta_k(m)$ belongs to the kernel of $\alpha_k(k, j)$ by \Cref{lemma: ker bar alpha = im beta & ker alpha = im bar beta}. Thus if writing $\bar \beta_k(m) = a\otimes m_1 + \varepsilon_ka\otimes m_2$ for $m_1, m_2 \in M(j)$, then $M(a)(m_1) + M(\varepsilon_k a)(m_2) = 0$. Now $\overline M(\bar a)(m) = \langle \bar a, a\otimes m_1 + \varepsilon_k a \otimes m_2 \rangle = m_2$. Then we have
\[
    \overline M(c')(\overline M(\bar a)(m)) = M(b)(M(\varepsilon_ka)(m_2)) = M(b)(M(a)(-m_1)) = 0
\]
since $M(b)M(a) = 0$.

(2) Let $m\in M(j)$. We only need to show that $\bar b \otimes M(b\varepsilon_ka)(m)$ is in $\ker \bar \alpha_k$, which coincides with $\image \beta_k$ by \Cref{lemma: ker bar alpha = im beta & ker alpha = im bar beta}. In fact, by \eqref{eq: expression of beta} where $k$ is pending, we have
\[
    \beta_k( M(a)(m)) = \bar b \otimes M(b\varepsilon_ka)(m) + \varepsilon_k \bar b \otimes M(ba)(m) = \bar b \otimes M(b\varepsilon_ka)(m).
\]

(3) Let $m\in M(i)$. Then $\overline M(\bar b)(m) = \bar \alpha_k(\bar b \otimes m)$. Now $\bar \beta_k\bar \alpha_k(\bar b \otimes m) = \gamma_k(\bar b \otimes m)$ by part (2) of \Cref{lemma: ker bar alpha = im beta & ker alpha = im bar beta}, which further equals $a\otimes M(c)(m)$ by \eqref{eq: def of gamma}. Thus we have
\[
    \overline M(\bar a)(\overline M(\bar b)(m)) = \langle \bar a, a\otimes M(c)(m) \rangle = 0.
\]

\textit{Case 2.} Now suppose that $k\in T$ is ordinary. The flip at $k$ is given below:
\begin{center}
\begin{tikzpicture}[scale=0.8]
    \fill[gray!20] (0,0) rectangle (2,2);
    \draw[thick] (0,2) -- (2,0);
    \node at (1.2,1.1) {$k$};

    \fill[gray!20] (5,0) rectangle (7,2);
    \draw[thick] (5,0) -- (7,2);
    \node at (6.2,0.9) {$k$};

    \draw[thick] (0,0) rectangle (2,2); 
    \draw[thick] (5,0) rectangle (7,2); 

    \node[below] at (1,0) {$v$}; 
    \node[left] at (0,1) {$u$}; 
    \node[above] at (1,2) {$q$}; 
    \node[right] at (2,1) {$p$}; 

    \node[below] at (6,0) {$v$}; 
    \node[left] at (5,1) {$u$}; 
    \node[above] at (6,2) {$q$}; 
    \node[right] at (7,1) {$p$}; 

    \draw[->] (3,1) -- (4,1) node[midway, above] {$\mu_k$};
    \end{tikzpicture}
\end{center}
The corresponding quivers (with possible loops omitted) are
\begin{center}
    \begin{tikzpicture}[scale=1.2]
        \begin{scope}[shift={(0,0)}]
            \node (k) at (0, 0) {$k$};
            \node (p) at (1, 0) {$p$};
            \node (q) at (0, 1) {$q$};
            \node (u) at (-1, 0) {$u$};
            \node (v) at (0, -1) {$v$};

            \draw[->] (k) to node[below]{\scriptsize $b_1$} (p);
            \draw[->] (p) to node[right]{\scriptsize $c_{(q,p)}$} (q);
            \draw[->] (q) to node[left]{\scriptsize $a_1$} (k);
            \draw[->] (v) to node[right]{\scriptsize $a_2$} (k);
            \draw[->] (k) to node[above]{\scriptsize $b_2$} (u);
            \draw[->] (u) to node[left]{\scriptsize $c_{(v,u)}$} (v);
        \end{scope}

        \draw[->] (2,0) to node[above]{$\mu_k$} (3,0);

        \begin{scope}[shift={(5,0)}]
            \node (k) at (0, 0) {$k$};
            \node (p) at (1, 0) {$p$};
            \node (q) at (0, 1) {$q$};
            \node (u) at (-1, 0) {$u$};
            \node (v) at (0, -1) {$v$};

            \draw[->] (k) to node[right]{\scriptsize $\bar{a}_1$} (q);
            \draw[->] (p) to node[above]{\scriptsize $\bar{b}_1$} (k);
            \draw[->] (k) to node[left]{\scriptsize $\bar{a}_2$} (v);
            \draw[->] (v) to node[right]{\scriptsize $c'_{(p, v)}$} (p);
            \draw[->] (u) to node[below]{\scriptsize $\bar{b}_2$} (k);
            \draw[->] (q) to node[left]{\scriptsize $c'_{(u,q)}$} (u);
        \end{scope}
    \end{tikzpicture}
\end{center}
We need to show the maps $\overline M(a)$ ($a$ an arrow in the right quiver above) satisfy relations in $\mathcal P(T')$. We prove the three relations
\[
    (1)\ \overline M(\bar a_2) \overline M(\bar b_1) = 0;\ (2)\ \overline M(c'_{(p, v)}) \overline M(\bar a_2) = 0;\ (3)\ \overline M(\bar b_1)  \overline M(c'_{(p, v)}) = 0
\]
in the lower right triangle formed by $v, k, p$. The rest three relations can be proven similarly.

We have $\bar \beta_k(v,k)\colon M(k)\rightarrow A(T)_{k,v}\otimes_{H_v}M(v)$. (Note that $q$ and $v$ might coincide.) Using the isomorphism $ A(T')_{v, k} = \Delta(T)_{k, v} \cong \Hom_{H_j}(A(T)_{k, v}, H_v)$ from (\ref{eq: map lambda right dual}), we have
\[
    \overline M(\bar a_2)(m) = \langle \bar a_2, \bar \beta_k(v,k)(m) \rangle \quad \text{for any $m\in \overline M(k)$},
\]
where $\langle -, - \rangle$ denotes the natural pairing $\Hom_{H_v}(A(T)_{k, v}, H_v) \times A(T)_{k, v} \otimes_{H_v} M(v) \rightarrow M(v)$. We also have $\bar \alpha_k(k, p)\colon \Delta(T)_{p,k}\otimes_{H_p} M(p)\rightarrow M(k)$. Then $\overline M(\bar b_1)(m) = \bar \alpha_k(k, p)(\bar b_1\otimes m)$ for any $m\in M(p)$.

The map $\overline M(c'_{(p, v)})$ is given through $\theta_k(p, v)$: for any $m\in M(v)$,
\begin{equation}\label{eq: map of created arrow}
    \overline M(c'_{(p, v)})(m) = M(b_1a_2)(m);
\end{equation}
see \eqref{eq: def struct morph using theta} and \eqref{eq: def theta ordinary}.

Now we check the relations.

(1) Let $m\in M(p)$. We have $\overline{M}(\bar a_2)(\overline{M}(\bar b_1)(m)) = \langle \bar a_2, \bar \beta_k(v, k)\bar \alpha_k(k, p)(\bar b_1\otimes m)\rangle$. By part(2) of \Cref{lemma: ker bar alpha = im beta & ker alpha = im bar beta}, this equals $\langle \bar a_2, \gamma_k(v,p)(\bar b_1\otimes m) \rangle$. However by the definition of $\gamma$ (see \Cref{subsection: construction of mutation}), we have $\gamma_k(v, p)(\bar b_1\otimes m) = 0$, thus $\overline{M}(\bar a_2)(\overline{M}(\bar b_1)(m))=0$.

(2) Let $m\in \overline M(k)$. We want to show
\[
    \overline M(c'_{(p, v)})(\overline M(\bar a_2)(m)) = M(b_1a_2)(\langle\bar a_2, \bar \beta_k(v,k)(m) \rangle) = 0.
\]
We express $\bar\beta_k(m) = a_1\otimes m_1 + a_2\otimes m_2$ for $m_1\in M(q)$ and $m_2\in M(v)$. Notice whether $q=v$ or not, this expression is always valid. Since $\image \beta_k\subset \ker \alpha_k$ by part (1) of \Cref{lemma: ker bar alpha = im beta & ker alpha = im bar beta}, we have $M(a_1)(m_1)+M(a_2)(m_2)=0$. Then we compute
\[
    M(b_1a_2)(\langle\bar a_2, \bar \beta_k(v,k)(m) \rangle) = M(b_1a_2)(m_2) = M(b_1)(-M(a_1)(m_1)) = 0,
\]
where the last step is because $M(b_1)M(a_1)=0$.

(3) Let $m\in M(v)$. We want to show 
\[
    \overline{M}(\bar b_1)(\overline{M}(c'_{(p,v})(m)) = \bar \alpha_k(k,p)(\bar b_1\otimes M(b_1a_2)(m)) = 0.
\]
By part (1) of \Cref{lemma: ker bar alpha = im beta & ker alpha = im bar beta}, it is equivalent to $\bar b_1\otimes M(b_1a_2)(m)\in \image \beta_k$. Using the ordinary case of \eqref{eq: expression of beta}, we have
\[
    \beta_k(M(a_2)(m)) = \bar b_1 \otimes M(b_1)(M(a_2)(m)) + \bar b_2 \otimes M(b_2)(M(a_2)(m)).
\]
The second summand vanishes because $M(b_2)M(a_2)=0$ while the first summand equals $\bar b_2\otimes M(b_1a_2)(m)$, as desired.
\end{proof}

\subsection{Independence of Splitting data}
In this section we show that our construction of $\mathcal P(T)$-module $\overline M$ is independent of the splitting data in \Cref{splitting data}.

\begin{lemma}\label{lemma: independence of split}
    The isomorphism class of $\overline M$ does not depend on the choices of the maps $\rho$ and $\sigma$.
\end{lemma}

\begin{proof}
    Suppose that we have another set of choices $\rho'$ and $\sigma'$ to define $\overline M'$. The two maps $\rho, \rho'\colon M_{\mathrm{out}}(k) \rightarrow \ker \gamma_k$ differ by $\xi\gamma_k$, i.e. $\rho' = \rho + \xi\gamma_k$ for some $\xi\colon \image \gamma_k \rightarrow \ker \gamma_k$. The two maps $\sigma, \sigma' \colon \ker \alpha_k/\image \gamma_k \rightarrow \ker \alpha_k$ differ by $\eta\colon \ker \alpha_k/\image \gamma_k \rightarrow \image \gamma_k$, i.e. $\sigma' = \sigma + \eta$.

    Let the $H_i$-isomorphisms $\delta_i \colon \overline M'(i) \rightarrow \overline M(i)$ be identity for $i\neq k$ and $\delta_k\colon \overline M'(k) \rightarrow \overline M(k)$ be
    \[
        \begin{bmatrix}
            \mathrm{id} & \pi\xi& \\
            & \mathrm{id} & -\eta \\
            & & \mathrm{id}
        \end{bmatrix} \in \operatorname{End}_{H_k}\left( \frac{\ker \gamma_k}{\image \beta_k} \oplus \image \gamma_k \oplus \frac{\ker \alpha_k}{\image \gamma_k} \right).
    \]
It is easy to check that the tuple $\delta \coloneqq (\delta_i)_i$ defines a $\mathcal P(\mu_k(T))$-isomorphism $\delta \colon \overline M' \rightarrow \overline M$.
\end{proof}

\subsection{Double mutation} \label{subsection: mutation involution}
We now turn to the double mutation $\mu_k^2(M)$. Denote $\mu_k(M)$ as before by $\overline M$. We have accordingly the maps
\[
    \bar \alpha_k \colon \overline M_{\mathrm{in}}(k) \rightarrow \overline M(k),\quad \bar \beta_{k} \colon \overline M(k) \rightarrow \overline M_{\mathrm{out}}(k),\quad \bar \gamma_k \colon \overline M_{\mathrm{out}}(k) \rightarrow \overline M_{\mathrm{in}} (k),
\]
where $\bar \gamma_k$ is defined through the $\mathcal P(T')$-module structure of $\overline M$, exactly as $\gamma_k$ is defined for $M$.

\begin{lemma}\label{lemma: bar gamma = beta alpha}
    $\bar \gamma_k = \beta_k\alpha_k$.
\end{lemma}

\begin{proof}
    We refer to the proof of \Cref{prop: define rep mutations} for the two cases where $k$ can be pending or ordinary. We use the formula \eqref{eq: def of gamma} to compute $\bar \gamma_k$ in both cases. 
    
    If $k$ is pending, the map $\bar \gamma_k$ is given by for any $m\in M(j)$,
    \[
        \bar \gamma_k(a\otimes m) = \bar b \otimes \overline M(c')(m) = \bar b \otimes M(b\varepsilon_k a)(m) = \beta_k \alpha_k (a\otimes m).
    \]

    If $k$ is ordinary, let $m\in M(v)$. Then by \eqref{eq: map of created arrow},
    \[
        \bar \gamma_k(a_2\otimes m) = \bar b_1 \otimes \overline M(c'_{(p, v)})(m) = \bar b_1\otimes M(b_1a_2)(m) = \beta_k\alpha_k(a_2\otimes m).
    \]
    Similarly, for any $m\in M(q)$, we have $\bar \gamma_k(a_1\otimes m) = \beta_k\alpha_k(a_1\otimes m)$.

    Combining the two cases, we conclude $\bar \gamma_k = \beta_k \alpha_k$.
\end{proof}

We omit the subscript $k$ in the maps $\alpha, \beta, \gamma$ in the following proposition and the proof.

\begin{proposition}\label{prop: double mutation}
Let $M$ be locally free. Assume that the $H_k$-modules $\ker \beta$ and $\image \alpha$ are both free. Then the $H_k$-submodule $\ker \beta \cap \image \alpha$ of $M(k)$ is free. Furthermore, we have
\[
    M \cong \mu_k^2(M) \oplus \dfrac{\ker \beta}{\ker \beta \cap \image \alpha}, 
\]
where the second summand is regarded as a $\mathcal P(T)$-module supported at vertex $k$.
\end{proposition}

\begin{proof}
    We first show that under the assumptions, the $H_k$-module $\ker \beta \cap \image \alpha$ is free. In fact, there is a natural short exact sequence
    \[
        0 \rightarrow \ker \beta \cap \image \alpha \rightarrow \image \alpha \rightarrow \image \beta \alpha \rightarrow 0.
    \]
    Notice that $\image \beta\alpha = \image \bar \gamma$ by \Cref{lemma: bar gamma = beta alpha}. The latter is free over $H_k$ by \Cref{lemma: image gamma free}. Hence $\ker \beta \cap \image \alpha$ is free given that $\image \alpha$ is free.

    We now compute the module structure of $\overline {\overline{M}} \coloneqq \mu_k^2(M)$. In fact, by construction the only part of $\overline {\overline M}$ possibly different from $M$ is within the diagram:
    \[
        \begin{tikzcd}[column sep = small]
            & \overline {\overline M}(k) = \dfrac{\ker \bar \alpha}{\image \bar \gamma}\oplus \image \bar \gamma \oplus \dfrac{\ker \bar \gamma}{\image \bar \beta} \ar[dr, "\bar {\bar \beta}"]& \\
            M_{\mathrm{in}}(k) \ar[ur, "\bar {\bar \alpha}"] & & M_{\mathrm{out}}(k) \ar[ll, "\bar{\bar \gamma}"]
        \end{tikzcd}
    \]
    First we notice that $\bar {\bar \gamma} = \bar \beta \bar \alpha = \gamma$ by \Cref{lemma: bar gamma = beta alpha} and part (2) of \Cref{lemma: ker bar alpha = im beta & ker alpha = im bar beta}. The definitions of $\bar{\bar \alpha}$ and $\bar{\bar \beta}$ depend on a choice of splitting data $\bar \rho \colon M_{\mathrm{in}}(k) \rightarrow \ker\bar \gamma$ and $\bar \sigma \colon \ker \bar \alpha /\image \bar \gamma \rightarrow \ker \bar \alpha$, which by \Cref{lemma: bar gamma = beta alpha} and \Cref{lemma: ker bar alpha = im beta & ker alpha = im bar beta} are 
    \[
        \bar \rho \colon M_{\mathrm{in}}(k) \rightarrow \ker \beta \alpha \quad \text{and} \quad \bar \sigma \colon \image \beta/ \image \beta\alpha \rightarrow \image \beta.
    \]
    Such splittings always exist because $\ker \beta\alpha$ and $\image \beta\alpha$ are free over $H_k$. Notice that the surjective map $\alpha \colon \ker \beta\alpha \rightarrow \image \alpha \cap \ker \beta$ induces an isomorphism $\image \alpha \cap \ker \beta = \ker \beta\alpha/\ker \alpha (= \ker \bar \gamma/\image \bar \beta)$. Now the module structure of $\overline {\overline M}$ can be written as
    \begin{equation}\label{eq: module structure twice mutation}
        \begin{tikzcd}[column sep = huge]
            M_{\mathrm{in}}(k) \ar[r, "{(0,\, \beta\alpha,\, \alpha \bar \rho)^\intercal}"] & \dfrac{\image \beta}{\image \beta\alpha} \oplus \image \beta\alpha \oplus (\image \alpha \cap \ker \beta) \ar[r, "{(\iota \bar \sigma,\, \iota,\, 0)}"] & M_{\mathrm{out}}(k).
        \end{tikzcd}
    \end{equation}

    The retraction $\bar \rho \colon M_{\mathrm{in}}(k) \rightarrow \ker \beta\alpha$ (quotienting out $\ker \alpha$) induces $\tilde \rho \colon \image \alpha \rightarrow \image \alpha \cap \ker \beta$, which satisfies $\tilde \rho \iota = \mathrm{id}_{\image \alpha \cap \ker \beta}$ and $\tilde \rho \alpha = \alpha \bar \rho \colon M_{\mathrm{in}}(k) \rightarrow \ker \beta\alpha$. We can extend $\tilde \rho$ to a retraction $\rho_1 \colon M(k) \rightarrow \ker \beta$ (which exists given that $\ker \beta$ is free and $M(k)$ is free), i.e., such that $\rho_1\iota = \mathrm{id}_{\ker \beta}$ and $\iota\bar\rho = \rho_1\iota$. Summarizing, we have the following commutative diagram
    \[
        \begin{tikzcd}
            M_{\mathrm{in}}(k) \ar[d, "\bar \rho "]\ar[r, "\alpha"] & \image \alpha \ar[d, "\tilde \rho"] \ar[r, hook, "\iota"] & M(k) \ar[d, "\rho_1"] \\
            \ker \beta \alpha \ar[r, "\alpha"] & \image \alpha \cap \ker \beta \ar[r, hook, "\iota"] & \ker \beta,
        \end{tikzcd}
    \]
    where $\tilde \rho$, $\bar \rho$ and $\rho_1$ are all retractions of natural inclusions.
    
    We choose another splitting $\ker \beta = (\image \alpha \cap \ker \beta) \oplus \ker \beta/(\image \alpha \cap \ker \beta)$. Together with the splittings determined by $\rho_1$ and $\bar \sigma$, the module structure of $M$ involving $M_{\mathrm{in}}(k) \xrightarrow{\alpha} M(k) \xrightarrow{\beta} M_{\mathrm{out}}(k)$ is shown as follows isomorphic to    
    \begin{equation}\label{eq: module structure m}
        \begin{tikzcd}[column sep = huge, row sep = tiny]
                \quad M_{\mathrm{in}}(k) \ar[r, "{(\beta\alpha,\, \rho_1\alpha)^\intercal}"] & \image \beta \oplus \ker \beta \ar[r, "{(\iota,\, 0)}"] & M_{\mathrm{out}}(k)\\
                \cong M_{\mathrm{in}}(k) \ar[r, "{(0,\, \beta\alpha,\, \rho_1\alpha)^\intercal}"] & \dfrac{\image \beta}{\image \beta\alpha} \oplus \image \beta\alpha \oplus \ker \beta \ar[r, "{(\iota\bar \sigma,\, \iota,\, 0)}"] & M_{\mathrm{out}}(k)\\
                \cong M_{\mathrm{in}}(k) \ar[r, "{(0,\,\beta\alpha,\, \alpha\bar \rho,\, 0)^\intercal}"] & \dfrac{\image \beta}{\image \beta\alpha} \oplus \image \beta\alpha \oplus (\image \alpha \cap \ker \beta) \oplus \dfrac{\ker \beta}{\image \alpha \cap \ker \beta} \ar[r, "{(\iota\bar\sigma,\, \iota,\, 0,\, 0)}"] & M_{\mathrm{out}}(k).
        \end{tikzcd}
    \end{equation}
    Comparing the last line of (\ref{eq: module structure m}) with (\ref{eq: module structure twice mutation}), we have the desired isomorphism
    \[
        M \cong \overline{\overline M} \oplus \dfrac{\ker \beta}{\image \alpha \cap \ker \beta}.
    \]

\end{proof}

\subsection{Mutations with decorations}\label{subsection: mutation decorations}
In this section we introduce the notion of decorated representations and extend the notion of mutation to this setting.

\begin{definition}\label{def: decorated representation}
A \emph{decorated $\mathcal P(T)$-module} $\mathcal M = (M, V)$ is a pair of a $\mathcal P(T)$-module $M$ and a tuple $V = (V_i)$ of $H_i$-modules $V_i$. It is called \emph{locally free} if $M$ is locally free and for any $i\in I$, $V_i$ is free over $H_i$.
\end{definition}

\begin{definition}\label{def: decorated mutation}
The \emph{decorated mutation} $\mu_k(\mathcal M)$ of $\mathcal M$ is defined to be the decorated $\mathcal P(\mu_k(T))$-module $(\mu_k(M)\oplus V_k, (\overline V_i))$, where $\overline V_i \coloneqq V_i$ for $i\neq k$, $\overline V_k \coloneqq \ker \beta_k/(\ker \beta_k \cap \image \alpha_k)$, and $V_k$ in the direct sum $\mu_k(M)\oplus V_k$ is regarded as a $\mathcal P(\mu_k(T))$-module.
\end{definition}

Denote by $E_k$ the $\mathcal P(T)$-module $M$ such that $M = M(k) = H_k$ and thus all structure morphisms vanish. From the above definition one sees immediately that $\mu_k(0, H_k) = (E_k, 0)$ and $\mu_k(E_k, 0) = (0, H_k)$.

The following is a direct corollary of \Cref{prop: double mutation}.

\begin{corollary}\label{cor: decorated mutation involution}
For a locally free decorated $\mathcal P(T)$-module $\mathcal M = (M, V)$ such that $\ker \beta_{k;M}$ and $\image \alpha_{k;M}$ are free, the decorated mutation $\mu_k$ is involutive, i.e., $\mu_k^2(\mathcal M) \cong \mathcal M$.
\end{corollary}

\begin{proof}
By definition we have $\mu_k(\mathcal M) = (\mu_k(M)\oplus V_k, (\overline V_i))$ and 
\[
    \mu_k^2(\mathcal M) = (\mu^2_k(M) \oplus \mu_k(V_k) \oplus \overline V_k, (\overline {\overline V}_i)_i)    
\]
where $\overline V_k = \ker \beta_{k;M}/(\ker \beta_{k;M} \cap \image \alpha_{k;M})$, $\overline{\overline V}_i = V_i$ for $i\neq k$, and 
\[
    \overline{\overline V}_k = \ker \beta_{k;M'}/ (\ker \beta_{k;M'} \cap \image \alpha_{k;M'}),
\]
where $M' = \mu_k(M)\oplus V_k \in \modu \mathcal P(\mu_k(T))$. It is clear that $\overline{\overline V}_k = V_k$ and $\mu_k(V_k) = 0$. By \Cref{prop: double mutation}, we have $\mu_k^2(M) \oplus \overline V_k \cong M$ if $\ker \beta_k$ and $\image \alpha_k$ are free, hence $\mu_k^2(\mathcal M) \cong \mathcal M$.
\end{proof}

From now on unless otherwise stated, the actual meaning of $\mu_k$ (with or without decoration) should just depend on the argument (being a module or a decorated module).

\begin{lemma}\label{lemma: when is mutation locally free}
    Let $\mathcal M = (M, V)$ be locally free. Then $\mu_k(\mathcal M) = (\overline M, \overline V)$ is locally free if and only if both $\ker \beta_k$ and $\image \alpha_k$ are free.
\end{lemma}

\begin{proof}
    By \Cref{lemma: image gamma free}, $\image \gamma_k$ is always free. Since the domain of $\gamma_k$ is free, $\ker \gamma_k$ is also free. By \Cref{def: decorated mutation},
    \[
        \overline M(k) = \frac{\ker \gamma_k}{\image \beta_k} \oplus \image \gamma_k \oplus \frac{\ker \alpha_k}{\image \gamma_k}\oplus V_k \quad \text{and} \quad \overline V_k = \frac{\ker \beta_k}{\ker \beta_k \cap \image \alpha_k},
    \]
    while other components of $\mu_k(\mathcal{M})$ are free given that $\mathcal M$ is locally free.

    Since the domains of $\beta_k$ and of $\alpha_k$ are both free (given that $M$ is locally free), $\ker \beta_k$ and $\image \alpha_k$ are free if and only if $\image \beta_k$ and $\ker \alpha_k$ are free.

    If $\ker \beta_k$ and $\image \alpha_k$ are free, by \Cref{prop: double mutation}, $\overline V_k$ is free. It is also clear that $\overline M(k)$ is free. Thus $\mu_k(\mathcal{M})$ is locally free.

    If $\mu_k(\mathcal{M})$ is free, then all above direct summands of $\overline M(k)$ are free. In particular, $\image \beta_k$ and $\ker \alpha_k$ are free, and thus $\ker \beta_k$ and $\image \alpha_k$ are free.
\end{proof}

\subsection{\texorpdfstring{An example of type $C_3$}{An example of type C3}}\label{subsection: example c3}
The quivers below arise from triangulations of a $4$-gon with one orbifold point related by a flip. The associated exchange matrices are of mutation type $C_3$ \cite{fomin2003cluster}.

\begin{center}
\begin{tikzpicture}[scale=0.8]
    \begin{scope}[shift={(0,-0.5)}]
        \node[] (3) at (0, 0) {$3$};
        \node[] (2) at (-1, {sqrt(3)}) {$2$};
        \node[] (1) at (1, {sqrt(3)}) {$1$};

        \draw[->] (3) to node[right]{\scriptsize $b$} (1);
        \draw[->] (1) to node[above]{\scriptsize $c$} (2);
        \draw[->] (2) to node[left]{\scriptsize $a$} (3);

        \draw[->] (3) to[out=240, in=300, looseness=6] node[below]{\scriptsize $\varepsilon_3$} (3);
    \end{scope}

    \draw[->] (2,0) to node[above]{$\mu_3$} (3,0);

    \begin{scope}[shift={(5,0.5)}]
        \node[] (3) at (0, 0) {$3$};
        \node[] (2) at (-1, -{sqrt(3)}) {$2$};
        \node[] (1) at (1, -{sqrt(3)}) {$1$};

        \draw[->] (1) to node[right]{\scriptsize $\bar b$} (3);
        \draw[->] (2) to node[below]{\scriptsize $c'$} (1);
        \draw[->] (3) to node[left]{\scriptsize $\bar a$} (2);

        \draw[->] (3) to[out=60, in=120, looseness=6] node[above]{\scriptsize $\varepsilon_3$} (3);
    \end{scope}
\end{tikzpicture}
\end{center}

We showcase below two instances of mutations of representations. Both $M$ and $N$ are locally free. However, $\mu_k(N)$ is not locally free as $\mu_k(N)(3)$ is not free over $H_3$.

\begin{center}
    \begin{tikzpicture}[scale=0.8]

    \node at (-2.5, 0) {$M = $};
    
    \begin{scope}[shift={(0,-0.5)}]
        \node[] (3) at (0, 0) {$\Bbbk^2$};
        \node[] (2) at (-1, {sqrt(3)}) {$\Bbbk$};
        \node[] (1) at (1, {sqrt(3)}) {$\Bbbk$};

        \draw[->] (3) to node[right]{\scriptsize $(0, 1)$} (1);
        \draw[->] (1) to node[above]{\scriptsize $0$} (2);
        \draw[->] (2) to node[left]{\scriptsize $(1, 0)^\intercal$} (3);

        \draw[->] (3) to[out=240, in=300, looseness=6] node[above left, xshift=-5pt]{\scriptsize $\begin{psmallmatrix}
            0 & 0 \\
            1 & 0
        \end{psmallmatrix}$} (3);
    \end{scope}

    \draw[->] (2,0) to node[above]{$\mu_3$} (3,0);

    \begin{scope}[shift={(5,0.5)}]
        \node[] (3) at (0, 0) {$0$};
        \node[] (2) at (-1, -{sqrt(3)}) {$\Bbbk$};
        \node[] (1) at (1, -{sqrt(3)}) {$\Bbbk$};

        \draw[->] (1) to node[right]{\scriptsize $0$} (3);
        \draw[->] (2) to node[below]{\scriptsize $\mathrm{id}$} (1);
        \draw[->] (3) to node[left]{\scriptsize $0$} (2);

        \draw[->] (3) to[out=60, in=120, looseness=6] node[above]{\scriptsize $0$} (3);
    \end{scope}

    \node at (7.5, 0) {$= \mu_3(M)$};
\end{tikzpicture}
\end{center}

\begin{center}
    \begin{tikzpicture}[scale=0.8]

    \node at (-2.5, 0) {$N = $};
    
    \begin{scope}[shift={(0,-0.5)}]
        \node[] (3) at (0, 0) {$\Bbbk^2$};
        \node[] (2) at (-1, {sqrt(3)}) {$\Bbbk$};
        \node[] (1) at (1, {sqrt(3)}) {$\Bbbk$};

        \draw[->] (3) to node[right]{\scriptsize $(1, 0)$} (1);
        \draw[->] (1) to node[above]{\scriptsize $0$} (2);
        \draw[->] (2) to node[left]{\scriptsize $(0, 1)^\intercal$} (3);

        \draw[->] (3) to[out=240, in=300, looseness=6] node[above left, xshift=-5pt]{\scriptsize $\begin{psmallmatrix}
            0 & 0 \\
            1 & 0
        \end{psmallmatrix}$} (3);
    \end{scope}

    \draw[->] (2,0) to node[above]{$\mu_3$} (3,0);

    \begin{scope}[shift={(5,0.5)}]
        \node[] (3) at (0, 0) {$\Bbbk^2$};
        \node[] (2) at (-1, -{sqrt(3)}) {$\Bbbk$};
        \node[] (1) at (1, -{sqrt(3)}) {$\Bbbk$};

        \draw[->] (1) to node[right]{\scriptsize $(0, 1)^\intercal$} (3);
        \draw[->] (2) to node[below]{\scriptsize $\mathrm{id}$} (1);
        \draw[->] (3) to node[left]{\scriptsize $(1, 0)$} (2);

        \draw[->] (3) to[out=60, in=120, looseness=6] node[above]{\scriptsize $0$} (3);
    \end{scope}

    \node at (7.5, 0) {$= \mu_3(N)$};
\end{tikzpicture}
\end{center}

\section{\texorpdfstring{Recurrence of $\mathbf g$-vectors and $F$-polynomials}{Recurrence of g-vectors and F-polynomials}} \label{section: recurrence}

In this section, we introduce $\mathbf g$-vectors and $F$-polynomials for locally free (decorated) representations. Their importance arises from that for certain representations they have the exact same recurrence under mutations as the $\mathbf g$-vectors and $F$-polynomials defined for cluster variables in \Cref{section: preliminaries}.

\subsection{$\mathbf g$-vectors of decorated representations}

Following Derksen--Weyman--Zelevinsky \cite[(1.13)]{derksen2010quivers} in the skew-symmetric case, we associate an integer-valued vector $\mathbf g(\mathcal M) = (g_i)_i\in \mathbb Z^n$ to a decorated locally free representation $\mathcal M = (M, V)$ as follows. Let $k\in I = \{1, \dots, n\}$. Recall the following diagram (of free $H_k$-modules) \eqref{eq: mutation diagram} from \Cref{subsection: construction of mutation}:
\[
    \begin{tikzcd}[column sep = small]
        & M(k) \ar[rd, "\beta_k"]& \\
        M_{\mathrm{in}}(k) \ar[ru, "\alpha_{k}"] && M_{\mathrm{out}}(k) \ar[ll, "\gamma_{k}"]
    \end{tikzcd}
\]
Notice that $\ker \gamma_k$ is always free over $H_k$ by \Cref{lemma: image gamma free}. We define
\begin{equation}\label{eq: def g vector}
    g_k = g_k(\mathcal M) \coloneqq  \rk_{H_k} \ker \gamma_k - \rk_{H_k} M(k) + \rk_{H_k} V_k \in \mathbb Z.
\end{equation}
When the base ring is clear, we omit the subscript from $\rk$. The vector $\mathbf g(\mathcal M)$ is called the \emph{$\mathbf g$-vector} of $\mathcal M$. For just a module $M\in \modu \mathcal P(T)$, we write $\mathbf g(M)  = \mathbf g((M, 0))$. It is clear that $\mathbf g$-vector is additive with respect to direct sums.

For $M$ such that $\ker \beta_k$ is free (see \Cref{subsection: mutation involution}), define 
\begin{equation}\label{eq: define h vector}
    h_k = h_k(\mathcal M) \coloneqq - \rk \ker \beta_k.
\end{equation}
If $h_i$ is defined for every $i\in I$, they form a (non-positive) integer vector $\mathbf h(\mathcal M) = (h_i)_{i}\in \mathbb Z^n$.

\begin{lemma}\label{lemma: recursion of g vector}
Let $\mathcal M = (M, V)$ be a decorated locally free $\mathcal P(T)$-module. Assume that the decorated mutation $\overline {\mathcal M} = (\overline M, \overline V) \coloneqq \mu_k(\mathcal M)$ is again locally free. Then the $\mathbf g$-vectors $\mathbf g({\overline {\mathcal M}}) = (g'_j)_{j}$ and $\mathbf g(\mathcal M) = (g_j)_{j}$ satisfy
\[
    g_j' = \begin{cases} -g_k \quad & \text{if $j = k$};\\
        g_j + [b_{j,k}]_+g_k - b_{j,k}h_k \quad & \text{if $j \neq k$},
    \end{cases}
\]
and $g_k = h_k - h_k'$ where $h_k'\coloneqq h_k(\overline{\mathcal M})$.
\end{lemma}

\begin{proof}

By \Cref{lemma: when is mutation locally free}, that $\overline {\mathcal M}$ is locally free is equivalent to the condition that $\ker \beta_k$ and $\image \alpha_k$ are free. Hence both $h_k$ and $h_k'$ are well-defined as the decorated mutation in the current situation is involutive. We next show that $ - h_k' = - h_k + g_k$, i.e.,
\begin{equation}\label{eq: desired formula for ker bar beta}
    \rk \ker \bar \beta_k = \rk \ker \beta_k + \rk \ker \gamma_k - \rk M(k) + \rk V_k.
\end{equation}
Write $\bar \alpha_i = \alpha_{i; \overline{M}}$, $\bar \beta_i = \beta_{i; \overline{M}}$, $\bar \gamma_i = \gamma_{i; \overline{M}}$ for $i\in I$ as in \eqref{eq: mutation diagram}. The map
\[
    \bar \beta_k \colon \frac{\ker \gamma_k}{\image \beta_k} \oplus \image \gamma_k \oplus \frac{\ker \alpha_k}{\image \gamma_k}\oplus V_k \longrightarrow M_{\mathrm{in}}(k)
\]
satisfies $\bar \beta_k = \iota \pi$ where $\iota$ denotes the inclusion $\ker \alpha_k \subset M_{\mathrm{in}}(k)$ and $\pi$ is the projection onto the second and the third summand followed by an isomorphism with $\ker \alpha_k$ given by the splitting data $\sigma\colon \ker \alpha_k/\image \gamma_k \rightarrow \ker \alpha_k$. Thus $\ker \bar \beta_k = (\ker \gamma_k / \image \beta_k)\oplus V_k$ and the desired equation (\ref{eq: desired formula for ker bar beta}) follows. It then follows immediately that $g_k' = -g_k$.

It remains to show that for $j\neq k$, $g_j' = g_j + [b_{j,k}]_+g_k - b_{j,k}h_k$, i.e.,
\begin{equation}\label{eq: desired formula for ker gamma}
    \rk_{H_j} \ker \bar \gamma_j = \begin{cases}
    \rk_{H_j} \ker \gamma_j + b_{j,k}\cdot \rk_{H_k} \ker \bar \beta_k \quad & \text{if $b_{j,k}\geq 0$}, \\
    \rk_{H_j} \ker \gamma_j + b_{j,k}\cdot \rk_{H_k} \ker \beta_k \quad & \text{if $b_{j,k}\leq 0$}.
    \end{cases}
\end{equation}
Notice that computing the quantity $g_j$ only uses the quadrilateral in $T$ where $j$ is a diagonal. Thus if $j$ and $k$ are not adjacent in $Q(T)$, mutation at $k$ does not change the component $g_j$. It then suffices to assume $b_{j,k}\neq 0$, and in view of the symmetry between $\mathcal M$ and $\overline {\mathcal M}$, it is enough to prove (\ref{eq: desired formula for ker gamma}) for the case $b_{j,k}<0$.

For every non-loop arrow $a\in Q_1(T)$ with $t(a) = j$, there is a unique non-loop arrow $c\in Q_1(T)$ with $h(c) = j$ (if it exists) so that the two arrows come from the same regular triangle $\triangle$. In this way, we see that $\ker \gamma_j$ is the direct sum of the kernel of every component 
\[
    \gamma_{j}^{(a)}\colon H_j(\bar a\otimes M(h(a))) \rightarrow H_j(c\otimes M(t(c))),\quad \gamma_j^{(a)}(\bar a\otimes m) = c\otimes M(b)(m),
\]
where $b$ is the third arrow in $\triangle$. (If such $c$ does not exist, we adopt the convention that the codomain of $\gamma_j^{(a)}$ is zero.)

\textit{Case 1.} Suppose $k$ is pending and let $i$, $j$ and $k$ form a regular triangle (counterclockwise) in $T$; see for example the figure in \Cref{proof: define rep mutations}. By definition we have
\begin{align}
    &\gamma_j(i, k) = \gamma_j^{(\bar a)} \colon \Delta(T)_{k, j} \otimes_{H_k} M(k) \rightarrow A(T)_{j, i} \otimes_{H_i} M(i), \quad \bar a\otimes m \mapsto c \otimes M(b)(m); \label{eq: formula gamma j} \\
    &\bar \gamma_j(k, i) = \bar \gamma_j^{(\bar c')} \colon \Delta(T')_{i, j} \otimes_{H_i} M(i) \rightarrow A(T')_{j, k} \otimes_{H_k} \overline M(k), \quad \bar c' \otimes m \mapsto \bar a \otimes \overline M(\bar b)(m). \label{eq: formula bar gamma j}
\end{align}
It suffices to show
\begin{equation}\label{eq: simplified gamma}
    \rk_{H_j} \ker \bar \gamma_j(k, i)  = \rk_{H_j} \ker \gamma_j(i, k) + b_{j,k}\cdot \rk_{H_k} \ker \beta_k.
\end{equation}
Indeed, the arc $j$ may be contained in another regular triangle $\triangle$ in $T$, which is also shared by $T'$. According to the above analysis, the other component of $\gamma_j$ and that of $\bar \gamma_j$ both arising from $\triangle$ are identical and thus do not have any impact on (\ref{eq: desired formula for ker gamma}).

By the formulas \eqref{eq: formula gamma j} and \eqref{eq: formula bar gamma j}, we have
\begin{align} \label{eq: description ker gamma}
\begin{split}
    & \ker \gamma_j(i, k) = H_j \{\bar a \otimes m \mid m\in \ker M(b)\} \cong H_j\otimes_\Bbbk \ker M(b), \\
    & \ker \bar \gamma_j(k, i) = H_j \{\bar c' \otimes m \mid m \in \ker \overline M(\bar b)\} \cong H_j\otimes_\Bbbk \ker \overline M(\bar b).
\end{split}
\end{align}
Consider the $\Bbbk$-linear map $M(b\varepsilon_k) \colon M(k) \rightarrow M(i)$. We claim
\[
    \ker \overline{M}(\bar b) = M(b\varepsilon_k)(\ker M(b)).
\]
By \eqref{eq: overline M bar b}, we have $\overline M(\bar b) = \bar \alpha_k(\bar b\otimes m)$ for $m\in M(i)$. Thus $\ker \overline M(\bar b) = \{m\in M(i)\mid \bar b\otimes m\in \ker \bar \alpha_k\}$. By part (1) of \Cref{lemma: ker bar alpha = im beta & ker alpha = im bar beta}, $\ker \bar \alpha_k = \image \beta_k$. Further we have
\[
    \image \beta_k = \{\bar b \otimes M(b\varepsilon_k)(m') + \varepsilon_k\bar b\otimes M(b)(m')\mid m'\in M(k)\}.
\]
Therefore we can describe
\[
    \ker \overline M(\bar b) = \{m\in M(i)\mid m = M(b\varepsilon_k)(m'),\, M(b)(m')=0,\, m'\in M(k)\} = M(b\varepsilon_k)(\ker M(b)),
\]
as claimed. Now $M(b\varepsilon_k)$ restricts to a surjective map $M(b\varepsilon_k) \colon \ker M(b) \rightarrow \ker \overline M(\bar b)$. The kernel of this map is $\{m\in M(k)\mid M(b)(m)=0,\, M(b\varepsilon_k)(m)=0\} = \ker \beta_k$. Then we have $\dim_\Bbbk \ker M(b) = \dim_\Bbbk \ker \overline M(\bar b) + \dim_\Bbbk \ker \beta_k$ and by (\ref{eq: description ker gamma})
\[
    \rk_{H_j} \ker \gamma_j(i, k) = \rk_{H_j} \ker \bar \gamma_j (k, i) + \dim_\Bbbk \ker \beta_k.
\]
Notice that when $k$ is pending, $b_{j,k} = -2 = -d_k$ and thus $\dim_\Bbbk \ker \beta_k = -b_{j,k} \rk_{H_k} \ker \beta_k$. The equation (\ref{eq: simplified gamma}) follows.

\textit{Case 2}. Let $k$ be ordinary. Recall the following local quivers (with possible loops omitted) under the flip $\mu_k$, for example as in \Cref{proof: define rep mutations}.
\begin{center}
    \begin{tikzpicture}[scale=1.2]
        \begin{scope}[shift={(0,0)}]
            \node (k) at (0, 0) {$k$};
            \node (p) at (1, 0) {$p$};
            \node (q) at (0, 1) {$q$};
            \node (u) at (-1, 0) {$u$};
            \node (v) at (0, -1) {$v$};

            \draw[->] (k) to node[below]{\scriptsize $b_1$} (p);
            \draw[->] (p) to node[right]{\scriptsize $c_1$} (q);
            \draw[->] (q) to node[left]{\scriptsize $a_1$} (k);
            \draw[->] (v) to node[right]{\scriptsize $a_2$} (k);
            \draw[->] (k) to node[above]{\scriptsize $b_2$} (u);
            \draw[->] (u) to node[left]{\scriptsize $c_2$} (v);
        \end{scope}

        \draw[->] (2,0) to node[above]{$\mu_k$} (3,0);

        \begin{scope}[shift={(5,0)}]
            \node (k) at (0, 0) {$k$};
            \node (p) at (1, 0) {$p$};
            \node (q) at (0, 1) {$q$};
            \node (u) at (-1, 0) {$u$};
            \node (v) at (0, -1) {$v$};

            \draw[->] (k) to node[right]{\scriptsize $\bar{a}_1$} (q);
            \draw[->] (p) to node[above]{\scriptsize $\bar{b}_1$} (k);
            \draw[->] (k) to node[left]{\scriptsize $\bar{a}_2$} (v);
            \draw[->] (v) to node[right]{\scriptsize $d_1$} (p);
            \draw[->] (u) to node[below]{\scriptsize $\bar{b}_2$} (k);
            \draw[->] (q) to node[left]{\scriptsize $d_2$} (u);
        \end{scope}
    \end{tikzpicture}
\end{center}
Now $j$ can be $q$ or $v$. (Note that $q$ and $v$ may coincide.) It suffices to prove the case $j=v$ by symmetry.

By the analysis before \emph{Case 1} and as similarly argued in \emph{Case 1}, any other component of $\gamma_v$ and that of $\bar \gamma_v$ arising from outside the above quivers are identical and thus do not have any impact on \eqref{eq: desired formula for ker gamma}. Now we have two sub-cases according to whether $q$ and $v$ coincide.

\emph{Case 2.1.} Suppose $q\neq v$. Then we have the following component of $\gamma_v$ and that of $\bar \gamma_v$:
\begin{align*}
    & \gamma_{v}^{(a_2)}\colon H_v(\bar a_2\otimes M(k)) \rightarrow H_v(c_2\otimes M(u)),\quad \bar a_2\otimes m \mapsto c_{2}\otimes M(b_2)(m); \\
    & \bar \gamma_{v}^{(\bar d_1)} \colon H_v(\bar d_1\otimes M(p)) \rightarrow H_v(\bar a_2\otimes M(k)),\quad \bar d_1\otimes m \mapsto \bar a_2\otimes \overline M(\bar b_1)(m).
\end{align*}
Then \eqref{eq: desired formula for ker gamma} reduces to
\[
    \rk_{H_v} \ker \bar \gamma_v^{(\bar d_1)} = \rk_{H_v} \ker \gamma_v^{(a_2)} + b_{v, k}\cdot \rk_{H_k} \ker \beta_k.
\]
Notice that 
\begin{align*}
    & \ker \gamma_v^{(a_2)} = H_v\{\bar a_2\otimes m\mid m\in \ker M(b_2)\} \cong H_v\otimes_\Bbbk \ker M(b_2),\\
    & \ker \bar \gamma_v^{(\bar d_1)} = H_v\{\bar d_1 \otimes m \mid m \in \ker \overline M(\bar b_1)\} \cong H_v\otimes \ker \overline M(\bar b_1).
\end{align*}
The $\Bbbk$-linear map $M(b_1)\colon M(k) \rightarrow M(p)$ restricts to a surjective map $M(b_1) \colon \ker M(b_2) \rightarrow \ker \overline M(\bar b_1)$, having kernel $\{m\in M(k) \mid M(b_1)(m) = 0,\ M(b_2)(m) = 0\}$, which clearly coincides with $\ker \beta_k$. Therefore we have
\[
    \dim_\Bbbk \ker M(b_2) = \dim_\Bbbk \ker \overline M(\bar b_1) + \dim_\Bbbk \ker \beta_k.
\]
Then equation (\ref{eq: desired formula for ker gamma}) follows as $b_{v, k} = -1$ in this case.

\emph{Case 2.2} Suppose $v=q$. We update the quivers:
\begin{center}
    \begin{tikzpicture}[scale=1.5]
        \begin{scope}[shift={(0,0)}]
            \node (k) at (0, 0) {$k$};
            \node (p) at (1, 0) {$p$};
            
            \node (u) at (-1, 0) {$u$};
            \node (v) at (0, 1) {$v$};

            \draw[->] (k) to node[below]{\scriptsize $b_1$} (p);
            \draw[->] (p) to node[right]{\scriptsize $c_1$} (v);
            \draw[->, bend right = 10] (v) to node[left]{\scriptsize $a_2$} (k);
            \draw[->, bend left = 10] (v) to node[right]{\scriptsize $a_1$} (k);
            \draw[->] (k) to node[below]{\scriptsize $b_2$} (u);
            \draw[->] (u) to node[left]{\scriptsize $c_2$} (v);
        \end{scope}

        \draw[->] (2.1,0.3) to node[above]{$\mu_k$} (2.9,0.3);

        \begin{scope}[shift={(5,0)}]
            \node (k) at (0, 0) {$k$};
            \node (p) at (1, 0) {$p$};

            \node (u) at (-1, 0) {$u$};
            \node (v) at (0, 1) {$v$};

            \draw[->, bend right=10] (k) to node[right]{\scriptsize $\bar{a}_2$} (v);
            \draw[->] (p) to node[below]{\scriptsize $\bar{b}_1$} (k);
            \draw[->, bend left=10] (k) to node[left]{\scriptsize $\bar{a}_1$} (v);
            \draw[->] (v) to node[right]{\scriptsize $d_1$} (p);
            \draw[->] (u) to node[below]{\scriptsize $\bar{b}_2$} (k);
            \draw[->] (v) to node[left]{\scriptsize $d_2$} (u);
        \end{scope}
    \end{tikzpicture}
\end{center}
Notice that no other vertices are connected to $v$ (in either $Q(T)$ or $Q(T')$). In this case, $v$ is necessarily ordinary and $b_{v,k}= -2$. Then \eqref{eq: desired formula for ker gamma} follows from combining the two equations:
\begin{align*}
    & \rk_{H_v} \ker \bar \gamma_v^{(\bar d_1)} = \rk_{H_v} \ker \gamma_v^{(a_2)} -\rk_{H_k} \ker \beta_k, \\
    & \rk_{H_v} \ker \bar \gamma_v^{(\bar d_2)} = \rk_{H_v} \ker \gamma_v^{(a_1)} -\rk_{H_k} \ker \beta_k,
\end{align*}
each of which can be proven similarly as in the previous sub-case.
\end{proof}

\subsection{\texorpdfstring{Recurrence of $F$-polynomials}{Recurrence of F-polynomials}}

Let $F(M)(y_1, \dots, y_n)$ denote the \emph{locally free} $F$-polynomial of $M\in \modu \mathcal P(T)$ as defined by (\ref{eq: loc free f poly}). For $\mathcal M = (M, V)$, let $F({\mathcal M}) = F(M)$.

\begin{proposition}\label{prop: recursion of f polynomial}
Let $\mathcal M = (M, V)$ be a locally free decorated $\mathcal P(T)$-module such that its mutation $\mu_k(\mathcal M) = \overline {\mathcal M} = (\overline M, \overline V)$ is again locally free. Their $F$-polynomials are related by
\[
    (y_k+1)^{h_k} F(\mathcal M)(y_1, \dots, y_n) = (y'_k + 1)^{h_k'} F({\overline{\mathcal M}})(y_1', \dots, y_n').
\]
where $y_k' = y_k^{-1}$ and for $i\neq k$, $y_i' = y_iy_k^{[b_{ki}]_+}(1 + y_k)^{-b_{ki}}$.
\end{proposition}

\begin{proof}

Let $N = (N(1), \dots, N(n))$ be an $n$-tuple where each $N(i)$ is a free $H_i$-submodule of $M(i)$. It defines a point in $\operatorname{Gr_{l.f.}}(\mathbf e, M)$ where $e =  (e_i)_i \coloneqq (\rk_{H_i} N(i))_i$ if
\begin{equation}\label{eq: condition 1}
    M(c)(N(j)) \subset N(i) \quad \text{for any arrow $c\colon j\rightarrow i$ of $Q(T)$ not incident to $k$ and}
\end{equation}
\begin{equation}\label{eq: condition 2}
    \alpha_k(N_{\mathrm{in}}(k)) \subset N(k) \subset \beta_k^{-1}(N_{\mathrm{out}}(k))
\end{equation}
where $N_{\mathrm{in}}(k)$ and $N_{\mathrm{out}}(k)$ are defined as in (\ref{eq: def m in and m out}).

Let $e' = (e_i)_{i\neq k}$. For every pair of non-negative integers $r\leq s$, let $Z_{e';r,s}$ be the set of tuples $(N(i))_{i\neq k}$ where each $N(i)$ is a rank $e_i$ free submodule of $M(i)$ such that
\begin{enumerate}
    \item[(C1)] $(N(i))_{i\neq k}$ satisfies (\ref{eq: condition 1});
    \item[(C2)] there exists a free submodule $N(k)$ of $M(k)$ such that the $n$-tuple $(N(i))_i$ satisfies (\ref{eq: condition 2});
    \item[(C3)] the injective envelope of $\alpha_k(N_{\mathrm{in}}(k))$ has rank $r$ and a maximal free submodule of $\beta_k^{-1}(N_{\mathrm{out}}(k))$ has rank $s$.
\end{enumerate}

Let $\tilde{Z}_{e;r,s}(M)$ be the subset of $\operatorname{Gr_{l.f.}}(\mathbf e, M)$ of tuples $(N(i))_i$ such that $(N(i))_{i\neq k}$ belongs to $Z_{e';r,s}(M)$. Decompose $Z_{e';r,s}(M)$ further into locally closed subsets $Z_{e';r,s}(M) = \bigsqcup_{j\in J} Z_j$ such that the map $\tilde{Z}_{e;r,s}(M) \rightarrow Z_{e';r,s}(M)$ (by forgetting $N(k)$) is a fiber bundle on each $Z_j$. Notice that the fiber over some $(N(i))_{i\neq k}$ is the set of $N(k)$ (of rank $e_k$) satisfying (\ref{eq: condition 2}). By \cite[Proposition 4.9]{mou2022locally}, it has the Euler characteristic $\binom{s-r}{e_k-r}$. Then we have
\[
    \chi({Z}_{e';r,s}(M)) = \sum_{j\in J} \binom{s-r}{e_k-r} Z_j = \binom{s-r}{e_k-r} \chi ( \tilde{Z}_{e;r,s}(M))
\]
and thus
\begin{equation}\label{eq: decomposition of quiver grass}
    \chi(\operatorname{Gr_{l.f.}}(\mathbf e, M)) = \chi\left(\bigsqcup_{r,s} \tilde{Z}_{e;r,s}(M)\right) = \sum_{r,s} \binom{s-r}{e_k-r} \chi(Z_{e';r,s}(M)).
\end{equation}
The $F$-polynomial of $M$ can be rewritten as
\begin{equation}\label{eq: rewrite f-polynomial for M}
    F(M)(y_1, \dots, y_n) = \sum_{e',r,s} \chi(Z_{e';r,s}(M))y_k^r(y_k+1)^{s-r}\prod_{i\neq k}^n y_i^{e_i}.
\end{equation}

Recall the following diagrams for $M$ and $\overline M$:
\[
    \begin{tikzcd}[column sep = small]
        & M(k) \ar[rd, "\beta_k"]& \\
        M_{\mathrm{in}}(k) \ar[ru, "\alpha_k"] && M_{\mathrm{out}}(k) \ar[ll, "\gamma_k"]
    \end{tikzcd}
    \xlongrightarrow[]{\mu_k}
    \begin{tikzcd}[column sep = small]
        & \overline M(k) \ar[ld, swap, "\bar \beta_k"]& \\
        M_{\mathrm{in}}(k) \ar[rr, swap, "\bar \gamma_k"]  &&  M_{\mathrm{out}}(k) \ar[lu, swap, "\bar \alpha_k"]
    \end{tikzcd}
\]
We next show the following two sets are identical
\begin{equation}\label{eq: equality two varieties}
    Z_{e';r,s}(M) = Z_{e';\bar r, \bar s}(\overline M)
\end{equation}
when
\begin{equation}\label{eq: relations between parameters}
    \bar r \coloneqq \sum_{i} [-b_{ki}]_+e_i - h_k - s, \quad \bar s \coloneqq \sum_i [b_{ki}]_+e_i - h_k' - r.
\end{equation}
We prove (\ref{eq: equality two varieties}) by showing that any $(N(i))_{i\neq k} \in Z_{e';r,s}$ also belongs to $Z_{e'; \bar r, \bar s}$, that is, $(N(i))_{i\neq k}$ satisfies (C1)-(C3) for the $\mathcal P(\mu_k(T))$-module $\overline M$; the other inclusion is obtained in view of the symmetry between $\mathcal M$ and $\overline{\mathcal M}$.

To check (C1), it suffices to examine the newly created arrows in $Q(\mu_k(T))$ not incident to $k$. The linear maps associated to those arrows are encoded in $\bar \gamma_k$, which equals $\beta_k \alpha_k$. As $\alpha_k(N_{\mathrm{in}}(k)) \subset \beta_k^{-1}(N_{\mathrm{out}}(k))$, we have $\bar \gamma_k (N_{\mathrm{in}}(k)) \subset N_{\mathrm{out}}(k)$ and hence (C1) is satisfied by $(N(i))_{i\neq k}$.

For (C2), first we have $\bar \beta_k \bar \alpha_k (N_{\mathrm{out}}(k)) = \gamma_k (N_{\mathrm{out}}(k)) \subset N_{\mathrm{in}}(k)$. Then we examine how $\bar \alpha_k(N_{\mathrm{out}}(k))$ is embedded in $\bar \beta_k^{-1}(N_{\mathrm{in}}(k))$. We have $\bar \alpha_k (N_{\mathrm{out}}(k)) = N_{\mathrm{out}}(k)/(N_{\mathrm{out}}(k) \cap \image {\beta_k}) \subset \coker \beta_k$. Recall that to define $\overline M$, a choice of splitting $M_\mathrm{out}(k) = \ker \gamma_k \oplus \image \gamma_k$ has been made. Since $\image \beta_k \subset \ker \gamma_k$ by \Cref{lemma: relations of alpha beta gamma}, we obtain an induced splitting $\coker \beta_k = (\coker \beta_k/\image \gamma_k) \oplus \image \gamma_k$. The splitting data also induces a splitting
\[
    \bar \alpha_k (N_{\mathrm{out}}(k)) = \frac{N_{\mathrm{out}}(k)/(N_{\mathrm{out}}(k) \cap \image \beta_k)}{\gamma_k(N_{\mathrm{out}}(k))} \oplus \gamma_k(N_{\mathrm{out}}(k)).
\]
The natural map from $N_{\mathrm{out}}(k)$ to $\coker \beta_k$ induces an inclusion of the first summand into $\coker \beta_k/\image \gamma_k$. The second summand (which is free by the definition of $\gamma_k$) embeds into $N_{\mathrm{in}}(k) \cap \ker \alpha_k$ since $\alpha_k\gamma_k = 0$.

On the other hand, we have
\begin{equation}\label{eq: decomposition of bar beta inverse}
    \bar \beta_k^{-1}(N_{\mathrm{in}}(k)) = \frac{\coker \beta_k}{\image \gamma_k} \oplus (N_{\mathrm{in}}(k) \cap \ker \alpha_k) \oplus V_k,
\end{equation}
where the first summand is assumed to be free. Therefore $\bar \alpha_k(N_{\mathrm{out}}(k))$ is contained in the free submodule $\coker \beta_k/\image \gamma_k \oplus \gamma_k(N_{\mathrm{out}}(k))$ of $\bar \beta_k^{-1}(N_{\mathrm{in}}(k))$. Now we see that (C2) is also satisfied.

Last we check (C3) with $\bar r$ and $\bar s$ given in (\ref{eq: relations between parameters}). Consider the exact sequence
\[
    0 \rightarrow \ker \beta_k \rightarrow \beta_k^{-1}(N_{\mathrm{out}}(k)) \xrightarrow[]{\beta} N_{\mathrm{out}}(k) \cap \image \beta_k \rightarrow 0.
\]
Since $\ker \beta_k$ is assumed to be free (of rank $-h_k$), the exact sequence splits $\beta_k^{-1}(N_{\mathrm{out}}(k)) = \ker \beta_k \oplus (N_{\mathrm{out}}(k) \cap \image \beta_k)$. Then $N_{\mathrm{out}}(k) \cap \image \beta_k$ has a maximal free submodule of rank $s + h_k$. Now with $\rk N_{\mathrm{out}}(k) = \sum_i [-b_{ki}]_+e_i$, we obtain
\[
    \rk E(\bar \alpha(N_{\mathrm{out}}(k)))= \rk E\left( \frac{N_{\mathrm{out}}(k)}{N_{\mathrm{out}}(k) \cap \image \beta_k} \right)= \rk N_{\mathrm{out}}(k) - (h_k + s) = \bar r,
\]
where $E(\cdot)$ denotes the injective hull.

For $\bar \beta_k^{-1}(N_{\mathrm{in}}(k))$, the first and third summands together form $\ker \bar \beta_k$, which is free of rank $-h_k'$. Denote by $F(\cdot)$ a maximal free submodule, which is unique up to isomorphism. Since $\alpha(N_{\mathrm{in}}(k)) = N_{\mathrm{in}}(k) / (N_{\mathrm{in}}(k) \cap \ker \alpha)$, we have
\[
    \rk F( N_{\mathrm{in}}(k) \cap \ker \alpha ) = \rk N_{\mathrm{in}}(k) - \rk E( \alpha(N_{\mathrm{in}}(k))) = \sum_i [b_{ki}]_+e_i - r
\]
and thus $\rk F(\bar \beta^{-1}(N_{\mathrm{in}}(k))) = \bar s$. We have finished showing (\ref{eq: equality two varieties}).

Now consider the decomposition (\ref{eq: decomposition of quiver grass}) but for $\overline M$ and rewrite $F_{\overline M}$ in the way as in (\ref{eq: rewrite f-polynomial for M}):
\begin{equation}
    F({\overline M})(y_1', \dots, y_n') = \sum_{e', \bar r, \bar s} \chi(Z_{e';\bar r, \bar s}(\overline M)) (y_k')^{\bar r}(y_k'+1)^{\bar s - \bar r} \prod_{i\neq k}^n (y_i')^{e_i}.
\end{equation}
Using (\ref{eq: equality two varieties}), (\ref{eq: relations between parameters}) and the relation between $y_i$ and $y_i'$, we express
\begin{multline*}
        F({\overline M})(y_1', \dots, y_n') = \sum_{e', r, s}\chi(Z_{e';r,s}(M))
        y_k^{s+h_k-\sum [-b_{ki}]_+e_1}\\
        \cdot (y_k^{-1} + 1)^{s-r+ \sum [b_{ki}]_+e_i- \sum [-b_{ki}]_+e_i + h_k - h_k'} y_k^{\sum [b_{ki}]_+e_i}(1 + y_k)^{- \sum b_{ki} e_i}\prod_{i\neq k}^n y_i^{e_i}.
\end{multline*}
\begin{equation*}
    = y_k^{h_k'}(y_k+1)^{h_k-h_k'}\sum_{e', r, s}\chi(Z_{e';r,s}(M)) y_k^{r}(y_k+1)^{s-r}\prod_{i\neq k}^n y_i^{e_i}.
\end{equation*}
Comparing the above expression with (\ref{eq: rewrite f-polynomial for M}) finishes the proof.
\end{proof}

Suppose the vector $\mathbf h(M)$ is defined (see (\ref{eq: define h vector})). The next proposition describes how $\mathbf h(M)$ can be obtained from the locally free $F$-polynomial $F(M)$.

\begin{proposition}[cf. {\cite[Proposition 3.3]{derksen2010quivers}}]\label{prop: rep h vector from f poly}
    Under the assumption that $F(M)\in \mathbb Q_\mathrm{sf}(y_1, \dots, y_n)$ (the semifield of subtraction-free rational expressions), the vector $\mathbf h(M) = (h_1, \dots, h_n)$ satisfies the equation
    \begin{equation}
        x_1^{h_1} \dots x_n^{h_n} = F(M) \mid_{\mathrm{Trop}}(x_i^{-1}\prod_{j\neq i} x_j^{[-b_{j,i}]_+} \leftarrow y_i).
    \end{equation}
\end{proposition}

\begin{proof}
    It follows almost exactly from the proof in \cite{derksen2010quivers} except that now we are taking locally free subrepresentations instead of the usual ones.
\end{proof}

\section{\texorpdfstring{$\tau$-rigid modules and $\tau$-tilting theory}{tau-rigid modules and tau-tilting theory}}\label{section: tau tilting}

Not every locally free module $M$ has again locally free mutations, see \Cref{lemma: when is mutation locally free} and \Cref{subsection: example c3}. However, we will see that $\tau$-rigid $\mathcal P(T)$-modules are locally free and are preserved under mutations. In this section, we review basics of $\tau$-rigid modules and the $\tau$-tilting theory of Adachi--Iyama--Reiten \cite{adachi2014tau}.

\subsection{\texorpdfstring{$\tau$-rigid modules}{tau-rigid modules}}
Let $A$ be a finite-dimensional $\Bbbk$-algebra where $\Bbbk$ is some ground field and $\modu A$ denote the category of finitely generated left $A$-modules. There are dualities
\[
    D \coloneqq \Hom_\Bbbk(-, \Bbbk) \colon \modu A \leftrightarrow \modu A^{\mathrm{op}}\quad \text{and} \quad (-)^*\coloneqq \Hom_A(-, A)\colon \operatorname{proj} A \leftrightarrow \operatorname{proj} A^\mathrm{op}.
\]
They induce equivalences (called \emph{Nakayama functors})
\[
    \nu \coloneqq D(-)^* \colon \operatorname{proj} A \rightarrow \operatorname{inj} A \quad \text{and} \quad \nu^{-1} \coloneqq (-)^*D \colon \operatorname{inj} A \rightarrow \operatorname{proj} A.
\]
For $X\in \modu A$ with a minimal projective presentation (resp. minimal injective copresentation)
\[
    \begin{tikzcd}
        P_1 \ar[r, "d_1"] & P_0 \ar[r, "d_0"] & X \ar[r] & 0
    \end{tikzcd}
    \quad (\text{resp.}
    \begin{tikzcd}
        0 \ar[r] & X \ar[r, "d_0"] & I_0 \ar[r, "d_1"] & I_1
    \end{tikzcd}
    ),
\]
we define $\tau X \in \modu A$ (resp. $\tau^{-1}X\in \modu A$) by exact sequences
\[
    \begin{tikzcd}[column sep = normal]
        0 \ar[r] & \tau X \ar[r] & \nu P_1 \ar[r, "\nu d_1"] & \nu P_0
    \end{tikzcd}
    \quad (\text{resp.}
    \begin{tikzcd}
        \nu^{-1} I_0 \ar[r, "\nu^{-1}d_1", shorten >= -0.3ex] & \nu^{-1} I_1 \ar[r] & \tau^{-1} X \ar[r] & 0
    \end{tikzcd}
    ).
\]
The module $\tau X$ (resp. $\tau^{-1}X$) is called the Auslander-Reiten (AR) (resp. inverse) translation of $X$.

\begin{definition}
    An $A$-module $M$ is called $\tau$-rigid if $\Hom_A(M, \tau M) = 0$ and is called $\tau^{-1}$-rigid if $\Hom_A(\tau^{-1}M, M) = 0$.
\end{definition}

Any $\tau$-rigid (or $\tau^{-1}$-rigid) module $M$ is rigid, i.e. $\operatorname{Ext}^1_A(M, M) = 0$ by the AR duality, which says $\underline{\Hom}_A(X, Y) \cong D\operatorname{Ext}_A^1(Y, \tau X)$ for any $X, Y\in \modu A$ where $\underline{\Hom}_A(X, Y)$ is the quotient of $\Hom_A(X, Y)$ by the subspace of morphisms factoring through projective modules. The AR duality also leads to the following proposition.

\begin{proposition}[\cite{MR0617088}] \label{prop: characterize tau rigid}
    Let $M$ and $N$ be a pair of modules in $\modu A$. We have
    \begin{enumerate}
        \item $\Hom_A (\tau^{-1}M, N) = 0$ if and only if $\operatorname{Ext}^1_A(N', M) = 0$ for all submodules $N'$ of $N$;
        \item $\Hom_A (N, \tau M) = 0$ if and only if $\operatorname{Ext}^1_A(M, N'') = 0$ for all factor modules $N''$ of $N$.
    \end{enumerate}
\end{proposition}

\begin{lemma}\label{lemma: tau rigid implies loc free}
    Any $\tau$-rigid (or $\tau^{-1}$-rigid) $\mathcal P(T)$-module is locally free.
\end{lemma}

\begin{proof}
    In fact, for the algebra $\mathcal P(T)$ arising from a triangulation of an orbifold, we show a stronger statement that any rigid module is locally free. It suffices to show that any non-locally free $M$ is not rigid, that is, has a non-trivial self-extension. Let $k\in T$ such that $M(k)$ is not free over $H_k$. Then there is a non-trivial extension $0\rightarrow M(k) \rightarrow N \rightarrow M(k) \rightarrow 0$. Choose a splitting $N = M(k) \oplus M(k)$ as vector spaces. Define $M'$ a $\mathcal P(T)$-module by the tuple of $H_i$-modules $(M'(i))_i$ where $M'(i) = M(i) \oplus M(i)$ for $i\neq k$ and $M'(k) = N$ and for $(i, j)\in \Omega(T)$, the structure ($\Bbbk$-linear) maps 
    \[
        M'(a_{(i, j)}) = \begin{bmatrix}
        M(a_{(i, j)}) & 0 \\
        0 & M(a_{(i, j)})
        \end{bmatrix} \colon M'(j) \longrightarrow M'(i).
    \]
    The $\mathcal P(T)$-module $M'$ is a non-trivial self-extension of $M$.
\end{proof}

\begin{lemma}\label{lemma: tau rigid free kernel}
    For any $\tau$-rigid $\mathcal P(T)$-module $M$ (hence locally free by \Cref{lemma: tau rigid implies loc free}) and any $k\in T$, the image of the $H_k$-morphism $\alpha_k \colon M_{\mathrm{in}}(k) \rightarrow M(k)$ is free over $H_k$.
\end{lemma}

\begin{proof}
    If $\image \alpha_k$ is not free, then the simple module $S_k$ is a quotient of $M$. In view of \Cref{prop: characterize tau rigid}, we construct below a non-trivial extension from $M$ to $S_k$ to derive a contradiction.

    Let $M'(i) = M(i)$ for $i\neq k$ and $M'(k) = M(k) \oplus S \in \modu H_k$ where $S = H_k/(\varepsilon_k)$. Choose a splitting (in $\modu H_k$) $\image \alpha_k = F \oplus T$ such that $F$ is a maximal free submodule. Let $\pi \colon \image \alpha_k \rightarrow S$ be a surjective $H_k$-morphism such that the restriction on $T$ is non-zero. We define
    \[
        \alpha_k' \coloneqq (\alpha_k, \pi\alpha_k)^\intercal \colon M_{\mathrm{in}}(k) \rightarrow M'(k) \quad \text{and} \quad
        \beta_k' \coloneqq (\beta_k, 0) \colon M'(k) \rightarrow M_{\mathrm{out}}(k).
    \]
    These two maps together with all other $M'(i, j) \coloneqq M(i, j)$ for $i, j\neq k$ define a $\mathcal P(T)$-module structure on $M' = (M'(i))_i$.

    The submodule $S\subset M'(k)$ defines an inclusion $S_k\subset M'$ with $M$ the quotient. Now we have the exact sequence $\begin{tikzcd}[column sep = scriptsize]
        0 \ar[r] & S_k \ar[r, "\iota"] & M' \ar[r] & M \ar[r] & 0    
    \end{tikzcd}$. We show that this is a non-trivial extension from $M$ to $S_k$.

    Suppose this is a retraction $\rho \colon M' \rightarrow S_k$ such that $\rho\iota = \mathrm{id}$. Its restriction on $M'(k)$ is given by $(f, \mathrm{id}) \colon M(k) \oplus S \rightarrow S$. The kernel of $(f, \mathrm{id})$ must contain $\image \alpha_k(M')$, that is 
    \[
    	f(\alpha_k(m)) + \pi \alpha_k (m) = 0 \quad \text{for any $m\in M_{\mathrm{in}}(k)$}.
    \]
    Let $\alpha_k(m)\in T$ such that $\pi\alpha_k(m) \neq 0$. However $f(\alpha_k(m))$ must be zero. In fact, since $T$ is contained in $\varepsilon_kM(k)$, the element $\alpha_k(m)$ must be annihilated by the $H_k$-morphism $f\colon M(k) \rightarrow S_k$. This contradiction implies the extension from $M$ to $S_k$ we have constructed is non-trivial.
\end{proof}

The following lemma is dual to \Cref{lemma: tau rigid free kernel}. We give a direct proof for completeness.
\begin{lemma}\label{lemma: tau inverse rigid free cokernel}
    For any $\tau^{-1}$-rigid $\mathcal P(T)$-module $M$ (hence locally free by \Cref{lemma: tau rigid implies loc free}) and any $k\in T$, the kernel of the $H_k$-morphism $\beta_k \colon M(k) \rightarrow M_{\mathrm{out}}(k)$ is free over $H_k$.
\end{lemma}

\begin{proof}
    Suppose that $\ker \beta_k$ is not free, then $M$ has $S_k$ as a submodule. We show that there is a non-trivial extension from $S_k$ to $M$, which contradicts with $M$ being $\tau^{-1}$-rigid by \Cref{prop: characterize tau rigid}.

    Since $\beta_k$ is not surjective, the restriction of $\beta_k$ on $\operatorname{Socle} M(k)$ to $\operatorname{Socle} M_{\mathrm{out}}(k)$ is also not surjective. Let $S = H_k/(\varepsilon_k)\in \modu H_k$ and choose $f\colon S\rightarrow M_{\mathrm{out}}(k)$ such that $f(S)\not\subset \beta_k(\operatorname{Socle} M(k))$.
    
    Let $M'(i) = M(i)$ for $i\neq k$ and $M'(k) = M(k) \oplus S \in \modu H_k$. Define
    \[
    	\beta'_k\coloneqq (\beta_k, f) \colon M'(k) \rightarrow M_{\mathrm{out}}(k) \quad \text{and} \quad \alpha_k'\coloneqq (\alpha_k, 0)^\intercal \colon M_{\mathrm{in}}(k) \rightarrow M'(k).
    \]
    These maps together with $M'(i, j)\coloneqq M(i, j)$ define a $\mathcal P(T)$-module structure on the tuple $M' \coloneqq (M(i))_i$. There is a short exact sequence
    \begin{equation}
    	\begin{tikzcd}
    		0 \ar[r] & M \ar[r] & M' \ar[r, "\pi"] & S_k \ar[r] & 0
    	\end{tikzcd}
    \end{equation}
    where $\pi$ is given by the projection from $M'(k)$ onto $S$.
    
    We show the above exact sequence does not split, hence a non-trivial extension from $S_k$ to $M$. Suppose there is a section of $\pi$ given by an $H_k$-morphism $(g, \mathrm{id})^\intercal\colon S \rightarrow M'(k)$. Let $a$ be a generator of $S$. Then $(g(a), a)$ must be in $\ker \beta_k'$, that is $\beta_k(g(a)) + f(a) = 0$. Notice that $g(a)\in \operatorname{Socle} M(k)$. However we have chosen $f$ such that $f(a)\notin \beta_k(\operatorname {Socle} M(k))$, a contradiction.
\end{proof}

\subsection{\texorpdfstring{$\tau$-tilting theory}{tau-tilting theory}}\label{subsection: tau tilting theory}

A remarkable feature of $\tau$-rigid modules is that they can be organized in a pattern analogous to clusters. There is also a mutation operation in this setup. Let $A$ be a basic finite-dimensional $\Bbbk$-algebra of rank $n$.

\begin{definition}
	A pair $(M, P)$ (resp. $(M, I)$) of $M\in \modu A$ and $P\in \operatorname{proj} A$ (resp. $I\in \operatorname{inj} A$) is called \emph{$\tau$-rigid} (resp. \emph{$\tau^{-1}$-rigid}) if $M$ is $\tau$-rigid and $\Hom_A(P, M) =0$ (resp. $\Hom_A(M, I) = 0$).
\end{definition}

\begin{definition}
    A $\tau$-rigid pair $(M, P)$ is called \emph{support $\tau$-tilting} (resp. \emph{almost complete support $\tau$-tilting}) if $|M| + |P| = n$ (resp. $|M| + |P| = n-1$) where $|\cdot |$ counts the number of non-isomorphic indecomposable direct summands.
\end{definition}

\begin{theorem}[{\cite[Theorem 0.4]{adachi2014tau}}]\label{thm: air mutation}
    Any basic almost complete support $\tau$-tilting pair can be completed to precisely two basic support $\tau$-tilting pairs.
\end{theorem}

Let $\operatorname{s\tau-tilt} A$ denote the set of all basic support $\tau$-tilting pairs of $A$ up to isomorphism. \Cref{thm: air mutation} defines the Adachi--Iyama--Reiten (AIR) \emph{mutation} operation on $\operatorname{s\tau-tilt} A$ that given an indecomposable summand $X$ of $\mathcal M = (M, P)\in \operatorname{s\tau-tilt} A$, there is a unique indecomposable $\tau$-rigid pair $X'$ not isomorphic to $X$ such that $\mu_X(\mathcal M) \coloneqq X' \oplus \mathcal M/X\in \stau{A}$. Therefore the set $\operatorname{s\tau-tilt} A$ can be endowed with a graph structure where the vertex set is $\operatorname{s\tau-tilt} A$ and two pairs $\mathcal M$ and $\mathcal M'$ are joined by an edge if they are AIR mutations of each other.

\begin{remark}\label{rmk: turn decorated rep to a pair}
    Let $A = \mathcal P(T)$. A $\tau$-rigid (resp. $\tau^{-1}$-rigid) pair $(M, P)$ can be turned into a decorated representation $(M, V)$ where $V = \bigoplus H_i^{a_i}$ if $P = \bigoplus_i P_i^{a_i}$ (resp. $I = \bigoplus_i I_i^{a_i}$). It follows that $V$ is disjoint from the support of the module $M$. Conversely, a decorated representation $(M, V)$ can be turned into a pair $(M, P(V))$ or $(M, I(V))$.
\end{remark}

\begin{definition}\label{def: tau rigid dec rep}
    A decorated representation $(M, V)$ is said to be $\tau$-rigid (resp. $\tau^{-1}$-rigid) if the pair $(M, P(V))$ (resp. $(M, I(V))$) is $\tau$-rigid (resp. $\tau^{-1}$-rigid).
\end{definition}

\section{\texorpdfstring{The $E$-invariant}{The E-invariant}} \label{section: e invariant}

In this section, we extend the $E$-invariant introduced in \cite{derksen2010quivers} to locally free decorated representations of $\mathcal P(T)$. In this case, the relation to $\tau$-tilting theory is unveiled in \Cref{prop: E invariant AR translation} and \Cref{prop: e rigid is tau rigid}.

\subsection{}
Let $\mathcal M = (M, V)$ and $\mathcal N = (N, W)$ be (locally free) decorated representations of $\mathcal P = \mathcal P(T)$. Denote $\langle M, N \rangle \coloneqq \dim_\Bbbk \Hom_\mathcal P(M, N)$ and $d_i(\mathcal M) = d_i(M) \coloneqq \dim_\Bbbk M(i)$. Recall the $\mathbf g$-vector $\mathbf g_\mathcal M = (g_i(\mathcal M))_i$ defined in (\ref{eq: def g vector}).

\begin{definition}\label{def: e invariant}
The integer-valued function \emph{$E$-invariant} of a pair $(\mathcal M, \mathcal N)$ is defined to be
\[
    E^{\mathrm{inj}}(\mathcal M, \mathcal N) \coloneqq \langle M, N \rangle + \sum_{i=1}^n d_i(\mathcal M) g_i(\mathcal N)\in \mathbb Z.
\]
The value $E(\mathcal M) \coloneqq E^\mathrm{inj}(\mathcal M, \mathcal M)$ is called the \emph{$E$-invariant} of $\mathcal M$. We say that $\mathcal M$ is $E$-rigid if $E(\mathcal M) = 0$.
\end{definition}

The superscript inj in $E^\mathrm{inj}$ indicates the relation with injective resolutions, as we will explain in \Cref{subsection: homological g and E}.

Following \cite[Section 6]{derksen2010quivers}, we say a homomorphism $\varphi\in \Hom_\mathcal P(M, N)$ is \emph{confined to $k$} if $\varphi(m) = 0$ for any $m\in M(\hat k) \coloneqq \bigoplus_{i\neq k}M(i)$. Such homomorphisms form a subspace $\Hom_\mathcal P^{[k]}(M, N) \subset \Hom_\mathcal P(M, N)$. Restricting $\varphi \in \Hom_\mathcal P^{[k]}(M, N)$ to $M(k)$ induces a canonical vector-space isomorphism
\begin{equation}\label{eq: confined hom space}
    \Hom_\mathcal P^{[k]}(M, N) = \Hom_{H_k}(\coker \alpha_{k; M}, \ker \beta_{k; N}),
\end{equation}
since a $\mathcal P$-homomorphism confined to $k$ kills $\image \alpha_{k;M}$ and lands in $\ker \beta_{k; N}$.

Write $\overline {\mathcal P} \coloneqq \mathcal P(\mu_k(T))$.

\begin{lemma}\label{lemma: confined space preserved under mutation}
    Let $(\overline M, \overline V) = \mu_k(\mathcal M)$ and $(\overline N, \overline W) = \mu_k(\mathcal N)$. There is a $\Bbbk$-linear isomorphism
    \begin{equation}\label{eq: invariance of hom space under mutation}
        \Hom_\mathcal P(M, N)/\Hom^{[k]}_\mathcal P(M, N) = \Hom_{\overline{\mathcal P}}(\overline M, \overline N)/\Hom^{[k]}_{\overline {\mathcal P}}(\overline M, \overline N).
    \end{equation}
\end{lemma}

\begin{proof}
    A $\mathcal P$-module $M$ is also a module over the subalgebra
        \[
            \mathcal P_{\hat k, \hat k} \coloneqq \bigoplus_{i, j\neq k} \mathcal P_{i, j} \subset \mathcal P\quad \text{where}\quad \mathcal P_{i,j} \coloneqq e_i\mathcal P e_j.
        \]
    The subspace $M(\hat k) \coloneqq \bigoplus_{i\neq k} M(i)$ is also a $\mathcal P_{\hat k, \hat k}$-module. Let $\rho\colon \Hom_\mathcal P(M, N) \rightarrow \Hom_{\mathcal P_{\hat k, \hat k}}(M(\hat k), N(\hat k))$ be the map restricting a homomorphism to $M(\hat k)$ and $N(\hat k)$.
    
    Consider the subspace
    \begin{multline*}
        \Hom_{\mathcal P(\hat k)}(M, N)\coloneqq \{\varphi\in \Hom_{\mathcal P_{\hat k, \hat k}}(M(\hat k), N(\hat k)) \mid \varphi(\ker \alpha_{k;M})\subset \ker \alpha_{k; N},\\
        \ \varphi(\image \beta_{k; M}) \subset \image \beta_{k; N} \}\subset  \Hom_{\mathcal P_{\hat k, \hat k}}(M(\hat k), N(\hat k)).
    \end{multline*}
    Then $\rho$ induces an isomorphism
    \[
        \Hom_\mathcal P(M, N)/\Hom_\mathcal P^{[k]}(M, N) = \Hom_{\mathcal P(\hat k)}(M, N).
    \]
    Note that we have an isomorphism of algebras $\mathcal P_{\hat k, \hat k} = \overline {\mathcal P}_{\hat k, \hat k}$ (which is straightforward to check per our constructions). In this context, the two modules $M(\hat k)\in \modu {\mathcal P}_{\hat k, \hat k}$ and $\overline M(\hat k)\in \modu \overline {\mathcal P}_{\hat k, \hat k}$ are isomorphic. This leads to the identification of the two spaces
    \[
        \Hom_{\mathcal P_{\hat k, \hat k}}(M(\hat k), N(\hat k)) = \Hom_{\overline{\mathcal P}_{\hat k, \hat k}}(\overline M(\hat k), \overline N(\hat k)).
    \]
    Furthermore, by \Cref{lemma: ker bar alpha = im beta & ker alpha = im bar beta} we have $\ker \alpha_{k; M} = \image \beta_{k; \overline M}$ and $\image \beta_{k; M} = \ker \alpha_{k; \overline M}$ (same for $N$ and $\overline N$). (This is true in the undecorated setting $\overline M = \mu_k(M)$ but is readily seen to still hold for $(\overline M, \overline V) = \mu_k((M, V))$.)

    Then the subspace $\Hom_{\mathcal P(\hat k)}(M, N)$ identifies with $\Hom_{\overline {\mathcal P}(\hat k)}(\overline M, \overline N)$, implying the desired isomorphism. We note that this property requires neither the original representations nor the mutations to be locally free.
\end{proof}

\begin{corollary}\label{lemma: an identity}
    Assume that $M, N, \overline M, \overline N$ are all locally free, then we have
    \begin{equation}\label{eq: hom space dim calc}
        \langle M, N \rangle + \dim ( \coker \alpha_{k; M} ) \cdot h_k(N) = \langle \overline M, \overline N \rangle + \dim (\coker \alpha_{k; \overline M}) \cdot h_k(\overline N).
    \end{equation}
\end{corollary}

\begin{proof}
    Notice that under the assumption, $\coker \alpha_{k; M}$, $\ker \beta_{k; N}$, $\coker \alpha_{k; \overline M}$ and $\ker \beta_{k; \overline N}$ are all free. Then the equation (\ref{eq: hom space dim calc}) follows directly from taking the dimensions of the two sides of (\ref{eq: invariance of hom space under mutation}) and using (\ref{eq: confined hom space}).
\end{proof}

Denote the symmetrizer by $\mathrm{diag}(c_i)_i$, where $c_i = 2$ when $i$ is pending and $c_i = 1$ when $i$ is ordinary.

\begin{theorem}\label{thm: e invariant under mutation}
Let $\overline {\mathcal M} = (\overline M, \overline V) = \mu_k(\mathcal M)$ and $\overline {\mathcal N} = (\overline N, \overline W) = \mu_k (\mathcal N)$. Assume that $\overline {\mathcal M}, \overline{\mathcal N}, \mathcal M, \mathcal N$ are all locally free. Then we have
\[
    E^\mathrm{inj}(\overline {\mathcal M}, \overline {\mathcal N}) - E^\mathrm{inj}(\mathcal M, \mathcal N) = c_kh_k(\overline M)h_k(N) - c_kh_k(M)h_k(\overline N).
\]
\end{theorem}

\begin{proof}
    The proof is similar to the proof of \cite[Theorem 7.1]{derksen2010quivers}. We deduce the claimed equation using \Cref{lemma: an identity}.
    
    Using the equality $\ker \alpha_{k; M} = \image \beta_{k; \overline M}$ (as noted in the proof of \Cref{lemma: confined space preserved under mutation}), we obtain
    \begin{align*}
        \dim (\coker \alpha_{k; M}) & = d_k(M) - \dim M_{\operatorname{in}}(k) + \dim (\ker \alpha_{k; M}) \\
        & = d_k(M) - \sum_{i}[-b_{ik}]_+ d_i(M) + d_k(\overline M) - \dim (\ker \beta_{k, \overline M}) \\
        & \overset{\eqref{eq: define h vector}}{=} c_k h_k(\overline M) + d_k(M) + d_k(\overline M) - \sum_{i} [-b_{ik}]_+d_i(M).
    \end{align*}
    By \Cref{cor: decorated mutation involution} (the involutivity of decorated mutations), we have $\mu_k(\overline M, \overline V) = (M, V)$, hence
    \[
        \dim(\coker \alpha_{k;\overline M}) = c_kh_k(M) + d_k(\overline M) + d_k(M) - \sum_{i}[b_{ik}]_+d_i(\overline M).
    \]
    Now we can rewrite (\ref{eq: hom space dim calc}) as
    \begin{align*}
        \langle \overline M, \overline N \rangle - \langle M, N\rangle & =
        \dim(\coker \alpha_{k;M})\cdot h_k(N) - \dim(\coker \alpha_{k;\overline M})\cdot h_k(\overline N) \\
        &= c_kh_k(\overline M)h_k(N) - c_k h_k(\overline N)h_k(M) \\
        & + (d_k(M) + d_k(\overline M))\cdot (h_k(N) - h_k(\overline N)) \\
        & + \sum_{i\neq k} ([b_{ik}]_+ h_k(\overline N) - [-b_{ik}]_+h_k(N))d_i(M).
    \end{align*}
    Applying $h_k(\overline N) = h_k(N) - g_k(\mathcal{N})$ and the transformation between $\mathbf g(\mathcal N)$ and $\mathbf g(\overline{\mathcal N})$ (\Cref{lemma: recursion of g vector}), the desired equation is easily seen equivalent to the above one.
\end{proof}

\begin{corollary}\label{cor: e invariant under mutation}
    If both $\mathcal{M}$ and $\overline{\mathcal{M}} = \mu_k(\mathcal{M})$ are locally free, then their $E$-invariants (computed with respect to the corresponding underlying associative algebras) are equal, i.e., $E(\mathcal M) = E(\overline{\mathcal{M}})$.    
\end{corollary}

\subsection{}
Denote by $T^\mathrm{op}$ the triangulation given by $T$ on the same orbifold but with the opposite orientation. The quiver $Q(T^\mathrm{op})$ is simply $Q(T)$ with all arrows reversed. Tracing the definition of the algebra $\mathcal P(T)$, we see that $\mathcal P(T^\mathrm{op}) = \mathcal P(T)^\mathrm{op}$. 

For a $\mathcal P(T)$-module $M$, taking the dual $\Hom_\Bbbk(M, \Bbbk)$ gives rise to a $\mathcal P(T^\mathrm{op})$-module $M^\star$. The following proposition is analogous to \cite[Proposition 7.3]{derksen2010quivers}. The proof therein can be adapted word by word to our situation.

\begin{proposition}\label{prop: e invariant opposite rep}
    Let $\mathcal M = (M, V)$ be a locally free decorated $\mathcal P(T)$-representation. We have $E(\mathcal M^\star) = E(\mathcal M)$ where $\mathcal M^\star = (M^\star, V)$.
\end{proposition}

\subsection{}\label{subsection: homological g and E}
In the rest of the section, we express the $E$-invariant in terms of $\tau$ the AR translation. First we give the following homological interpretation of $\mathbf g$-vectors.

\begin{proposition} \label{prop: g vector by inj}
Let $M$ be a locally free $\mathcal P(T)$-module and
\[
    0 \rightarrow M \rightarrow \bigoplus_{k=1}^n I_k^{p_k} \rightarrow \bigoplus_{k=1}^n I_k^{q_k}   
\]
be a minimal injective copresentation where $I_k$ denotes the injective hull of $S_k$ in $\modu \mathcal P(T)$. Then the $\mathbf g$-vector $\mathbf g_M = (g_k)_{k}$ (as defined in (\ref{eq: def g vector})) satisfies $g_k = - p_k + q_k$.
\end{proposition}

\begin{proof}
It follows from \cite[Lemma 3.3]{cerulli2015caldero} that
\[
    p_k = \dim \Hom_\mathcal P(S_k, M) \quad \text{and} \quad q_k = \dim \operatorname{Ext}^1_\mathcal P(S_k, M).
\]
Recall that $g_k \coloneqq \rk \ker \gamma_k - \rk M(k)$. Then it amounts to prove the equation
\begin{equation}\label{eq: g vector homological identity}
    \rk \ker \gamma_k - \rk M(k) = -\dim \Hom_\mathcal P (S_k, M) + \dim \operatorname{Ext}^1_\mathcal P (S_k, M).
\end{equation}

First notice that $\Hom_{\mathcal P}(S_k, M) = \Hom_{H_k} (S_k, \ker \beta_k)$, hence $\dim \Hom_\mathcal P(S_k, M) = \dim \operatorname{soc}(\ker \beta_k)$.

Let $0\rightarrow M \rightarrow M' \rightarrow S_k\rightarrow 0$ be a short exact sequence in $\operatorname{Ext}^1_\mathcal P(S_k, M)$. Then $M'(k) = M(k) \oplus S_k$ as $M(k)$ is free. Now this extension is determined by an $H_k$-morphism
\[
    \varphi \colon S_k \rightarrow M_{\mathrm{out}}(k) \quad \text{such that} \quad \varphi(S_k) \subset \ker \gamma_k 
\]
to satisfy the relations in $\mathcal P$. Let $a$ be a generator of $S_k$. Two extensions $\varphi$ and $\varphi'$ are equivalent in $\operatorname{Ext}^1_\mathcal P(S_k, M)$ if and only if $\varphi(a)$ and $\varphi'(a)$ differ by an element in the image of $\operatorname{soc}M(k)$ under $\beta_k$. Therefore we have
\[
    \operatorname{Ext}_\mathcal P^1(S_k, M) = \dfrac{\Hom_{H_k}(S_k, \ker \gamma_k)}{\beta_k(\operatorname{soc}M(k))}.
\]
Now we can rewrite the right-hand side of (\ref{eq: g vector homological identity}) as
\[
    -\dim \operatorname{soc}(\ker \beta_k) - \dim \beta_k(\operatorname{soc}(M(k)) + \dim \Hom_{H_k}(S_k, \ker \gamma_k).
\]
Since $\ker \gamma_k$ and $\ker \beta_k$ are both free, we have $\dim \Hom_{H_k}(S_k, \ker \gamma_k) = \rk \ker \gamma_k$ and $\dim \operatorname{soc}(\ker \beta_k) + \dim \beta_k(\operatorname{soc}(M(k)) = \dim \operatorname{soc}(M(k)) = \rk M(k)$. Then (\ref{eq: g vector homological identity}) follows.
\end{proof}

\begin{remark}
Proposition \ref{prop: g vector by inj} states that the $\mathbf{g}$-vector we have defined in \eqref{eq: def g vector} equals the vector considered in \cite[(10.15), (10.16), Proposition 10.4]{derksen2010quivers} and later in \cite[\S5]{adachi2014tau} for the vector space dual of $M$, which is a module over the opposite algebra $\mathcal{P}(T)^{\mathrm{op}}$. Similarly, the $E$-invariant we have defined in \Cref{def: e invariant} will turn out to be equal to the one considered in \cite[Theorem 10.5, Corollary 10.9]{derksen2010quivers} and later in \cite[\S5]{adachi2014tau}, as we are about to show.
\end{remark}

The following theorem due to Auslander and Reiten will be useful.

\begin{theorem}[{\cite[Theorem 1.4(b)]{auslander1985modules}}]\label{thm: dim hom space AR}
 Let $M$ and $N$ be representations of a finite-dimensional basic algebra $\Lambda$. Let
 \[
    0 \rightarrow N \to I_0 \to I_1
 \]
 be a minimal injective copresentation of $N$. Then we have
 \[
    \dim \Hom_{\Lambda}(\tau^{-1}N, M) = \dim \Hom_{\Lambda}(M, N) - \dim \Hom_{\Lambda}(M, I_0) + \dim \Hom_{\Lambda} (M, I_1).
 \]
\end{theorem}

We return to the situation when $\Lambda = \mathcal P(T)$ arising from a triangulation $T$.
\begin{proposition}\label{prop: E invariant AR translation}
For locally free decorated representations $\mathcal M = (M, V)$ and $\mathcal N = (N, W)$, we have
\[
    E^\mathrm{inj}(\mathcal M, \mathcal N) = \dim \Hom_\mathcal P(\tau^{-1}N, M) + \sum_{i=1}^n d_i(M) \cdot \rk W_i.
\]
\end{proposition}

\begin{proof}
    Let $0 \rightarrow N \rightarrow \bigoplus_{i=1}^n I_i^{p_i} \rightarrow \bigoplus_{i=1}^n I_i^{q_i}$ be a minimal injective copresentation of $N$. By \Cref{prop: g vector by inj}, we have $g_i(N) = -p_i + q_i$ for $i = 1,\dots, n$. Using \Cref{thm: dim hom space AR}, the right-hand side is expressed as
    \begin{align*}
        &\langle M, N\rangle + \sum_{i = 1}^n d_i(M) (-p_i + q_i) + \sum_{i=1}^n d_i(M) \rk W_i \\
        = & \langle M, N\rangle + \sum_{i = 1}^n d_i(M) g_i(N) + \sum_{i=1}^n d_i(M) \rk W_i \\
        = & \langle M, N\rangle + \sum_{i = 1}^n d_i(M) g_i(\mathcal N),
    \end{align*}
    which by definition is $E^\mathrm{inj}(\mathcal M, \mathcal N)$.
\end{proof}

\begin{corollary}\label{cor: formula proj e invariant}
    Define $E^\mathrm{proj}(\mathcal M, \mathcal N) \coloneqq E^\mathrm{inj}(\mathcal N^\star, \mathcal M^\star)$. Then we have
    \[
        E^\mathrm{proj}(\mathcal M, \mathcal N) = \dim \Hom_\mathcal P(N, \tau M) + \sum_{i = 1}^n d_i(N)\cdot \rk V_i.
    \]
\end{corollary}

\begin{proof} This is a direct corollary of \Cref{prop: E invariant AR translation}. We have
    \begin{align*}
        E^\mathrm{proj}(\mathcal M, \mathcal N) & \coloneqq E^\mathrm{inj}(\mathcal N^\star, \mathcal M^\star) \\
        & = \dim \Hom_{\mathcal P^\mathrm{op}}(\tau^{-1}(M^\star), N^\star) + \sum d_i(N)\cdot \rk V_i \quad \text{(by \Cref{prop: E invariant AR translation})} \\
        & = \dim \Hom_{\mathcal P^\mathrm{op}}((\tau M)^\star, N^\star) + \sum d_i(N)\cdot \rk V_i \\
        & = \dim \Hom_{\mathcal P}(N, \tau M) + \sum d_i(N)\cdot \rk V_i.
    \end{align*}
\end{proof}

The following proposition is a direct corollary of \Cref{prop: E invariant AR translation}, \Cref{cor: formula proj e invariant} and \Cref{prop: e invariant opposite rep}.
\begin{proposition}\label{prop: e rigid is tau rigid}
    Let $\mathcal M = (M, V)$ be a locally free decorated $\mathcal P(T)$-representation. Then the following are equivalent.
    \begin{enumerate}
        \item $\mathcal M$ is $E$-rigid, i.e. $E(\mathcal M) = 0$.
        \item The pair $(M, P(V))$ is $\tau$-rigid.
        \item The pair $(M, I(V))$ is $\tau^{-1}$-rigid.
    \end{enumerate}
\end{proposition}

\section{Applications to cluster algebras associated to orbifolds} \label{section: application cluster algebras}

Let $\mathbb T_n$ be the $n$-regular tree rooted at a vertex $t_0$. For any $t\in \mathbb T_n$, let
\[
    t_0 \frac{k_1}{\quad\quad} t_1 \frac{k_2}{\quad\quad} \dots\dots \frac{k_{p-1}}{\quad\quad} t_{p-1} \frac{k_p}{\quad\quad} t_p = t
\]
be the (unique) shortest path joining $t_0$ and $t$ in $\mathbb T_n$. Let a triangulation $T_0$ be associated with $t_0$, which we call the initial triangulation. This uniquely determines an association of a triangulation to any $t\in \mathbb T_n$ such that $T \coloneqq \mu_{k_p}\dots\mu_{k_1}(T_0)$ is associated to $t$.

A triangulation $T$ supplies an exchange matrix $B(T)$ as well as the algebra $\mathcal P(T)$. Let the arcs in $T$ be labeled by $\{1, \dots, n\}$. For any $\ell\in T$, write $\mathcal E_\ell ^-(T) \coloneqq (0, H_\ell)$, the (locally free) decorated representation of $\mathcal P(T)$ with rank 1 decoration at $\ell$. It is both $\tau$-rigid and $\tau^{-1}$-rigid (see \Cref{def: tau rigid dec rep}).

Applying the sequence of decorated mutations $(\mu_{k_1}, \dots, \mu_{k_p})$, we obtain the decorated $\mathcal P(T_0)$-representation
\begin{equation}\label{eq:def-of-dec-rep-M-ell-t-T0,t0}
    \mathcal M_{\ell;t}^{T_0; t_0} \coloneqq \mu_{k_1}\dots\mu_{k_p}(\mathcal E_\ell^-(T)).
\end{equation}
We will also consider the direct sum
\[
    \mathcal M^{T_0; t_0}_{t} \coloneqq \bigoplus_{\ell = 1}^n \mathcal M_{\ell;t}^{T_0; t_0} = \mu_{k_1}\dots\mu_{k_p}\left(\bigoplus_{\ell = 1}^n \mathcal E_\ell^-(T)\right).
\]

\begin{proposition}\label{prop: locally free mutation preserve}
    Every decorated representation $\mathcal M_{\ell; t}^{T_0; t_0}$ is either of the form $\mathcal E^-_i(T_0)$ for some $i$ or $(M, 0)$ where $M$ is an indecomposable (locally free) $\tau$-rigid and $\tau^{-1}$-rigid $\mathcal P(T_0)$-module. Furthermore $\mathcal M_t^{T_0; t_0}$ is a support $\tau$-tilting (also support $\tau^{-1}$-tilting) pair.
\end{proposition}

\begin{proof}
    We show by induction on the distance from $t$ to $t_0$ that 
    \begin{equation*}
        \mathcal M_{\ell; t}^{T_0; t_0} \ \text{is locally free indecomposable and} \ E(\mathcal M_{\ell; t}^{T_0; t_0}) = 0.
    \end{equation*}
    In the induction, we fix the vertex $t$ and allow the root $t_0$ to move towards $t$ to decrease the distance, as opposed to fixing $t_0$ and moving $t$ towards $t_0$. Thus for the inductive basis, it suffices to notice that $\mathcal{E}_\ell^-(T)$ is always locally free indecomposable and has $E(\mathcal E_\ell^-(T)) = 0$.

    For the inductive step, denote $T_1 = \mu_{k_1}(T_0)$ and assume that the decorated representation $\mathcal M_{\ell; t}^{T_1; t_1}$ is locally free indecomposable and satisfies $E(\mathcal M_{\ell; t}^{T_1; t_1}) = 0$. By \Cref{prop: e rigid is tau rigid}, $\mathcal M_{\ell; t}^{T_1;t_1}$ is both $\tau$-rigid and $\tau^{-1}$-rigid; see \Cref{def: tau rigid dec rep}. By \Cref{lemma: tau rigid free kernel,lemma: tau inverse rigid free cokernel,lemma: when is mutation locally free}, $\mathcal M_{\ell; t}^{T_0; t_0}=\mu_{k_1}(\mathcal M_{\ell; t}^{T_1; t_1})$ is again locally free. In this situation, by \Cref{cor: decorated mutation involution} and the fact that $\mu_k$ opens direct sums of decorated modules, $\mathcal M_{\ell; t}^{T_0; t_0}$ is indecomposable. Now Corollary \ref{cor: e invariant under mutation} asserts $E(\mathcal M_{\ell; t}^{T_0; t_0})= E(\mathcal{M}_{\ell;t}^{T_1;t_1})=0$. This finishes the induction.

    By \Cref{prop: e rigid is tau rigid}, each $\mathcal M_{\ell; t}^{T_0; t_0}$ is both $\tau$-rigid and $\tau^{-1}$-rigid.

    Almost the same inductive argument as above shows that $\mathcal M_t^{T; t_0} = (M, V)$ is $E$-rigid with $n$ indecomposable summands. Thus by \Cref{prop: e rigid is tau rigid}, the module $M$ is both $\tau$-rigid and $\tau^{-1}$-rigid and its support is disjoint from the decoration $V$. In addition $|M| + |V| = n$ by definition. This implies that $\mathcal M_t^{T; t_0}$ is simultaneously support $\tau$-tilting and $\tau^{-1}$-tilting; see the convention in \Cref{def: tau rigid dec rep}. 
\end{proof}

\begin{theorem}\label{thm: AIR mutations and DWZ mutations are compatible}
    Fix two triangulations $T_0$ and $T_1$ of $\surf$ related by the flip of an arc $k\in T_0$. If $\mathcal M$ is a support $\tau$-tilting pair over the algebra $\mathcal{P}(T_0)$, then $\mu_k(\mathcal M)$ is a support $\tau$-tilting pair over $\mathcal{P}(T_1)$. Moreover, whenever two support $\tau$-tilting pairs $\mathcal{M}_1$ and $\mathcal{M}_2$ over $\mathcal{P}(T_0)$ are related by an AIR mutation, then $\mu_{k}(\mathcal{M}_1)$ and $\mu_{k}(\mathcal{M}_2)$ are related by an AIR mutation as support $\tau$-tilting pairs over $\mathcal{P}(T_1)$.
\end{theorem}

\begin{proof}
    This is an immediate consequence of \Cref{prop: locally free mutation preserve} and the connectedness of the graph $\stau{\mathcal{P}(T)}$ of support $\tau$-tilting pairs proved in \cite{fu2023support} for all finite-dimensional gentle algebras.
\end{proof}

According to \Cref{prop: locally free mutation preserve} and \Cref{thm: AIR mutations and DWZ mutations are compatible}, if $t$ and $t'$ are joined by a $k$-labeled edge in $\mathbb T_n$, then (the support $\tau$-tilting pairs) $\mathcal M_t^{T_0;t_0}$ and $\mathcal M_{t'}^{T_0;t_0}$ are related by the AIR mutation at the $k$-th component. From this view, we have defined the graph map
\begin{equation}\label{eq: from tn to stau-tilt}
    \pi \colon \mathbb T_n \rightarrow \operatorname{s\tau-tilt} \mathcal P(T_0), \quad \pi(t) = \mathcal M_t^{T_0; t_0},
\end{equation}
which is a local isomorphism and onto $\operatorname{rs\tau-tilt} \mathcal P(T_0)$, the connected component of $\operatorname{s\tau-tilt} \mathcal P(T_0)$ containing $(0, \mathcal P)$.\footnote{As pointed out in the proof of \Cref{thm: AIR mutations and DWZ mutations are compatible}, $\stau{\mathcal P(T_0)}$ is connected, so $\pi$ is surjective and thus is a covering map of graphs.}

\begin{theorem}\label{thm: main theorem}
For $B = B(T_0)$, any $t\in \mathbb T_n$ and any $\ell\in \{1, \dots, n\}$, we have
\begin{equation}\label{eq: g vector f polynomial equality}
    \mathbf g_{\ell;t}^{B;t_0} = \mathbf g(\mathcal M) \quad \text{and} \quad F_{\ell;t}^{B;t_0} = F(\mathcal M),
\end{equation}
where $\mathcal M = \mathcal M_{\ell;t}^{B;t_0}$.
\end{theorem}

\begin{proof}
By \Cref{prop: locally free mutation preserve}, we have $\mathcal M_{\ell;t}^{B;t_0}$ is locally free, so its locally free $F$-polynomial is well-defined. We prove (\ref{eq: g vector f polynomial equality}) together with $\mathbf h_{\ell; t}^{B; t_0} = \mathbf h(\mathcal M)$ by induction on the distance from $t$ to $t_0$ in $\mathbb T_n$. 

Suppose that (\ref{eq: g vector f polynomial equality}) holds for $t_1$ and $t$, that is,
\[
    \mathbf g_{\ell; t}^{B(T_1); t_1} = \mathbf g(\mathcal M_{\ell; t}^{B(T_1); t_1}) \quad \text{and} \quad F_{\ell; t}^{B(T_1); t_1} = F(\mathcal M_{\ell; t}^{B(T_1); t_1})
\]
and that $\mathbf h_{\ell; t}^{B(T_1); t_1} = \mathbf h(\mathcal M_{\ell; t}^{B(T_1); t_1})$.

Since the recurrence of representation-theoretic $\mathbf g$-vectors given in \Cref{lemma: recursion of g vector} coincides with that of cluster-theoretic $\mathbf g$-vectors given by \eqref{eq: recursion g vector} of \Cref{prop: recurrence of f poly g vector}, we have $\mathbf g_{\ell;t}^{B;t_0} = \mathbf g(\mathcal M)$. In view of the formula $g_k = h_k - h_k'$ in both \Cref{lemma: recursion of g vector} and \Cref{prop: recurrence of f poly g vector}, the $k$-th components of $\mathbf h_{\ell; t}^{B; t_0}$ and of $\mathbf h(\mathcal M)$ are the same. Then by \Cref{prop: recursion of f polynomial} and \eqref{eq: recursion f poly}, we obtain $F_{\ell; t}^{B; t_0} = F(\mathcal M)$, and in particular, the polynomial $F(\mathcal M)(y_1, \dots, y_n)$ is in $\mathbb Q_\mathrm{sf}(y_1, \dots, y_n)$ by induction. Finally according to the definition of $\mathbf h_{\ell; t}^{B; t_0}$ \eqref{eq: def cluster h vector} and by \Cref{prop: rep h vector from f poly}, we have $\mathbf h_{\ell; t}^{B; t_0} = \mathbf h(\mathcal M)$. This finishes the induction.
\end{proof}

A direct corollary of the above theorem is that the graph covering map $\mathbb T_n \rightarrow \mathbf E(\mathcal A_{\bullet}(B))$ (where $t\in \mathbb T_n$ is sent to the associated unlabeled seed) factors through $\pi$ \eqref{eq: from tn to stau-tilt}:
\begin{center}
    \begin{tikzcd}
        \mathbb T_n \ar[r, "\pi"] \ar[rd] & \stau{\mathcal P(T_0)} \ar[d] \\
        & \mathbf E(\mathcal A_\bullet(B))
    \end{tikzcd}
\end{center}
This is because the cluster variables are determined by modules, and the coefficients and the exchange matrix of a seed can be read off from cluster exchange relations. Taking locally free Caldero--Chapoton functions induces a map from $\Delta(\mathcal P(T_0))$, the $\tau$-tilting complex (see \Cref{subsection: intro tau tilting}), to the cluster complex $\Delta(\mathcal A)$ (whose restriction on vertices is just the vertical map in the above diagram). Since non-isomorphic $\tau$-rigid pairs have distinct $\mathbf g$-vectors (which is generally true for any finite-dimensional algebra by \cite{demonet2019tau}), this map is seen to be an isomorphism between simplicial complexes. We summarize relevant statements below.

\begin{corollary}\label{cor: app in cluster algebra}
    For the cluster algebra $\mathcal A = \mathcal A_\bullet(B)$ where $B = B(T_0)$, we have
    \begin{enumerate}
        \item a seed is determined by its cluster;
        \item the (thus well-defined) cluster complex $\Delta(\mathcal A)$ is isomorphic to the $\tau$-tilting complex $\Delta(\mathcal P(T_0))$ such that $X_{\ell;t}^{B;t_0}$ corresponds to $\mathcal M_{\ell;t}^{B;t_0}$;
        \item the cluster variable $X_{\ell; t}^{B; t_0}$ is given by the locally free Caldero--Chapoton function
        \[
            \mathrm{CC}^\mathrm{l.f.}(\mathcal M) \coloneqq x^{\mathbf g(\mathcal M)} \cdot F(\mathcal M)(\hat y_1, \dots \hat y_n)
        \]
        where $\mathcal M = \mathcal M_{\ell;t}^{B;t_0}$.
    \end{enumerate}
\end{corollary}

\section{Arc and band representations}\label{section: arc representation}

\begin{definition}\label{def:generalized-arc}
A \emph{generalized arc} on $\surf$ is a continuous curve $\mycurve:[0,1]\rightarrow \Sigma $ that connects either a pair of points in $\mathbb{M}$, or a point in $\mathbb{M}$ and a point in $\mathbb{O}$, or two distinct points in $\mathbb{O}$, and satisfies the following conditions:
\begin{itemize}
\item except for its endpoints, $\mycurve$ is disjoint from $\partial\Sigma\cup\mathbb{O}$;
\item $\mycurve$ is not homotopically trivial in $\Sigma\setminus(\mathbb{M}\cup\mathbb{O})$ relative to its endpoints;
\item $\mycurve$ is not homotopic in $\Sigma\setminus(\mathbb{M}\cup\mathbb{O})$ relative to its endpoints to a boundary segment of $\Sigma$.
\end{itemize} 
\end{definition}

\begin{definition}\label{def:closed-curve}
A \emph{closed curve} on $\surf$ is a continuous curve $\mycurve:\mathbb{S}^1\rightarrow \Sigma $ that satisfies the following conditions:
\begin{itemize}
\item $\mycurve$ is disjoint from $\partial\Sigma\cup\mathbb{O}$;
\item $\mycurve$ is not homotopically trivial in $\Sigma\setminus(\mathbb{M}\cup\mathbb{O})$;
\item $\mycurve$ is not a loop enclosing a single orbifold point;
\item $\mycurve$ is not homotopic in $\Sigma\setminus(\mathbb{M}\cup\mathbb{O})$ to a curve $\mycurve':\mathbb{S}^1\rightarrow \Sigma$ that can be factored as
$$
\xymatrix{
\mathbb{S}^1 \ar[dr]_{\mycurve'}  \ar[rr]^{z\mapsto z^n} & & \mathbb{S}^1 \ar[dl]^{\mycurve''} \\
& \Sigma & 
}
$$
with $\mycurve''$ continuous and $n>1$.
\end{itemize} 
\end{definition}

Generalized arcs will be considered up to homotopy in $\Sigma\setminus(\mathbb{M}\cup\mathbb{O})$ relative to their endpoints, and up to orientation. Closed curves will be considered up to free homotopy in $\Sigma\setminus(\mathbb{M}\cup\mathbb{O})$, and up to orientation.

Let $\mycurve$ be either a generalized arc or a closed curve on $\surf$. We assume that $\mycurve$ is chosen within its homotopy class so as to minimize its number of self-crossings. Furthermore, given a triangulation $T$ of $\surf$ not containing $\mycurve$, we assume that for each arc $k\in T$, both $\mycurve$ and $k$ are chosen within their respective homotopy classes so as to minimize the number of crossings between $\mycurve$ and $k$.

\begin{definition}\label{def:forbidden-kinks}
    Let $\mycurve$ be either a generalized arc or a closed curve on $\surf$, and let $T$ be a triangulation of $\surf$ not containing $\mycurve$. A kink $\kink$ of $\mycurve$ around an orbifold point $q\in\orbset$ is \emph{forbidden with respect to $T$} if either:
    \begin{itemize}
    \item $\kink$ has multiplicity greater than $1$ (see \cite{geiss2023resolution} for a way to define formally the multiplicity of a kink); or
    \item $\kink$ occurs inside the digon containing $q$, entering and leaving this digon on different sides; or
    \item $\kink$ occurs inside the digon containing $q$, entering or leaving this digon through a marked point (which is then an endpoint of $\mycurve$). 
    \end{itemize}
\end{definition}

\begin{example} Figure \ref{Fig:forbiddenKinks} depicts some forbidden kinks. The relevant digon containing the orbifold point $q$ has been drawn, but, for simplicity, the pending arc surrounding $q$ has been omitted. For kinks that are not forbidden, we refer the reader to Figures \ref{Fig:localConfigsOfSegments} and \ref{Fig:casesForFlipVsMuts}.
\begin{figure}[h]
    \centering
    \includegraphics[scale=.075]{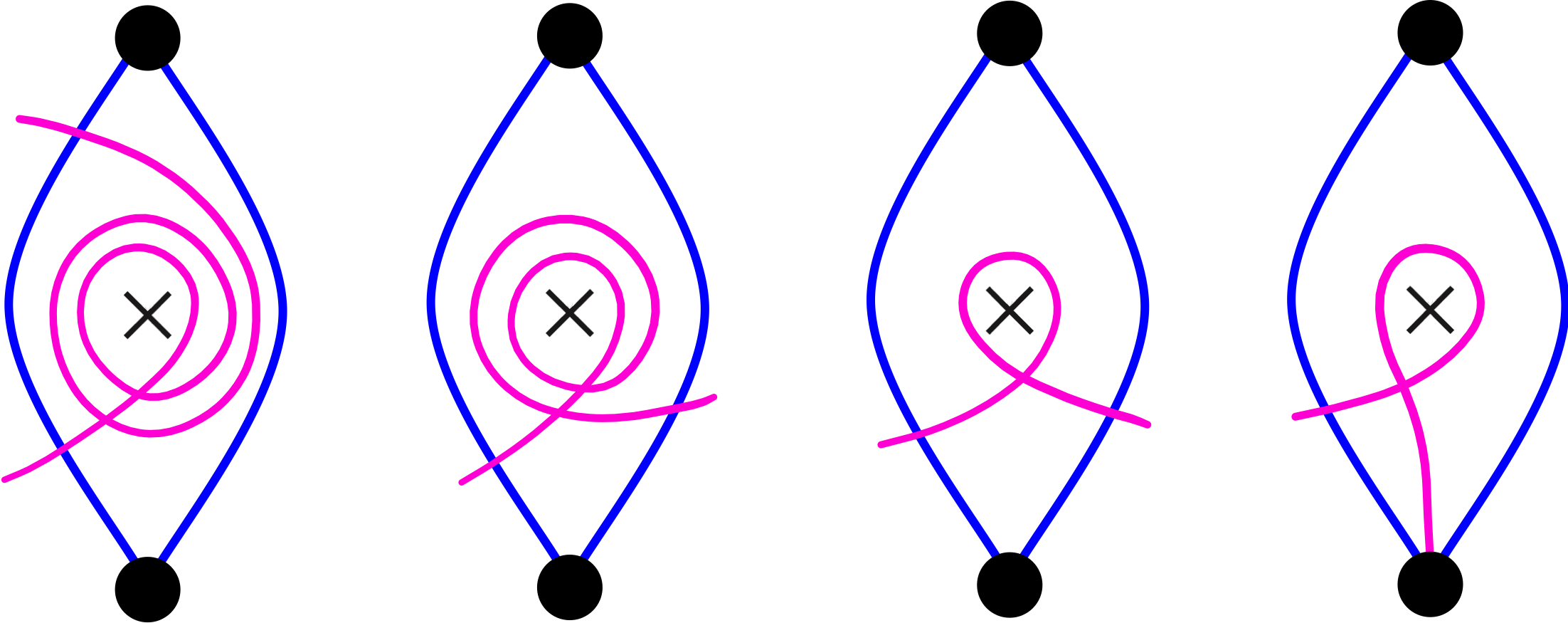}
    \caption{Some forbidden kinks. The two left-most depicted kinks have multiplicity greater than $1$.}\label{Fig:forbiddenKinks}
\end{figure}
\end{example}

Let $\mycurve$ be either a generalized arc or a closed curve on $\surf$, and let $T$ be a triangulation of $\surf$ not containing $\mycurve$. We assume that $\mycurve$ does not have forbidden kinks with respect to $T$. Furthermore, let
$$
U_{\mycurve}:=\begin{cases}
\{1\} & \text{if $\mycurve$ is a generalized arc};\\
\Bbbk\setminus\{0\} & \text{if $\mycurve$ is a closed curve}.
\end{cases}
$$
We fix an embedding of the quiver $Q(T)$ into the surface $\mathcal{O}$ so that each $3$-cycle of $Q$ arising from a regular triangle $\Delta$ is fully contained in $\Delta$ (thus appearing, \emph{a fortiori}, as a counter-clockwisely oriented $3$-cycle), and more importantly, so that for every pending arc $k\in T$, the arrow $\varepsilon_k$ is drawn as a loop oriented in the counter-clockwise sense. This is consistent with the way we have drawn the quivers of the form $Q(T)$ throughout the paper, as the reader can easily notice.

With the chosen embedding of $Q(T)$ into $\mathcal{O}$, there is an induced string (if $\mycurve$ is a generalized arc) or band (if $\mycurve$ is a closed curve), from which one can in turn define a string module $M(T,\mycurve,1)$ (if $\mycurve$ is a generalized arc) or a 1-parameter family $\{M(T,\mycurve,\lambda) \suchthat \lambda\in U_{\mycurve}\}$ of quasi-simple band modules (if $\mycurve$ is a closed curve). See \cite{butler1987Auslander-Reiten} for the definition of string and band modules. The corresponding decorated representation with zero decoration will be denoted $\mathcal{M}(T,\mycurve,\lambda):=(M(T,\mycurve,\lambda),0)$ for $\lambda\in U_{\mycurve}$.

On the other hand, if $T$ contains $\mycurve$, then $\mycurve$ is an arc, and we set $\mathcal{M}(T,\mycurve,1)$ to be the negative pseudo-simple representation whose decoration has rank one at the vertex $\mycurve$.

\begin{theorem}\label{thm:behavior-of-arc-and-band-reps-under-muts}
Let $\mycurve$ be either a generalized arc or a closed curve on $\surf$, and let $T$ and $T'$ be triangulations related by the flip of an arc $k\in T$. We assume that $\mycurve$ does not have forbidden kinks with respect to $T$ or $T'$.
\begin{enumerate}
\item If $k$ is not a pending arc, then there exists an automorphism $g=g_{T,\mycurve}$ of $U_{\mycurve}$ as an algebraic variety such that the decorated representations $\mu_k(\mathcal{M}(T,\mycurve,\lambda))$ and $\mathcal{M}(T',\mycurve,g(\lambda))$ are isomorphic through an isomorphism that acts as the identity on each of the vector spaces attached to the vertices of $Q(T)$ that are different from $k$.
\item If $k$ is a pending arc, surrounding the orbifold point $q\in\orbset$,
then there exists an automorphism $g=g_{T,\mycurve}$ of $U_{\mycurve}$ as an algebraic variety such that the decorated representations $\mu_k(\mathcal{M}(T,\mycurve,\lambda))$ and $\mathcal{M}(T',\mycurve,g(\lambda))$ are isomorphic through an isomorphism that acts as the identity on each of the vector spaces attached to the vertices of $Q(T)$ that are different from $k$.
\item If $\mycurve$ is not a generalized arc with at least one endpoint in $\orbset$, then for any triangulation $T$ with respect to which $\mycurve$ does not have forbidden kinks, the decorated representation $\mathcal{M}(T,\mycurve,\lambda)$ is locally free.
\end{enumerate}
\end{theorem}

Before proving Theorem \ref{thm:behavior-of-arc-and-band-reps-under-muts}, let us emphasize the following particular case.

\begin{corollary}
    Let $\mycurve$ be either a generalized arc with both endpoints in $\marked$, or a closed curve on $\surf$, and let $T$ and $T'$ be triangulations related by the flip of an arc $k\in T$. If $\mycurve$ does not have kinks whatsoever (e.g., if $\mycurve$ does not have self-crossings), then:
    \begin{enumerate}
    \item the decorated representations $\mu_k(\mathcal{M}(T,\mycurve))$ and $\mathcal{M}(T',\mycurve)$ (resp. $\mu_k(\mathcal{M}(T,\mycurve,\lambda))$ and $\mathcal{M}(T',\mycurve,\lambda)$) are isomorphic;
    \item 
    the decorated representations $\mathcal{M}(T,\mycurve)$ and $\mathcal{M}(T',\mycurve)$ (resp. $\mathcal{M}(T,\mycurve,\lambda)$ and $\mathcal{M}(T',\mycurve,\lambda)$) are locally~free.
    \end{enumerate}
\end{corollary}

\begin{proof}[Proof of \Cref{thm:behavior-of-arc-and-band-reps-under-muts}]
Consider the topological closure in $\Sigma$ of the union of the two triangles of $T$ that contain $k$. Let $\diamondsuit$ be its topological interior in $\Sigma\setminus \partial\Sigma$. Thus, if $k$ is not pending, then $\diamondsuit$ is an open quadrilateral; whereas if $k$ is pending, then $\diamondsuit$ is an open digon. 

By continuity of $\mycurve$ as a function from $[0,1]$ or $\mathbb{S}^1$ to $\Sigma$, the inverse image $\mycurve^{-1}(\diamondsuit)$ can certainly be written as a disjoint union of open intervals. Let $(a,b)$ be one such open interval. 
Since $\mycurve$ has been chosen within its homotopy class so as to minimize its number of self-crossings, and since $\mycurve$ does not have forbidden kinks (see \Cref{def:forbidden-kinks}),
\begin{itemize}
\item if $\diamondsuit$ is a quadrilateral, then $\mycurve$ maps $(a,b)$ homeomorphically onto $\mycurve(a,b)\subseteq\mycurve\cap\diamondsuit$;
\item if $\diamondsuit$ is a digon, then either $\mycurve$ maps $(a,b)$ homeomorphically onto $\mycurve(a,b)\subseteq\mycurve\cap\diamondsuit$, or $\mycurve(a,b)$ is a kink contained in $\diamondsuit$ that enters and leaves $\diamondsuit$ on the same side. 
\end{itemize}
This, together with the fact that for each arc $j\in T$, both $\mycurve$ and $j$ have been chosen within their respective homotopy classes so as to minimize the number of crossings between $\mycurve$ and $j$, implies that $\mycurve(a,b)$ coincides with one of the segments represented in Figure \ref{Fig:localConfigsOfSegments}.
\begin{figure}[h]
    \centering
    \includegraphics[scale=.1]{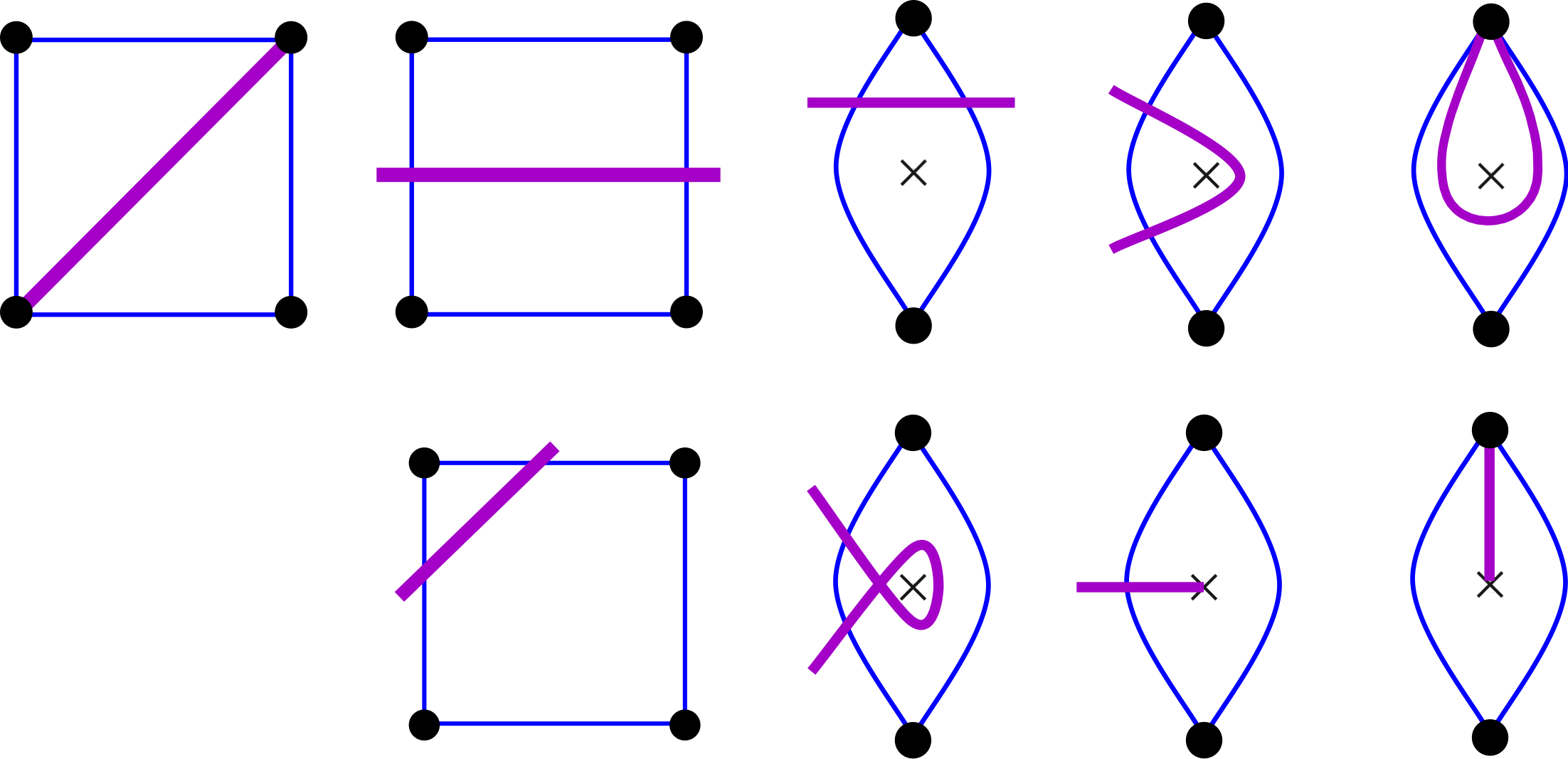}
    \caption{All possibilities for the segment $\mycurve(a,b)$}\label{Fig:localConfigsOfSegments}
\end{figure}

We see that if we draw both $k$ and $\mycurve(a,b)$ inside $\diamondsuit$, we necessarily obtain one of the configurations depicted in Figure \ref{Fig:casesForFlipVsMuts}.
\begin{figure}[h]
    \centering
    \includegraphics[scale=.19]{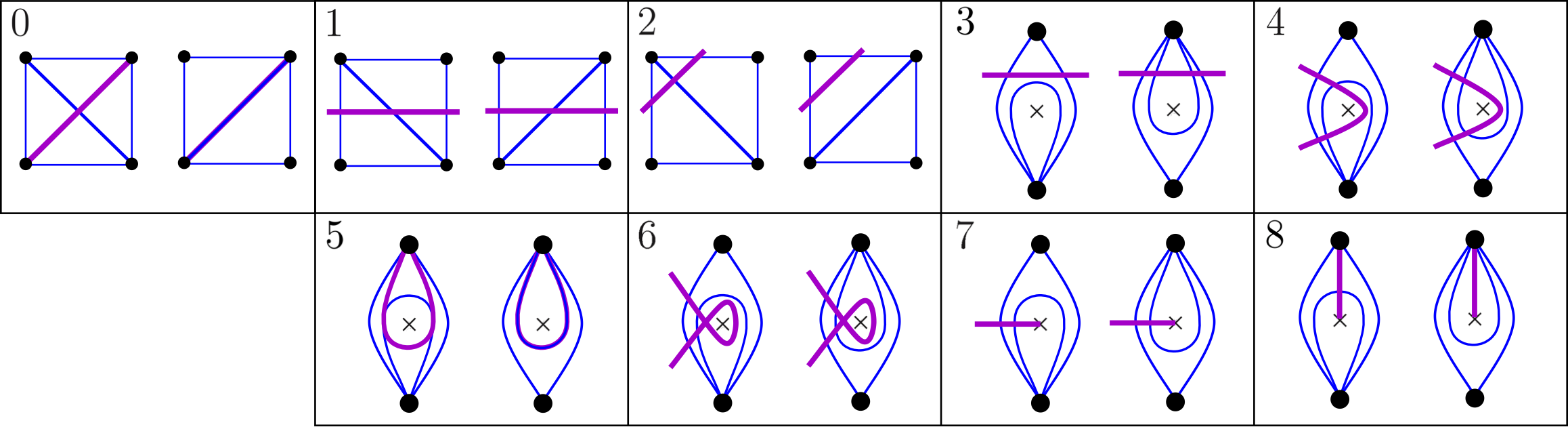}
    \caption{All possibilities for the pair $(\mycurve(a,b),k)$}\label{Fig:casesForFlipVsMuts}
\end{figure}

At this point, the third statement of Theorem \ref{thm:behavior-of-arc-and-band-reps-under-muts} becomes clear.
Furthermore, given the local additivity of mutations of representations,
in order to prove Theorem \ref{thm:behavior-of-arc-and-band-reps-under-muts}, it suffices proving that its statements hold for each of the configurations shown in Figure~\ref{Fig:casesForFlipVsMuts}. This can be verified directly in Tables~\ref{table:proof-of-behavior-of-arc-and-band-reps-under-muts-simply-laced},~\ref{table:proof-of-behavior-of-arc-and-band-reps-under-muts-non-simply-laced-1}~and~\ref{table:proof-of-behavior-of-arc-and-band-reps-under-muts-non-simply-laced-2}.
\end{proof}

\afterpage{%
    \clearpage
    \thispagestyle{empty}
    \begin{landscape}
        \hspace{-1cm}
\renewcommand{\arraystretch}{1.25}
{\tiny
\begin{tabular}{|c|c|c|c|c|c|c|c|c|}
\hline
 & Case 0.a & Case 0.b & Case 1.a & Case 1.b & Case 2.a & Case 2.b \\
\hline
$\begin{array}{cc}
\mathcal{M}(T,\mycurve(a,b),\lambda)\\
\lambda\neq 0
\end{array}$
 &
$
\xymatrix{
 & 0 \ar[d]&  \\
 0 \ar[dr] & \Bbbk  \ar[r] \ar[l] & 0 \ar[ul]\\
 & 0 \ar[u] & 
}
$  
 &
Negative simple
 &
$
\xymatrix{
 & 0 \ar[d]&  \\
 \Bbbk \ar[dr] & \Bbbk  \ar[r]^{\lambda} \ar[l]_{1} & \Bbbk \ar[ul]\\
 & 0 \ar[u] & 
}
$ 
&
$
\xymatrix{
 & 0 \ar[dl] & \\
 \Bbbk \ar[r]^{\lambda} & \Bbbk \ar[u] \ar[d] & \Bbbk \ar[l]_{1} \\
 & 0 \ar[ur] & 
}
$
&
$
\xymatrix{
 & \Bbbk \ar[d]_{\lambda}&  \\
 \Bbbk \ar[dr] & \Bbbk \ar[r] \ar[l]_{1} & 0 \ar[ul]\\
 & 0 \ar[u] & 
}
$
&
$
\xymatrix{
 & \Bbbk \ar[dl]_{\lambda} & \\
 \Bbbk \ar[r] & 0 \ar[u] \ar[d] & 0 \ar[l] \\
 & 0 \ar[ur] & 
}
$
\\
\hline
$
\xymatrix{
& M(k) \ar[dr]^{\beta}& \\
M_{\operatorname{in}}(k)  \ar[ur]^{\alpha} & & M_{\operatorname{out}}(k) \ar[ll]^{\gamma}
}
$
&
$
\xymatrix{
 & \Bbbk \ar[dr] & \\
0 \ar[ur] & & 0 \ar[ll]
}
$
&
$
\xymatrix{
 & 0 \ar[dr] & \\
0 \ar[ur] & & 0 \ar[ll]
}
$
&
$
\xymatrix{
 & \Bbbk \ar[dr]^{\left[\begin{array}{c}\lambda \\ 1\end{array}\right]} & \\
0 \ar[ur] & & \Bbbk^2 \ar[ll]
}
$
&
$
\xymatrix{
 & \Bbbk \ar[dr] & \\
\Bbbk^2 \ar[ur]^{\left[\begin{array}{cc}1 & \lambda\end{array}\right]} & & 0 \ar[ll]
}
$
&
$
\xymatrix{
 & \Bbbk \ar[dr]^{1} & \\
\Bbbk \ar[ur]^{\lambda} & & \Bbbk \ar[ll]^{0} 
}
$
&
$
\xymatrix{
 & 0 \ar[dr] & \\
\Bbbk \ar[ur] & & \Bbbk \ar[ll]^{\lambda}
}
$
\\
\hline
$
\xymatrix{
& \overline{M}(k) \ar[dl]^{\overline{\alpha}}& \\
M_{\operatorname{in}}(k)   & & M_{\operatorname{out}}(k) \ar[ul]^{\overline{\beta}}
}
$
&
$
\xymatrix{
& 0 \ar[dl] & \\
0 & & 0 \ar[ul]
}
$
&
$
\xymatrix{
& \Bbbk \ar[dl] & \\
0 & & 0 \ar[ul]
}
$
& 
$
\xymatrix{
& \Bbbk \ar[dl] & \\
0 & & \Bbbk^2 \ar[ul]_{\left[\begin{array}{cc}1 & -\lambda\end{array}\right]}
}
$
& 
$
\xymatrix{
 & \Bbbk \ar[dl]_{\left[\begin{array}{c}-\lambda \\ 1\end{array}\right]} & \\
 \Bbbk^2 & & 0 \ar[ul]
}
$
& 
$
\xymatrix{
 & 0 \ar[dl] & \\
 \Bbbk & & \Bbbk \ar[ul]
}
$
&
$
\xymatrix{
 & \Bbbk \ar[dl]_{1} & \\
 \Bbbk & & \Bbbk \ar[ul]_{\lambda}
}
$
\\
\hline
$\mu_k(\mathcal{M}(T,\mycurve(a,b),\lambda))$
&
Negative simple
&
$
\xymatrix{
 & 0 \ar[d]&  \\
 0 \ar[dr] & \Bbbk  \ar[r] \ar[l] & 0 \ar[ul]\\
 & 0 \ar[u] & 
}
$   
&
$
\xymatrix{
 & 0 \ar[dl] & \\
 \Bbbk \ar[r]^{-\lambda} & \Bbbk \ar[u] \ar[d] & \Bbbk \ar[l]_{1} \\
 & 0 \ar[ur] & 
}
$
&
$
\xymatrix{
 & 0 \ar[d]&  \\
 \Bbbk \ar[dr] & \Bbbk  \ar[r]^{-\lambda} \ar[l]_{1} & \Bbbk \ar[ul]\\
 & 0 \ar[u] & 
}
$ 
& 
$
\xymatrix{
 & \Bbbk \ar[dl]_{\lambda} & \\
 \Bbbk \ar[r] & 0 \ar[u] \ar[d] & 0 \ar[l] \\
 & 0 \ar[ur] & 
}
$
&
$
\xymatrix{
 & \Bbbk \ar[d]_{\lambda}&  \\
 \Bbbk \ar[dr] & \Bbbk \ar[r] \ar[l]_{1} & 0 \ar[ul]\\
 & 0 \ar[u] & 
}
$
\\
\hline
Automorphism $g_{\mycurve}:U_{\mycurve}\to U_{\mycurve}$ & $\myid$ & $\myid$ & $\lambda\mapsto-\lambda$ & $\lambda\mapsto-\lambda$ & $\myid$ & $\myid$
\\
\hline
\end{tabular}}
       \captionof{table}{Proof of Theorem \ref{thm:behavior-of-arc-and-band-reps-under-muts} for the segments $\mycurve(a,b)$}
       \label{table:proof-of-behavior-of-arc-and-band-reps-under-muts-simply-laced}
    \end{landscape}
    \clearpage
}

\afterpage{%
    \clearpage
    \thispagestyle{empty}
    \begin{landscape}
        \hspace{-3.575cm}
\renewcommand{\arraystretch}{1.05}
{\tiny
\begin{tabular}{|c|c|c|c|c|c|c|c|c|}
\hline
 & Case 3.a & Case 3.b & Case 4.a & Case 4.b & Case 5.a & Case 5.b \\
\hline
$\begin{array}{cc}
\mathcal{M}(T,\mycurve(a,b),\lambda)\\
\lambda\neq 0
\end{array}$
 &
 $
\xymatrix{
\Bbbk \ar[dr] & & \Bbbk \ar[ll]_{\lambda} \\
& 0 \ar[ur] &
}$ 
&
$
\xymatrix{
 & \Bbbk^2 \ar@(ur,ul)_{\left[\begin{array}{cc}0 &0\\ 1 & 0\end{array}\right]} \ar[dl]_{\left[\begin{array}{cc}0 & 1\end{array}\right]} & \\
 \Bbbk \ar[rr]_{0} & & \Bbbk \ar[ul]_{\left[\begin{array}{c}\lambda \\ 0\end{array}\right]}
}$ 
 &
 $
\xymatrix{
\Bbbk^2  \ar[dr]_{\left[\begin{array}{cc}1 &0\\ 0 & \lambda\end{array}\right]} & & 0 \ar[ll]\\
 & \Bbbk^2 \ar@(dl,dr)_{\left[\begin{array}{cc}0 &0\\ 1 & 0\end{array}\right]} \ar[ur] & 
}
$ 
&
$
\xymatrix{
 & \Bbbk^2  \ar@(ur,ul)_{\left[\begin{array}{cc}0 &0\\ 1 & 0\end{array}\right]} \ar[dl]_{\left[\begin{array}{cc}\lambda &0\\ 0 & 1\end{array}\right]} & \\
 \Bbbk^2 \ar[rr] & & 0 \ar[ul]
}
$
 &
 $
\xymatrix{
0 \ar[rd] & & 0 \ar[ll]\\
 & \Bbbk^2 \ar@(dl,dr)_{\left[\begin{array}{cc}0 &0\\ 1 & 0\end{array}\right]} \ar[ur] & 
}
$ 
&
\begin{tabular}{c}
Negative\\ pseudosimple
\end{tabular}
\\
\hline
$
\xymatrix{
& M(k) \ar[dr]^{\beta}& \\
M_{\operatorname{in}}(k)  \ar[ur]^{\alpha} & & M_{\operatorname{out}}(k) \ar[ll]^{\gamma}
}
$
&
$
\xymatrix{
 & 0 \ar[dr] & \\
 \frac{\Bbbk[\varepsilon]}{\varepsilon^2}\smallotimesk \Bbbk \ar[ur] & & \frac{\Bbbk[\varepsilon]}{\varepsilon^2}\smallotimesk \Bbbk \ar[ll]^{\lambda\myid}
}
$
&
$
\hspace{-0.2cm}\xymatrix{
 & \frac{\Bbbk[\varepsilon]}{\varepsilon^2} \ar[dr]^{\myid} \\
 \frac{\Bbbk[\varepsilon]}{\varepsilon^2}\smallotimesk \Bbbk \ar[ur]^{\lambda\myid} & & \frac{\Bbbk[\varepsilon]}{\varepsilon^2}\smallotimesk \Bbbk \ar[ll]^{0}
}\hspace{-0.2cm}
$
&
$
\hspace{-0.2cm}
\xymatrix{
 & \frac{\Bbbk[\varepsilon]}{\varepsilon^2} \ar[dr] & \\
 \frac{\Bbbk[\varepsilon]}{\varepsilon^2}\smallotimesk \Bbbk^2 \ar[ur]^{\left[\hspace{-0.15cm}\begin{array}{cccc}1 & 0 & 0 & 0 \\ 0 & 1 & \lambda & 0\end{array}\hspace{-0.15cm}\right]} & & 0 \ar[ll]
}
\hspace{-0.2cm}
$
&
$
\xymatrix{
 & \frac{\Bbbk[\varepsilon]}{\varepsilon^2} \ar[dr]^{\left[\begin{array}{cc}0 & 0\\ \lambda & 0\\ 1 & 0\\ 0 & 1\end{array}\right]} \\
 0 \ar[ur] & & \frac{\Bbbk[\varepsilon]}{\varepsilon^2}\smallotimesk \Bbbk^2 \ar[ll]
}
$
&
$
\hspace{-0.2cm}
\xymatrix{
& \frac{\Bbbk[\varepsilon]}{\varepsilon^2} \ar[dr] & \\
0  \ar[ur] & & 0 \ar[ll]
}\hspace{-0.2cm}
$
&
$
\xymatrix{
& 0 \ar[dr] & \\
0  \ar[ur] & & 0 \ar[ll]
}
$
\\
\hline
$
\xymatrix{
& \overline{M}(k) \ar[dl]_{\overline{\beta}}& \\
M_{\operatorname{in}}(k)   & & M_{\operatorname{out}}(k) \ar[ul]_{\overline{\alpha}}
}
$
&
$
\hspace{-0.2cm}\xymatrix{
 & \frac{\Bbbk[\varepsilon]}{\varepsilon^2}\smallotimesk \Bbbk \ar[dl]_{\myid} & \\
 \frac{\Bbbk[\varepsilon]}{\varepsilon^2}\smallotimesk \Bbbk & & \frac{\Bbbk[\varepsilon]}{\varepsilon^2}\smallotimesk \Bbbk \ar[ul]_{\lambda\myid} 
}\hspace{-0.2cm}
$
&
$
\xymatrix{
 & 0 \ar[dl] & \\
 \frac{\Bbbk[\varepsilon]}{\varepsilon^2}\smallotimesk \Bbbk & & \frac{\Bbbk[\varepsilon]}{\varepsilon^2}\smallotimesk \Bbbk \ar[ul] 
}
$
& 
$
\xymatrix{
& \frac{\Bbbk[\varepsilon]}{\varepsilon^2} \ar[dl]_{\left[\begin{array}{cc}0 & 0\\ -\lambda & 0 \\ 1 & 0\\ 0 & 1\end{array}\right]} & \\
\frac{\Bbbk[\varepsilon]}{\varepsilon^2}\smallotimesk \Bbbk^2 & & 0 \ar[ul]
}
$
& 
$
\hspace{-0.2cm}
\xymatrix{
 & \frac{\Bbbk[\varepsilon]}{\varepsilon^2} \ar[dl] & \\
 0 & & \frac{\Bbbk[\varepsilon]}{\varepsilon^2}\smallotimesk \Bbbk^2 \ar[ul]_{\left[\hspace{-0.15cm}\begin{array}{cccc}1 & 0 & 0 & 0\\ 0 & 1 & -\lambda & 0\end{array}\hspace{-0.15cm}\right]}
}
\hspace{-0.2cm}
$
& 
$
\xymatrix{
& 0 \ar[dl] & \\
0   & & 0 \ar[ul]
}
$
&
$
\hspace{-0.2cm}
\xymatrix{
& \frac{\Bbbk[\varepsilon]}{\varepsilon^2} \ar[dl] & \\
0   & & 0 \ar[ul]
}
\hspace{-0.2cm}
$
\\
\hline
$\mu_k(\mathcal{M}(T,\mycurve(a,b),\lambda))$
&
$
\xymatrix{
 & \Bbbk^2 \ar@(ur,ul)_{\left[\begin{array}{cc}0 & 0 \\ 1 & 0\end{array}\right]} \ar[dl]_{\left[\begin{array}{cc}0 & 1\end{array}\right]} & \\
 \Bbbk \ar[rr]_{0} & & \Bbbk \ar[ul]_{\left[\begin{array}{c}\lambda \\ 0\end{array}\right]}
}
$
&
$
\xymatrix{
 & 0 \ar@(ur,ul) \ar[dl] & \\
 \Bbbk \ar[rr]_{\lambda} & & \Bbbk \ar[ul]
}
$ 
&
$
\xymatrix{
 & \Bbbk^2 \ar@(ur,ul)_{\left[\begin{array}{cc}0 & 0 \\ 1 & 0\end{array}\right]} \ar[dl]_{\left[\begin{array}{cc}-\lambda & 0 \\ 0 & 1\end{array}\right]}& \\
 \Bbbk^2 \ar[rr] & & 0 \ar[ul]
}
$
&
$
\xymatrix{
\Bbbk^2  \ar[dr]_{\left[\begin{array}{cc}1 &0\\ 0 & -\lambda\end{array}\right]} & & 0 \ar[ll]\\
 & \Bbbk^2 \ar@(dl,dr)_{\left[\begin{array}{cc}0 &0\\ 1 & 0\end{array}\right]} \ar[ur] & 
}
$ 
& 
\begin{tabular}{c}
Negative\\ pseudosimple
\end{tabular}
&
$
\xymatrix{
0 \ar[rd] & & 0 \ar[ll]\\
 & \Bbbk^2 \ar@(dl,dr)_{\left[\begin{array}{cc}0 &0\\ 1 & 0\end{array}\right]} \ar[ur] & 
}
$ 
\\
\hline
Automorphism $g_{\mycurve}:U_{\mycurve}\to U_{\mycurve}$ & $\myid$ & $\myid$ & $\lambda\mapsto-\lambda$ & $\lambda\mapsto-\lambda$ & $\myid$ & $\myid$
\\
\hline
\end{tabular}}
       \captionof{table}{Proof of Theorem \ref{thm:behavior-of-arc-and-band-reps-under-muts} for the segments $\mycurve(a,b)$}
       \label{table:proof-of-behavior-of-arc-and-band-reps-under-muts-non-simply-laced-1}
    \end{landscape}
    \clearpage
}

\afterpage{%
    \clearpage
    \thispagestyle{empty}
    \begin{landscape}
        \hspace{-3.6cm}
\renewcommand{\arraystretch}{1.1}
{\tiny
\begin{tabular}{|c|c|c|c|c|c|c|c|c|}
\hline
 & Case 6.a & Case 6.b & Case 7.a & Case 7.b & Case 8.a & Case 8.b \\
\hline
$\begin{array}{cc}
\mathcal{M}(T,\mycurve(a,b),\lambda)\\
\lambda\neq 0
\end{array}$
 &
$
\xymatrix{
\Bbbk^2  \ar[dr]_{\left[\begin{array}{cc}1 &0\\ 0 & \lambda\end{array}\right]} & & 0 \ar[ll]\\
 & \Bbbk^2 \ar@(dl,dr)_{\left[\begin{array}{cc}0 &1\\ 0 & 0\end{array}\right]} \ar[ur] & 
}
$ 
&
$
\xymatrix{
 & \Bbbk^2  \ar@(ur,ul)_{\left[\begin{array}{cc}0 &1\\ 0 & 0\end{array}\right]} \ar[dl]_{\left[\begin{array}{cc}\lambda &0\\ 0 & 1\end{array}\right]} & \\
 \Bbbk^2 \ar[rr] & & 0 \ar[ul]
}
$
&
$
\xymatrix{
\Bbbk  \ar[dr]_{1} & & 0 \ar[ll]\\
 & \Bbbk \ar@(dl,dr)_{0} \ar[ur] & 
}
$ 
&
$
\xymatrix{
 & \Bbbk  \ar@(ur,ul)_{0} \ar[dl]_{1} & \\
 \Bbbk \ar[rr] & & 0 \ar[ul]
}
$
 &
$
\xymatrix{
0 \ar[rd] & & 0 \ar[ll]\\
 & \Bbbk \ar@(dl,dr)_{0} \ar[ur] & 
}
$ 
&
$
\text{Negative simple}
$ 
\\
\hline
$
\xymatrix{
& M(k) \ar[dr]^{\beta}& \\
M_{\operatorname{in}}(k)  \ar[ur]^{\alpha} & & M_{\operatorname{out}}(k) \ar[ll]^{\gamma}
}
$
&
$
\xymatrix{
& \frac{\Bbbk[\varepsilon]}{\varepsilon^2}  \ar[dr] & \\
\frac{\Bbbk[\varepsilon]}{\varepsilon^2}\smallotimesk \Bbbk^2 \ar[ur]^{\left[\begin{array}{cccc}0 & 1 & \lambda & 0\\ 0 & 0 & 0 & \lambda\end{array}\right]} & & 0 \ar[ll]
}
$
&
$
\xymatrix{
 & \frac{\Bbbk[\varepsilon]}{\varepsilon^2} \ar[dr]^{\left[\begin{array}{cc}\lambda & 0\\ 0 & \lambda\\ 0 & 1\\ 0 & 0\end{array}\right]} & \\
 0 \ar[ur] & & \frac{\Bbbk[\varepsilon]}{\varepsilon^2}\smallotimesk \Bbbk^2 \ar[ll]
}
$
&
$
\xymatrix{
 & \Bbbk \ar[dr] & \\
 \frac{\Bbbk[\varepsilon]}{\varepsilon^2}\smallotimesk \Bbbk \ar[ur]^{\left[\begin{array}{cc}1 & 0\end{array}\right]} & & 0 \ar[ll]
}
$
&
$
\xymatrix{
& \Bbbk \ar[dr]^{\left[\begin{array}{c}0 \\ 1\end{array}\right]} & \\
0 \ar[ur] & & \frac{\Bbbk[\varepsilon]}{\varepsilon^2}\smallotimesk \Bbbk \ar[ll]
}
$
&
$
\xymatrix{
 & \Bbbk \ar[dr] & \\
 0 \ar[ur] & & 0 \ar[ll]
}
$
&
$
\xymatrix{
& 0 \ar[dr] & \\
0 \ar[ur] & & 0 \ar[ll]
}
$
\\
\hline
$
\xymatrix{
& \overline{M}(k) \ar[dl]_{\overline{\beta}}& \\
M_{\operatorname{in}}(k)   & & M_{\operatorname{out}}(k) \ar[ul]_{\overline{\alpha}}
}
$
&
$
\xymatrix{
 & \frac{\Bbbk[\varepsilon]}{\varepsilon^2} \ar[dl]_{\left[\begin{array}{cc}-\lambda & 0\\ 0 & -\lambda\\ 0 & 1\\ 0 & 0\end{array}\right]} & \\
 \frac{\Bbbk[\varepsilon]}{\varepsilon^2}\smallotimesk \Bbbk^2 & & 0 \ar[ul]
}
$
&
$
\xymatrix{
 & \frac{\Bbbk[\varepsilon]}{\varepsilon^2} \ar[dl] & \\
 0  & & \frac{\Bbbk[\varepsilon]}{\varepsilon^2}\smallotimesk \Bbbk^2 \ar[ul]_{\left[\begin{array}{cccc}0 & 1 & -\lambda & 0\\ 0 & 0 & 0 & -\lambda\end{array}\right]}
}
$
& 
$
\xymatrix{
& \Bbbk \ar[dl]_{\left[\begin{array}{c}0 \\ 1\end{array}\right]} & \\
\frac{\Bbbk[\varepsilon]}{\varepsilon^2}\smallotimesk \Bbbk & & 0 \ar[ul]
}
$
& 
$
\xymatrix{
& \Bbbk \ar[dl] & \\
0  & & \frac{\Bbbk[\varepsilon]}{\varepsilon^2}\smallotimesk \Bbbk \ar[ul]_{\left[\begin{array}{cc}1 & 0\end{array}\right]}
}
$
& 
$
\xymatrix{
& 0 \ar[dl] & \\
0 & & 0 \ar[ul]
}
$
&
$
\xymatrix{
& \Bbbk \ar[dl] & \\
0 & & 0 \ar[ul]
}
$
\\
\hline
$\mu_k(\mathcal{M}(T,\mycurve(a,b),\lambda))$
&
$
\xymatrix{
 & \Bbbk^2  \ar@(ur,ul)_{\left[\begin{array}{cc}0 &1\\ 0 & 0\end{array}\right]} \ar[dl]_{\left[\begin{array}{cc}-\lambda &0\\ 0 & 1\end{array}\right]} & \\
 \Bbbk^2 \ar[rr] & & 0 \ar[ul]
}
$
&
$
\xymatrix{
\Bbbk^2  \ar[dr]_{\left[\begin{array}{cc}1 &0\\ 0 & -\lambda\end{array}\right]} & & 0 \ar[ll]\\
 & \Bbbk^2 \ar@(dl,dr)_{\left[\begin{array}{cc}0 &1\\ 0 & 0\end{array}\right]} \ar[ur] & 
}
$  
&
$
\xymatrix{
 & \Bbbk  \ar@(ur,ul)_{0} \ar[dl]_{1} & \\
 \Bbbk \ar[rr] & & 0 \ar[ul]
}
$
&
$
\xymatrix{
\Bbbk  \ar[dr]_{1} & & 0 \ar[ll]\\
 & \Bbbk \ar@(dl,dr)_{0} \ar[ur] & 
}
$ 
& 
Negative simple
&
Simple
\\
\hline
Automorphism $g_{\mycurve}:U_{\mycurve}\to U_{\mycurve}$ & $\lambda\mapsto-\lambda$ & $\lambda\mapsto-\lambda$ & $\myid$ & $\myid$ & $\myid$ & $\myid$
\\
\hline
\end{tabular}}
       \captionof{table}{Proof of Theorem \ref{thm:behavior-of-arc-and-band-reps-under-muts} for the segments $\mycurve(a,b)$}
       \label{table:proof-of-behavior-of-arc-and-band-reps-under-muts-non-simply-laced-2}
    \end{landscape}
    \clearpage
}

As an application of \Cref{thm:behavior-of-arc-and-band-reps-under-muts}, let us show that whenever $\mycurve$ is an arc on $\mathcal{O}$, the representation $\mathcal{M}(T,\mycurve,1)$
provides an explicit, direct computation of the $\tau$-rigid indecomposable pair that corresponds to $\mycurve$ and whose locally free Caldero--Chapoton function is a cluster
variable in the coefficient-free skew-symmetrizable Fomin-Zelevinsky cluster algebra associated to the orbifold $\mathcal{O}$. As in \Cref{section: application cluster algebras}, let $T_0$
be a triangulation of $\mathcal{O}=\surf$. Write $n:=|T_0|$ and fix a numbering $k_1,\ldots,k_n$ of the arcs in $T_0$. Using this same indexing for the columns and rows of the matrix
$B(T_0)$, consider the coefficient-free Fomin--Zelevinsky cluster pattern
\begin{equation}\label{eq:cluster-pattern-for-B(T_0)}
    (\mathbf{x}_{t}^{B(T_0);t_0},B_t^{B(T_0);t_0})_{t\in \mathbb{T}_n}
\end{equation}
with initial seed
$(\mathbf{x}_{t_0},B(T_0))$. Thus,
\begin{equation}\label{eq:cluster-variable-in-cluster-pattern}
\mathbf{x}_{t}^{B(T_0);t_0} = (x_{1;t}^{B(T_0);t_0},\ldots, x_{n;t}^{B(T_0);t_0}).
\end{equation}
For each $t\in \mathbb{T}_n$ there is exactly one path $t_0 \frac{k_1}{\quad} t_1  \frac{k_2}{\quad} \cdots \frac{k_p}{\quad} t_p=t$ from $t_0$ to $t$. By \Cref{lemma: flip and mutation matrix} we have $B_t^{B(T_0);t_0}=B(T_t)$, where $T_t=f_{k_p}\cdots f_{k_1}(T_0)=(\mycurve_{1;t},\ldots,\mycurve_{n;t})$ is the ordered triangulation obtained applying the corresponding sequence of flips. For each $\ell=1,\ldots,n$, consider the decorated $\mathcal{P}(T_0)$-representation \eqref{eq:def-of-dec-rep-M-ell-t-T0,t0}, that is,
$$
\mathcal M_{\ell;t}^{T_0; t_0} \coloneqq \mu_{k_1}\dots\mu_{k_p}(\mathcal E_\ell^-(T_t)).
$$

\begin{corollary}\label{coro:string-reps-give-cluster-variables}
With the notation from the preceding paragraph,
\begin{enumerate}
\item for each $t\in \mathbb{T}_n$ and $\ell=1,\ldots,n$, $\mathcal M_{\ell;t}^{T_0; t_0} \cong \mathcal{M}(T_0,\mycurve_{\ell;t},1)$ as decorated representations of~$\mathcal{P}(T_0)$;
\item $x_{\ell;t}^{B(T_0);t_0}=\mathrm{CC}^{\mathrm{l.f.}}(\mathcal{M}(T_0,\mycurve_{\ell;t},1))$ in the coefficient-free skew-symmetrizable Fomin--Zelevinsky cluster algebra $\mathcal{A}(B(T_0))$.
\end{enumerate}
\end{corollary}

With a slight change of perspective, now let \eqref{eq:cluster-pattern-for-B(T_0)} and \eqref{eq:cluster-variable-in-cluster-pattern} denote instead the pattern associated to $(\mathbf{x}_{t_0},B(T_0))$ in Chekhov--Shapiro's generalized cluster algebra considered in \cite[Section 2]{LM1}. A combination of \Cref{coro:string-reps-give-cluster-variables}-(1) with \cite[Theorem 10.4]{LM1} yields:

\begin{corollary}\label{coro:string-reps-give-generalized-cluster-variables}
    In Chekhov--Shapiro's coefficient-free generalized cluster algebra associated to $\surf$, with all orbifold points being of order $3$, we have
    $$x_{\ell;t}^{B(T_0);t_0}=\mathrm{CC}(\mathcal{M}(T_0,\mycurve_{\ell;t},1))$$
    where $\mathrm{CC}(\mathcal{M})$ is the Caldero--Chapoton function whose coefficients are Euler characteristics of full quiver Grassmannians, instead of Euler characteristics of locally free quiver Grassmannians.
\end{corollary}

\begin{remark} In work \cite{Banaian2023InPreparation} independent from this paper, Esther Banaian and Yadira Valdivieso will prove \Cref{coro:string-reps-give-generalized-cluster-variables} by exploiting Banaian--Kelley's perfect-matching expansions from snake graphs \cite{banaian2020snake}, following essentially the same strategy used by Geiss--Labardini--Schröer in \cite{geiss2022schemes} to show the corresponding result in the absence of orbifold points (and without punctures).
\end{remark}

\section{\texorpdfstring{Example: cluster algebras of type $\widetilde C_n$}{Cluster algebras of type tilde Cn}} \label{section: example}

In this section, we consider a class of special cases to our constructions and results corresponding to cluster algebras of affine type $\widetilde C_n$.

Let $B = (b_{i,j})$ be an $(n+1)\times (n+1)$ skew-symmetrizable matrix such that $|b_{i, i+1}| = |b_{i+1, i}| = 1$ for $i = 2, \dots, n-1$, $|b_{1, 2}| = |b_{n+1, n}| = 1$, $|b_{2,1}| = |b_{n, n+1}| = 2$, and $b_{i, j} = 0$ otherwise. In the language of \cite{geiss2017quivers}, the matrix $B = B(C, \Omega)$ is associated to an orientation $\Omega$ of the affine type $\widetilde C_n$ Cartan matrix 
\[
    C = \begin{bmatrix}
        2 & -1 &  &  &  \\
        -2 & 2 & \ddots &  &  \\
        & -1 & \ddots & -1 &  \\
        &  & \ddots & 2 & -2 \\
        &  &  & -1 & 2
        \end{bmatrix}.
\]
Accordingly we say that any cluster algebra $\mathcal A$ with an exchange matrix $B$ is of type $\widetilde C_n$.

There is an orbifold model for the cluster structure of $\mathcal A(B)$. Let $\mathcal O$ be a disk with $n$ boundary marked points and two orbifold points in the interior; see \Cref{fig: orbifold example} for the case $n = 6$. It is not hard to see that any matrix $B$ as above arises as $B(T)$ from some triangulation $T$ of $\mathcal O$. According to \cite{felikson2012cluster}, the cluster complex of $\mathcal A(B)$ is exactly the arc complex $\mathbf {Arc}(\mathcal O)$.

Our representation theoretic approach starts with the construction of the gentle algebra $\mathcal P(T)$ from a triangulation $T$ of $\mathcal O$. When $T$ induces acyclic $B(T)$ (i.e. of the form $B(C, \Omega)$), the algebra $\mathcal P(T)$ then identifies with the finite-dimensional algebra $H(C, D, \Omega)$ defined by Geiss--Leclerc--Schr\"oer \cite{geiss2017quivers} with the minimal symmetrizer $D = \operatorname{diag}(2, 1, \dots, 1, 2)$ for some orientation $\Omega$. The quiver of $\mathcal P(T)$ is depicted below with one loop at the two ending vertices respectively (corresponding to two pending arcs) and any orientation of the internal edges can be realized.
\[
    \begin{tikzcd}[]
        1 \arrow[loop left, leftarrow, distance=3em, start anchor={[yshift=-1ex]west}, end anchor={[yshift=1ex]west}] \arrow[dash]{r} & 2 \arrow[dash]{r} & \cdots \arrow[dash]{r} & n \arrow[dash]{r} & n + 1 \arrow[loop right, leftarrow, distance=3em, start anchor={[yshift=1ex]east}, end anchor={[yshift=-1ex]east}]
    \end{tikzcd}.
\]

More generally for $C$ a Cartan matrix, $\Omega$ an acyclic orientation and $D$ a symmetrizer, Geiss--Leclerc--Schr\"oer have proposed a Caldero--Chapoton type formula (featuring locally free representations of $H(C, D, \Omega)$) for cluster variables \cite{geiss2018quivers}. They further proved the formula when $C$ is of Dynkin (i.e. finite) type. As discussed above, the case when $C$ is of type $\widetilde C_n$ coincides with the orbifold $\mathcal O$ with a triangulation $T$ inducing acyclic $B(T)$. Therefore as a special case of \Cref{thm: main theorem} (and in particular (3) of \Cref{cor: app in cluster algebra}), we have proved the locally free Caldero--Chapoton formula in type $\widetilde C_n$. We remark that our result not only provides such a formula for an initial seed with an acyclic exchange matrix, but also for any seed since the construction of $\mathcal P(T)$ would not require $B(T)$ to be acyclic.

In \cite[\S11]{LM1}, we gave a type-$\widetilde{C}_2$ example to illustrate how the (ordinary) Caldero--Chapoton functions of two support $\tau$-tilting pairs related by a flip satisfy the exchange relations appearing in the (coefficient-free) generalized cluster algebra associated by Chekhov--Shapiro to an unpunctured digon with two orbifold points of order 3. Here we revisit that example. Given that our current perspective is dual to that of \cite{LM1}, the quivers with potential appearing here will be opposite to those appearing in \cite{LM1}, the representations here will be dual to those in \cite{LM1}, and, as above, we shall be working with injective $\mathbf g$-vectors. The skew-symmetrizable matrices associated to triangulations will be exactly those in \cite{LM1} though, i.e., not their negatives or transposes.

Consider the triangulations $T = T_0,T_1,T_2,T_3$ and $T_4$ from Figure \ref{Fig_ExTriangsC2tilde_GoodExample}. Arcs in these triangulations are labeled by $\{1, 2, 3\}$. The sequence of flips from $T$ to $T_4$ is $(k_1, k_2, k_3, k_4) = (1, 3, 2, 3)$. One can associated $T_i$ to $t_i$ in the subgraph $t_0 \frac{1}{\quad\quad} t_1  \frac{3}{\quad\quad} t_2 \frac{2}{\quad\quad} t_3 \frac{3}{\quad\quad} t_4$ of $\mathbb T_3$. The skew-symmetrizable matrix $B(T)$ is of type $\widetilde{C}_2$. 
\begin{figure}[h]
                \centering
\begin{tikzpicture}[scale = 1.2]
\filldraw[fill=gray!20, thick](0,0) circle (1);
\draw[] (0,-1) -- (0,1);
\draw[] (0,1) .. controls (-1.5, -0.7) and (0, -1) .. (0,1);
\draw[] (0,1) .. controls (1.5, -0.7) and (0, -1) .. (0, 1);
\node[] at (-0.4, 0) {$\star$};
\node[] at (0.4, 0) {$\star$};
\node[] at (0.6, -0.5) {\tiny $1$};
\node[] at (0.1, -0.7) {\tiny $2$};
\node[] at (-0.6, -0.5) {\tiny $3$};
\filldraw(0, -1) circle (1.5pt);
\filldraw(0, 1) circle (1.5pt);
\node[] at (-1, 1) {$T_0$};

\filldraw[fill=gray!20, thick](2.5,0) circle (1);
\filldraw(2.5, -1) circle (1.5pt);
\filldraw(2.5, 1) circle (1.5pt);
\node[] at (2.1, 0) {$\star$};
\node[] at (2.9, 0) {$\star$};
\draw[] (2.5,-1) -- (2.5,1);
\draw[] (2.5,1) .. controls (1, -0.7) and (2.5, -1) .. (2.5,1);
\draw[] (2.5,-1) .. controls (4, 0.7) and (2.5, 1) .. (2.5, -1);
\node[] at (1.9, -0.5) {\tiny $3$};
\node[] at (2.4, -0.3) {\tiny $2$};
\node[] at (3.1, 0.5) {\tiny $1$};
\node[] at (1.5, 1) {$T_1$};

\filldraw[fill=gray!20, thick](5,0) circle (1);
\filldraw(5, -1) circle (1.5pt);
\filldraw(5, 1) circle (1.5pt);
\node[] at (4.6, 0) {$\star$};
\node[] at (5.4, 0) {$\star$};
\draw[] (5,-1) -- (5,1);
\draw[] (5,-1) .. controls (3.5, 0.7) and (5, 1) .. (5, -1);
\draw[] (5,-1) .. controls (6.5, 0.7) and (5, 1) .. (5, -1);
\node[] at (5.6, 0.5) {\tiny $1$};
\node[] at (4.9, 0.6) {\tiny $2$};
\node[] at (4.4, 0.5) {\tiny $3$};
\node[] at (4, 1) {$T_2$};

\filldraw[fill=gray!20, thick](7.5,0) circle (1);
\filldraw(7.5, -1) circle (1.5pt);
\filldraw(7.5, 1) circle (1.5pt);
\node[] at (7.1, 0) {$\star$};
\node[] at (7.9, 0) {$\star$};
\draw[] (7.5,-1) .. controls (6, 0.7) and (7.5, 1) .. (7.5, -1);
\draw[] (7.5,-1) .. controls (9, 0.7) and (7.5, 1) .. (7.5, -1);
\draw[] (7.5, -0.05) ellipse (0.9 and 0.95);
\node[] at (8.1, 0.5) {\tiny $1$};
\node[] at (7.5, 0.8) {\tiny $2$};
\node[] at (6.9, 0.5) {\tiny $3$};
\node[] at (6.5, 1) {$T_3$};

\filldraw[fill=gray!20, thick](10,0) circle (1);
\filldraw(10, -1) circle (1.5pt);
\filldraw(10, 1) circle (1.5pt);
\node[] at (9.6, 0) {$\star$};
\node[] at (10.4, 0) {$\star$};
\draw[] (10,-1) .. controls (11.5, 0.7) and (10, 1) .. (10, -1);
\draw[] (10, -0.05) ellipse (0.9 and 0.95);
\draw[] plot [smooth] coordinates {(10, -1) (10.6, -0.6) (10.78, -0.25) (10.73, 0.4) (10, 0.65) 
(9.5, 0.15) (9.4, -0.15) (9.6, -0.25) (10, 0.5) (10.67, 0.38) (10.7, -0.2) (10.5, -0.58) (10, -1)};
\node[] at (9.9, -0.6) {\tiny $1$};
\node[] at (10, 0.8) {\tiny $2$};
\node[] at (9.4, -0.35) {\tiny $3$};
\node[] at (9, 1) {$T_4$};

\end{tikzpicture}
\caption{Five triangulations of the digon with two orbifold points, related by a sequence of flips.}
\label{Fig_ExTriangsC2tilde_GoodExample}
\end{figure}
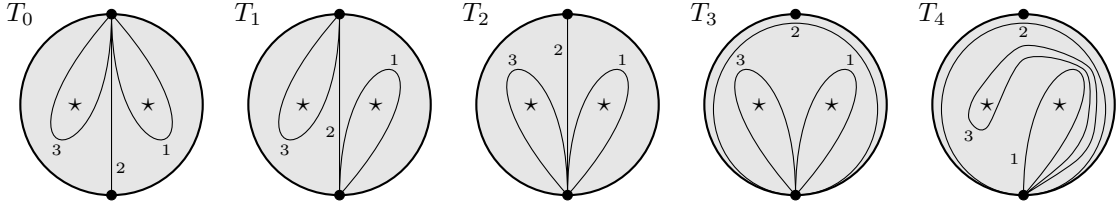

\begin{table}[ht]
\centering
    \begin{tabular}{|c|c|c|}
    \hline
         & $Q(T_j)$ & $B(T_j)$  \\
    \hline
            && \\
    $T_4$ & { $\xymatrix{&  3 \ar[dr] \ar@(lu,ld)   & & 1 \ar[ll]  \ar@(rd,ru)  & \\ & & 2 \ar[ur] & & }$} & { $\left[\begin{array}{ccc} 0 & 1  & -2 \\ -2 & 0 & 2 \\ 2 & -1 & 0 \end{array}\right]$ } \\
        && \\
    $T_3$ & { $\xymatrix{& 3 \ar@(lu,ld) \ar[rr] & & 1 \ar@(rd,ru) \ar[dl]  & \\ & &  2 \ar[ul] & & }$} & { $\left[\begin{array}{ccc} 0 & -1 & 2 \\ 2 & 0 & -2 \\ -2 & 1 & 0 \end{array}\right]$ }  \\
        && \\
    $T_2$    & { $\xymatrix{ & 3 \ar@(lu,ld) \ar[dr] & & 1 \ar@(rd,ru)  & \\ & & 2 \ar[ur] & & }$}  & {$\left[\begin{array}{ccc} 0 & 1 & 0\\ -2 & 0 & 2 \\ 0 & -1 & 0 \end{array}\right]$ }  \\
        && \\
    $T_1$     & {  $\xymatrix{ & 3  \ar@(lu,ld)  & & 1  \ar@(rd,ru)  & \\ & & 2 \ar[ul] \ar[ur]   & & }$}   & { $\left[\begin{array}{ccc} 0 & 1 & 0 \\ -2 & 0 & -2 \\ 0 & 1 & 0 \end{array}\right]$}   \\
        && \\
    $T_0$ & { $\xymatrix{ & 3  \ar@(lu,ld)  & & 1 \ar@(rd,ru) \ar[dl] &\\ & & 2 \ar[ul] & & }$} & { $\left[\begin{array}{ccc} 0 & -1 & 0\\ 2 & 0 & -2 \\ 0 & 1 & 0\end{array}\right]$}  \\
    \hline     
    \end{tabular}
    \caption{Quivers and matrices associated to $T_i$}
    \label{table:quivers-and-matrices-four-triangulations-type-tildeC}
\end{table}

In \Cref{table:quivers-and-matrices-four-triangulations-type-tildeC} we can see the quivers and skew-symmetrizable matrices of the triangulations.

Denote the decorated representations $\mathcal E_i^-(T_4) \coloneqq (0, H_i)$ for $i = 1, 2, 3$. Consider the decorated representations 
\begin{align*}
    \mathcal{M}_{1;4}\oplus\mathcal{M}_{2;4}\oplus\mathcal{M}_{3;4} &:=\mathcal E_1^-(T_4)\oplus\mathcal E_2^-(T_4)\oplus \mathcal E_3^-(T_4)  &  \text{in} \quad \decrep{\mathcal{P}(T_4)} \\
    \mathcal{M}_{1;3}\oplus\mathcal{M}_{2;3}\oplus\mathcal{M}_{3;3} &:=\mu_{3}(\mathcal E_1^-(T_4)\oplus\mathcal E_2^-(T_4)\oplus \mathcal E_3^-(T_4)) &  \text{in} \quad \decrep{\mathcal{P}(T_3)} \\
    \mathcal{M}_{1;2}\oplus\mathcal{M}_{2;2}\oplus\mathcal{M}_{3;2} &:=\mu_{2}\mu_{3}(\mathcal E_1^-(T_4)\oplus\mathcal E_2^-(T_4)\oplus \mathcal E_3^-(T_4)) &  \text{in} \quad \decrep{\mathcal{P}(T_2)}\\
    \mathcal{M}_{1;1}\oplus\mathcal{M}_{2;1}\oplus\mathcal{M}_{3;1} &:=\mu_{3}\mu_{2}\mu_{3}(\mathcal E_1^-(T_4)\oplus\mathcal E_2^-(T_4)\oplus \mathcal E_3^-(T_4)) &  \text{in} \quad \decrep{\mathcal{P}(T_1)}\\
    \mathcal{M}_{1;0}\oplus\mathcal{M}_{2;0}\oplus\mathcal{M}_{3;0} &:=\mu_1\mu_{3}\mu_{2}\mu_{3}(\mathcal E_1^-(T_4)\oplus\mathcal E_2^-(T_4)\oplus \mathcal E_3^-(T_4)) &  \text{in} \quad \decrep{\mathcal{P}(T_0)}.
\end{align*}
These are explicitly computed in \Cref{table: tau rigid pairs} below.

\begin{table}[ht]
    \centering
    \begin{tabular}{|c|c|c|c|}
    \hline
         &  $\mathcal{M}_{1;j}$ & $\mathcal{M}_{2;j}$ & $\mathcal{M}_{3;j}$ \\
    \hline 
         $j = 4$ & $\mathcal E_1^-(T_4)$ & $\mathcal E_2^-(T_4)$ & $\mathcal E_3^-(T_4)$ \\
    
         $j = 3$ & $\mathcal E_1^-(T_3)$ & $\mathcal E_2^-(T_3)$ & \begin{tikzcd}[row sep = scriptsize, column sep = scriptsize, ampersand replacement=\&]
         \mathbb C^2 \ar[loop above, "{\begin{bsmallmatrix} 0 & 1\\ 0 & 0 \end{bsmallmatrix}}"]\ar[rr] \& \& 0\ar[dl]\ar[loop above]\\
         \& 0 \ar[ul] \&
         \end{tikzcd} \\
    
         $j = 2$ & $\mathcal E_1^-(T_2)$ & \begin{tikzcd}[row sep = scriptsize, column sep = scriptsize, ampersand replacement=\&]
         0 \ar[loop above] \ar[dr]\& \& 0  \ar[loop above]\\
         \& \mathbb C \ar[ur] \&
         \end{tikzcd} & \begin{tikzcd}[row sep = scriptsize, column sep = scriptsize, ampersand replacement=\&]
         \mathbb C^2 \ar[loop above, "{\begin{bsmallmatrix} 0 & 1 \\ 0 & 0 \end{bsmallmatrix}}"] \ar[dr, "{\begin{bsmallmatrix} 1 & 0 \\ 0 & 1 \end{bsmallmatrix}}", swap]\& \& 0  \ar[loop above]\\
         \& \mathbb C^2 \ar[ur] \&
         \end{tikzcd} \\
    
         $j = 1$ & $\mathcal E_1^-(T_1)$ & \begin{tikzcd}[row sep = scriptsize, column sep = scriptsize, ampersand replacement=\&]
         \mathbb C^2 \ar[loop above, "{\begin{bsmallmatrix} 0 & 1\\ 0 & 0 \end{bsmallmatrix}}"] \& \& 0  \ar[loop above]\\
         \& \mathbb C \ar[ul, "{\begin{bsmallmatrix} 0 \\ 1 \end{bsmallmatrix}}"] \ar[ur] \&
         \end{tikzcd} & \begin{tikzcd}[row sep = scriptsize, column sep = scriptsize, ampersand replacement=\&]
         \mathbb C^2 \ar[loop above, "{\begin{bsmallmatrix} 0 & 1 \\ 0 & 0 \end{bsmallmatrix}}"]   \& \& 0  \ar[loop above]\\
         \& \mathbb C^2 \ar[ul, "{\begin{bsmallmatrix} 1 & 0 \\ 0 & 1 \end{bsmallmatrix}}"] \ar[ur] \&
         \end{tikzcd}\\
    
         $j = 0$ & \begin{tikzcd}[row sep = scriptsize, column sep = scriptsize, ampersand replacement=\&]
         0 \ar[loop above] \& \& \mathbb C^2 \ar[loop above, "{\begin{bsmallmatrix} 0 & 1\\ 0 & 0 \end{bsmallmatrix}}"] \ar[dl]\\
         \& 0  \ar[ul] \&
         \end{tikzcd} & \begin{tikzcd}[row sep = scriptsize, column sep = scriptsize, ampersand replacement=\&]
         \mathbb C^2 \ar[loop above, "{\begin{bsmallmatrix} 0 & 1\\ 0 & 0 \end{bsmallmatrix}}"] \& \& \mathbb C^2 \ar[loop above, "{\begin{bsmallmatrix} 0 & 1\\ 0 & 0 \end{bsmallmatrix}}"] \ar[dl, "{\begin{bsmallmatrix} 1 & 0 \end{bsmallmatrix}}"]\\
         \& \mathbb C \ar[ul, "{\begin{bsmallmatrix} 0 \\ 1 \end{bsmallmatrix}}"] \&
         \end{tikzcd} & \begin{tikzcd}[row sep = scriptsize, column sep = scriptsize, ampersand replacement=\&]
         \mathbb C^2 \ar[loop above, "{\begin{bsmallmatrix} 0 & 1\\ 0 & 0 \end{bsmallmatrix}}"] \& \& \mathbb C^4 \ar[dl, "{\begin{bsmallmatrix} 0&1&0&0\\ 0&0&1&0 \end{bsmallmatrix}}"]\ar[loop above, "{\begin{bsmallmatrix}0&0&0&0\\ 1&0&0&0\\ 0&0&0&1\\ 0&0&0&0 \end{bsmallmatrix}}"]\\
         \& \mathbb C^2 \ar[ul, "{\begin{bsmallmatrix} 1 & 0 \\ 0 & 1 \end{bsmallmatrix}}"] \&
         \end{tikzcd}\\
         & & & \\
    \hline           
    \end{tabular}
    
    \caption{Indecomposable $E$-rigid decorated representations.}
    \label{table: tau rigid pairs}
\end{table}

Consider the decorated representation $\mathcal N := (\begin{tikzcd}[column sep = small, ampersand replacement=\&]
\mathbb C^2 \ar[loop left, "{\begin{bsmallmatrix} 0 & 1 \\ 0 & 0\end{bsmallmatrix}}"]  \& 0 \ar[l] \& 0 \ar[l]\ar[loop right]
\end{tikzcd},0)$. In view of \cite[\S11]{LM1}, the decorated representations $\mathcal M_{1;0}\oplus\mathcal M_{2;0} \oplus \mathcal M_{3;0} $ and $\mathcal M_{1;0}\oplus \mathcal M_{2;0} \oplus \mathcal N$ of $\mathcal{P}(T_0)$ correspond to the triangulations $T_4$ and $T_3$ respectively. In the notation system of \Cref{section: application cluster algebras}, they are $\mathcal M_{t_4}^{B(T_0);t_0}$ and $\mathcal M_{t_3}^{B(T_0); t_0}$ respectively.
Since
$$
0\rightarrow \mathcal{M}_{3;0}\rightarrow I_3(T_0) \rightarrow 0
 \quad \text{and} \quad
 0\rightarrow \mathcal{N}\rightarrow I_3(T_0)\rightarrow I_2(T_0)^2 
$$
are the minimal injective resolutions of $\mathcal{M}_{3;0}$ and $\mathcal{N}$,
their (injective) $\mathbf g$-vectors are
\[
    \mathbf g^{\mathcal P(T_0)}(\mathcal M_{3; 0}) = (0, 0, -1) \quad \text{and} \quad
    \mathbf g^{\mathcal P(T_0)}(\mathcal N) = (0, 2, -1).
\]
Furthermore, the locally free $F$-polynomials of $\mathcal{M}_{3;0}$ and $\mathcal{N}$ are
$$
    F({\mathcal{M}_{3;0}})= 1 +y_3+2y_2y_3+2y_1y_2y_3+y_2^2y_3+2y_1y_2^2y_3+y_1^2y_2^2y_3 \quad \text{and} \quad
    F({\mathcal{N}})= 1+y_3.
$$
Hence, their locally free Caldero--Chapoton functions are
\begin{align*}
    \mathrm{CC}^\mathrm{l.f.}(\mathcal{M}_{3;0}) &=
x_3^{-1}+x_2^{-2}x_3^{-1}+2x_1^{-1}x_2^{-2}+2x_1^{-1}+x_1^{-2}x_2^{-2}x_3+2x_1^{-2}x_3+x_1^{-2}x_2^{2}x_3,
\\
\mathrm{CC}^\mathrm{l.f.}(\mathcal{N})&=x_2^2x_3^{-1}+x_3^{-1}.
\end{align*}

As for the decorated representations of $\mathcal{P}(T_0)$ shared by both $T_3$ and $T_4$, we have
\[
    \mathbf{g}^{\mathcal{P}(T_0)}(\mathcal{M}_{1;0}) = (-1, 0, 0) \quad \text{and} \quad \mathbf g^{\mathcal P(T_0)}(\mathcal M_{2;0}) = (0, 1, -1);
\]
\[
    F(\mathcal M_{1;0}) = 1 + y_1 \quad \text{and} \quad F(\mathcal M_{2;0}) = 1 + y_3 + y_2y_3 + y_1y_2y_3.
\]
Hence, their locally free Caldero--Chapoton functions are
\begin{align*}
    \mathrm{CC}^\mathrm{l.f.}(\mathcal{M}_{1;0}) &= x_1^{-1}+x_1^{-1}x_2^2,
\\
    \mathrm{CC}^\mathrm{l.f.}(\mathcal{M}_{2;0})&=
x_2x_3^{-1}+x_2^{-1}x_3^{-1}+x_1^{-1}x_2^{-1}+x_1^{-1}x_2.
\end{align*}

Direct computation shows that 
$$
\mathrm{CC}^\mathrm{l.f.}(\mathcal{M}_{3;0})\mathrm{CC}^\mathrm{l.f.}(\mathcal{N})=
\mathrm{CC}^\mathrm{l.f.}(\mathcal{M}_{1;0})^2+\mathrm{CC}^\mathrm{l.f.}(\mathcal{M}_{2;0})^2,
$$
which is precisely the equality predicted by the cluster exchange relation when flipping a pending arc (\Cref{fig: exchange relation pending}). The reader can directly verify that the representations $\mathcal{M}_{1;0}$, $\mathcal{M}_{2;0}$, $\mathcal{M}_{3;0}$ and $\mathcal{N}$ are precisely dual to those in \cite[\S11]{LM1}.

\bibliographystyle{amsalpha}

\newcommand{\etalchar}[1]{$^{#1}$}
\providecommand{\bysame}{\leavevmode\hbox to3em{\hrulefill}\thinspace}
\providecommand{\MR}{\relax\ifhmode\unskip\space\fi MR }
\providecommand{\MRhref}[2]{%
  \href{http://www.ams.org/mathscinet-getitem?mr=#1}{#2}
}
\providecommand{\href}[2]{#2}

\end{document}